\documentclass{amsart}
\usepackage{graphicx}
\usepackage{amscd}
\usepackage{amsthm}
\usepackage{amssymb}
\usepackage{tikz-cd}
\usetikzlibrary{arrows,matrix}
\usepackage[mathscr]{eucal}
\usepackage[all]{xy}
\usepackage{graphicx,textcomp}
\usepackage{appendix}
\usepackage{hyperref}

\usepackage[english]{babel}

\newtheorem{theorem}{Theorem}[section]
\newtheorem{corollary}{Corollary}[theorem]
\newtheorem{lemma}[theorem]{Lemma}
\newtheorem{proposition}{Proposition}[section]

\theoremstyle{definition}
\newtheorem{definition}[theorem]{Definition}

\theoremstyle{remark}
\newtheorem{remark}[theorem]{Remark}

\numberwithin{equation}{section}

\newcommand{\bx}{{\bf x}}

\newcommand{\by}{\mathbf{y}}
\newcommand{\bA}{\mathbb{A}}

\newcommand{\R}{\mathbb{R}}

\newcommand{\Q}{\mathbb{Q}}

\newcommand{\C}{\mathbb{C}}

\renewcommand{\H}{\mathbb{H}}

\newcommand{\W}{\mathbb{W}}

\newcommand{\Oo}{\mathcal{O}}

\newcommand{\cL}{\mathcal{L}}

\newcommand{\calA}{\mathcal{A}}
\newcommand{\p}{\mathfrak{p}}

	\newcommand{\Orth}{\operatorname{O}}
	\newcommand{\Uni}{\operatorname{U}}

	\newcommand{\Gr}{\operatorname{Gr}}

\newcommand{\Sym}{\mathrm{Sym}}
\newcommand{\Pol}{\mathrm{Pol}}
\newcommand{\Herm}{\mathrm{Herm}}
\newcommand{\Hom}{\mathrm{Hom}}
\newcommand{\rU}{\mathrm{U}}
\newcommand{\rO}{\mathrm{O}}
\newcommand{\Sp}{\mathrm{Sp}}
\newcommand{\Mp}{\mathrm{Mp}}
\newcommand{\GL}{\mathrm{GL}}
\newcommand{\Span}{\mathrm{span}}
\newcommand{\Det}{\mathrm{det}}

\begin{document}

	\title[Theta series and cycles]{Theta series and generalized special cycles on Hermitian locally symmetric manifolds}

		\author[Yousheng Shi] {Yousheng Shi*}

		\thanks{* Partially supported by NSF grant DMS-1518657}
		
	    \date{\today}
				\address{Department of Mathematics, University of Wisconsin, Madison, WI
			53706, USA}
		\email{shi58@wisc.edu}
		
		\begin{abstract}
		 We study generalized special cycles on Hermitian locally symmetric spaces $\Gamma \backslash D$ associated to the groups $G=\mathrm{U}(p,q)$, $\mathrm{Sp}(2n,\R) $ and $\mathrm{O}^*(2n) $. These cycles are algebraic and covered by  symmetric spaces associated to subgroups of $G$ which are of the same type. We show that Poincar\'e duals of these generalized special cycles can be viewed as Fourier coefficients of a theta series. This gives new cases of theta lifts from the cohomology of Hermitian locally symmetric manifolds associated to $G $ to vector-valued automorphic functions on the groups $G'=\mathrm{U}(m,m)$, $\mathrm{O}(m,m)$ or $\mathrm{Sp}(m,m)$ which forms a reductive dual pair with $G$.

		\end{abstract}

		\maketitle

\setcounter{tocdepth}{1}
\tableofcontents

\section{Introduction}
\subsection{Generalized special cycles}

There are four classes of irreducible reductive dual pairs over $\R$ of type I in the sense of Howe \cite{Howe1} (c.f. \cite{Adams}):
\begin{enumerate}
    \item $(\Orth(p,q),\Sp(2n,\R))$,\
    \item $(\Uni(p,q),\Uni(r,s))$, \
    \item $(\Sp(p,q),\Orth^*(2n))$, \
    \item $(\Orth(m,\C),\Sp(2n,\C))$.
\end{enumerate}
Each group $G$ belonging to any of the seven families of groups in the above table is the group that preserves a non-degenerate Hermitian or skew-Hermitian form $(,)$ over a real, complex or quaternionic vector space $V$. Let $\Gamma$ be a  torsion free congruence subgroup of $G$ such that $\Gamma\backslash G$ is compact. In this introduction we  assume that $G = \mathbb{G}(\mathbb{R})$ is the set of real points of an algebraic group $\mathbb{G}$ and $\Gamma$ is a congruence subgroup of $\mathbb{G}(\mathbb{Z})$. In general we will need to assume  $\mathbb{G}(\mathbb{R}) = G \times G_c$ where $G_c$ is compact, but we omit this technicality in the introduction. Furthermore we assume we have chosen a lattice $\mathcal{L}$ in $V$ which is invariant under $\Gamma$.  In each of the cases that we are interested in, the symmetric space $D=G/K$ associated to $G$, where $K$ is a fixed maximal compact subgroup, has a realization as an open subspace of a Lagrangian Grassmannian associated to $V$. In what follows let $M = \Gamma\backslash D$. Once we have chosen an orientation of $D$, after passing to a (possibility deeper) congruence subgroup of $\Gamma$ we may assume $M$ is a compact oriented manifold. 

We define and study cycles which are called ``generalized special cycles'', to be denoted  $C_{\bx,z'}$ (see below for the explanation of the notation),  in the locally symmetric spaces $M$. In this paper, we restrict our attention to the cases
$G=\Uni(p,q)$, $\Sp(2n,\R) $ and $\Orth^*(2n)$ which we denote by case A, B and C respectively throughout the paper. The members of the above three families of groups are all the groups that show up in real reductive dual pairs of type I  whose symmetric spaces are of Hermitian type with the exception of $\rO(p,2)$. In these cases $M$ is a compact K\"ahler manifold which is in fact a connected complex algebraic variety (\cite{BailyBorel}), and the cycles $C_{\bx,z'}$ are algebraic cycles (see Theorem \ref{specialcyclesaresubvarieties}). 

We now briefly introduce the definition of $C_{\bx,z'}$ (see Section \ref{generalizedspecial} for more details). Roughly speaking these cycles come from embeddings of smaller groups of the same type as $G$ into $G$.
In what follows we will let $V = \C^{p+q}$ for case A, $V = \R^{2n}$ for case B, and $V=\H^n$ for case C, where $\H$, the Hamilton quaternions, acts by right multiplication. We assume that $(,)$ is a Hermitian form, skew symmetric form and a skew-Hermitian form on $V$ respectively and $G$ is the linear isometry group of $(,)$.  
Let $\bx = (x_1, x_2,\cdots,x_m)  \in V^m$. We assume that the vectors $x_1,x_2,\cdots,x_m$ are linearly independent and the restriction of the form $(,)$ on $U=\Span\{\bx\}$ is non-degenerate. In particular, in case B this implies that $m=2r$ for some positive integer $r$. In case A, let $(r,s)$ be the signature of $(,)|_U$.
We then have the orthogonal splitting
\begin{equation}\label{splittingofV}
V=U\oplus U^\bot.
\end{equation}
For any non-degenerate subspace $U \subseteq V$ define 
\[G(U)=\{g\in G\mid gv=v, \forall v\in U^\bot\},\]
and $D(U)$ to be the symmetric space of $G(U)$. We would like to embed $D(U^\bot)$ into $D$. However such an embedding in general can only be defined after the choice of a point $z'\in D(U)$:
\[\rho_{U,z'}:D(U^\bot)\rightarrow D.\]
The image of $\rho_{U,z'}$ is a complex analytic subvariety of $D$ which we denote by $D_{U,z'}$ or $D_{\bx,z'}$.
We then pass to the locally symmetric space level and still denote  by $\rho_{U,z'}$ (by abuse of notation) the induced map
\[\rho_{U,z'}:\Gamma_U\backslash D(U^\bot) \rightarrow M=\Gamma \backslash D,\]
where $\Gamma_U=G(U^\bot)\cap \Gamma$. The image $C_{U,z'}$ or $C_{\bx,z'}$ of $\rho_{U,z'}$ is then an algebraic subvariety of $M$ which we call a generalized special cycle. 
In general, $\rho_{U,z'} $ is not an embedding and $C_{\bx,z'}$ is singular. However by passing to a deeper congruence subgroup of $\Gamma$ we can resolve the singularity (see Lemma \ref{resolutionofsingularity}). 
We can think of $C_{\bx,z'}$ as an element in the Chow group $\mathrm{Ch}_*(M)$ or an element in the Singular homology group $H_*(M)$. 
\begin{remark}\label{homology classindependentofchoice}
It is an important fact that the homology class of $C_{\bx,z'}$ {\bf does not} depends on the choice $z'$ (Proposition \ref{homologyindependent}). Hence when only the homology class is considered, we will often write $[C_\bx]$ instead of $[C_{\bx,z'}]$. 
\end{remark}

\begin{remark}
When $s$ or $r$ is equal to $0$ in cases A, the cycle $C_{\bx,z'}$ is called a special cycle by Kudla and Millson (see \cite{KMCompo}, \cite{KMI}, \cite{KMII}, \cite{KM90}). In these cases $G(U)$ is compact and its symmetric space is a single point. In other words, the choice of $z'$ in the definition of $C_{\bx,z'}$ is not necessary.
\end{remark}

One of the main goals of the paper is to construct Poincar\'e dual of $C_{\bx,z'}$ in terms of differential forms by studying differential forms on $D$ with values in the Weil representation.  
Let $\mathcal{S}(V^m)$ be the Schwartz functions on $V^m$. Then there is a reductive dual pair $(G,G')$ such that $G\times \tilde{G}'$ acts on $\mathcal{S}(V^m)$ smoothly and by a unitary representation $\omega$, where $\tilde{G}'$ is a double cover of $G'$.
In cases of our interests, we have the following data
\begin{enumerate}
    \item Case A: $G=\rU(p,q)$, $G'=\rU(m,m)$ with $m=r+s$, $0\leq r \leq p$, $0\leq s \leq q $, \
    \item Case B: $G=\Sp(2n,\R)$, $G'=\rO(m,m)$ with $m=2r$, $0\leq r \leq n$,\
    \item Case C: $G=\rO^*(2n)$, $G'=\Sp(m,m)$, $0\leq m \leq n$.
\end{enumerate}
The action of $G$ via $\omega$ is just the induced action on functions
\[(\omega(g)f)(\bx)=f(g^{-1}\bx).\]
for $f \in \mathcal{S}(V^m)$, where $\bx\in V^m$ is viewed as a $(p+q)$ by $m$ matrix.

Let $\Omega^\bullet(D,\mathcal{S}(V^m))^G$ be the set of $G$-invariant differential forms on $D$ with values in $\mathcal{S}(V^m)$, it is a chain complex graded by the Hodge bi-degree. In section \ref{Andersoncocyle}, we construct a differential form $\varphi\in \Omega^{(d',d')}(D,\mathcal{S}(V^m))^G$ where $d'$ is the complex codimension of the generalized special cycle $C_{\bx,z'}$ where $\bx\in V^m$ is non-degenerate. In case A, the construction of $\varphi$ depends on a pair of integers $(r,s)$ such that $r+s=m$ and we call the pair of integers $(r,s)$ in the theorem the signature of $\varphi$.
The following is the first main result of this paper, see Theorem \ref{closedness} and its corollary.
\begin{theorem}\label{thmA0}
The form $\varphi$ is closed.
\end{theorem}
\begin{remark}
In case A when $s=0$, the cocycle $\varphi$ is actually the special form defined by Kudla and Millson in \cite{KM90}.
\end{remark}

Recall that $\mathcal{L}$ is a lattice in $V$ fixed by $\Gamma$. By \cite{Weil}, we can choose an arithmetic subgroup $\Gamma'\subset G'$ such that the theta distribution $\theta_{\mathcal{L}} $ 
\[\theta_{\cL}(\psi)=\sum_{\bx\in \cL^m} \psi(\bx),\  \psi\in \Omega^\bullet(D,\mathcal{S}(V^m))^G\]
is $ \Gamma\times \tilde{\Gamma}'$-invariant where $\tilde{\Gamma}'$ is the inverse image of $\Gamma'$ in $\tilde{G}'$. Hence for the $\varphi$ defined in Theorem \ref{thmA0} we can define a function
\begin{equation}
    \theta_{\cL,{\varphi}}(z,g')=\sum_{\bx \in \cL^m} \omega(g')\varphi(z,\bx)\in \Omega^\bullet(M)\otimes \mathcal{A}(\Gamma'\backslash G'),
\end{equation}
where $\mathcal{A}(\Gamma'\backslash G')$ is the space of functions on $ \Gamma'\backslash G'$. We also define 
\begin{equation}
    \theta_{\cL,\beta,{\varphi}}(z,g')=\sum_{\bx\in \cL^m,(\bx,\bx)=\beta} \omega(g'){\varphi}(z,\bx)
\end{equation}
for a matrix $\beta$ which is Hermitian in $M_m(\C)$in case A, skew symmetric in $M_m(\R)$ in case B and  skew Hermitian in $M_m(\H)$ in case C. We have the following Fourier expansion of $\theta_{\cL,\varphi_\infty}$:
\[\theta_{\cL,{\varphi}}(z,g')=\sum_{\beta}  \theta_{\cL,\beta,{\varphi}}(z,g'),\]
where $\beta$ runs over all possible inner product matrices $(\bx,\bx)$. We call $\theta_{\cL,\beta,{\varphi}} $ the $\beta$-th Fourier term of $\theta_{\cL,{\varphi}} $ as each $\theta_{\cL,\beta,{\varphi}} $ is a character function under the action of $\Gamma'\cap N'$, where $N'$ is the unipotent radical of the Siegel parabolic (see Section \ref{dualpair}).

From now on we assume that $\beta$ is non-degenerate. The set 
	    $\{\bx\in \mathcal{L}^m|(\bx,\bx)=\beta\}$
consists of finitely many $\Gamma$-orbits. We choose $\Gamma$-orbit representatives $\{\bx_1,\ldots,\bx_o\}$ and define 
	    \[U_i=\Span \{\bx_i\}, 1\leq i \leq o.\]
For each $1\leq i \leq o$ choose a base point $z_i \in D(U_i)$. Let $C_{\bx_i,z_i}$ be the generalized special cycle. Then all these cycles have the same complex codimension $d'$. Let
\[{\bf z}=\{z_1,z_2,\ldots, z_o\}.\]
Then define
	    \[C_{\beta,{\bf z}}=\sum_{i=1}^o C_{\bx_i, z_i}.\]
$C_{\beta,{\bf z}}$ is a cycle in the Chow group of $\Gamma \backslash D$. By Remark \ref{homology classindependentofchoice}, the homology class $[C_{\beta,{\bf z}}]$ is independent of the choice of ${\bf z}$, so we simply denote by $[C_\beta]$ its homology class.

We want to relate the form $\varphi$ in Theorem \ref{thmA0} to generalized special cycles. We will prove the following theorem in Theorem \ref{generatingseries}. 
\begin{theorem}\label{PDthm}
Let $\varphi$ be as in Theorem \ref{thmA0}. For  $\bx\in \cL^m$ such that $U=\Span\{x\}$ is non-degenerate and in case A of signature $(r,s)$ which matches the signature of $\varphi$, 
we have 
\[\sum_{\by \in \Gamma \cdot \bx}[\varphi(z,g',\by)]=\kappa(g',\beta) \mathrm{PD}([C_{\bx}]),\]
where $[\varphi]$ is the cohomology class of $\varphi$ in $ H^*(M)$
, $\mathrm{PD}([C_\bx])\in H^*(M)$ is the Poincar\'e dual of $[C_\bx]$, and $\kappa$ is a function that is analytic in $G'$.
If  $\beta=(\bx,\bx)$, then we have  
\[[\theta_{\cL,\beta,\varphi}(z,g')]=\kappa(g',\beta) \mathrm{PD}([C_\beta]).\]

\end{theorem}

In case A when $s=0$, Theorem \ref{PDthm} and \ref{thmB} are proved in \cite{KM90}, where the exact value of $\kappa(g',\beta)$ is calculated and shown to be never zero. However in the more general case of this paper the exact value of $\kappa(g',\beta)$ is hard to obtain. Instead we calculate certain asymptotic value of $\kappa(g',\beta)$ in Section \ref{methodofLaplacesection} which implies the following theorem (see Theorem \ref{genericnonzero}). 
\begin{theorem}\label{thmB}
Let $\beta$ satisfies the same assumption as in Theorem \ref{PDthm}.
For a generic $g'\in G'$, the function $\kappa(g',\beta) $ in Theorem \ref{PDthm}
is not zero. 
\end{theorem}

In the appendix, we will show that the canonical special class $\varphi$ transforms under an irreducible representation of a maximal compact group $\tilde{K}'\subset \tilde{G}'$. Moreover $\theta_{\cL,\varphi}$ can be viewed as a matrix coefficient of an automorphic vector bundle on $\tilde{\Gamma}'\backslash \tilde{G}'/ \tilde{K}'$.

\subsection{Related works}
The modularity of the generating series of intersection numbers of special cycles was first studied in \cite{HZ} in case of Hillbert modular surfaces. Later in a series of work  (\cite{KMCompo}, \cite{KMI}, \cite{KMII}, \cite{KM90}), Kudla and Millson proved the modularity of generating series of special cycles for higher rank locally symmetric spaces associated to $G=\Orth(p,q)$ (resp. $\Uni(p,q)$ and $\Sp(p,q)$). To be more specific they constructed via Weil representation differential forms that are Poincar\'e duals to
$C_\bx$ when $(,)|_{\Span\{\bx\}}$ is positive definite. By applying the theta distribution to these forms one get an automorphic form in the dual group $\Sp(2r,\R)$ (resp. $\Uni(r,r)$, $\Orth^*(2r)$) of $G$. Moreover they prove in \cite{KM90} that these differential forms are holomorphic with respect to the dual group $G'=\Sp(2r,\R)$ (resp. $\Uni(r,r)$) on the cohomology level, thus give rise to holomorphic modular forms after applying theta distribution.

The theory of Kudla and Millson have some generalizations and applications. We just briefly mention some here.
\begin{enumerate}
    \item When $G=\Orth(p,q)$ and $\Gamma$ is not co-compact, the boundary behavior (after the compactification of $\Gamma \backslash D$) of  the special forms constructed by Kudla and Millson has been studied in \cite{FM1}, \cite{FM2}, \cite{FM3} and \cite{FM4}.\ 
    \item Using the results of Kudla and Millson, together with the classification of unitary representations with nonzero cohomology of Vogan-Zuckerman \cite{VoganZuckerman} and the endoscopic classification of automorphic representations of $G'$ (c.f.\cite{mok2015endoscopic}), \cite{BMM1} and \cite{BMM2} are able to prove certain cases of Hodge Conjecture on arithmetic hyperbolic spaces and arithmetic quotients of complex balls.\ 
    \item In some cases one can lift the modularity theorem by Kudla and Millson from cohomology groups to Chow groups (c.f. \cite{borcherds1999}, \cite{zhang2009modularity} and \cite{bruinier2015kudla}) or even arithmetic Chow groups (c.f. \cite{bruinier2017modularity} and \cite{howard2017arithmetic}). \ 
    \item Garcia \cite{garcia} views the special forms of Kudla and Millson as characteristic classes of  super connections and generalizes the construction to certain period domains. 
\end{enumerate}

This paper is an attempt to generalize the work of Kudla and Millson. What is new in this paper is the definition of generalized special cycles and the discovery of a new class of special forms in $\Omega(D,\mathcal{S}(V^m))^G$ which turn out to be Poincar\'e duals of the generalized special cycles and can be used as kernels of geometric theta lifts.
When $G=\Sp(2n,\R)$ or $\rO^*(2n)$ or when $G=\rU(p,q)$ but $(,)|_{\Span\{\bx\}}$ is not positive definite, there is no corresponding special cycles in the sense of Kudla and Millson. So in order to have a similar theory, one has to define generalized special cycles.

In a sequel \cite{MS}, Millson and the author will extend the results of this paper to the groups $G=\Orth(p,q)$, $\Sp(p,q) $, $\Orth(m,\C) $ or $\Sp(2n,\C)$.

\subsection{Sketch of the proof of the main theorems}	 
The form $\varphi$ in Theorem \ref{thmA0} is ultimately constructed from $\varphi^+$ which is a holomorphic differential form in $\Omega^{(d',0)}(D,\mathcal{P})^G $ discovered by \cite{Anderson}, where $ \mathcal{P}$ is certain Fock model of the Weil representation of a compact dual pair. Using the fact that any holomorphic differential form on a K\"ahler manifold is closed together with a result of \cite{Anderson}, we can prove that $\varphi^+$ is closed. This implies that $\varphi$ is closed as well.

In order to prove Theorem \ref{PDthm}, we construct a fiber bundle in Section \ref{fibrationpi}:
\[\pi:\Gamma_U \backslash D\rightarrow \Gamma_U \backslash D_{U,z'},\]
whose fibers are (topologically) Euclidean spaces. We show that $\varphi(z,g',\bx)$ is a constant multiple of the Thom form for the above fibration. To be more precise,
\begin{equation}\label{keyidentity}
   \int_{\Gamma_U \backslash D} \eta\wedge\varphi(z,g'.\bx)=\kappa(g',\beta) \int_{\Gamma_U \backslash D_{U,z'}} \eta,
\end{equation}
where 
\begin{equation}\label{kappaintegralintro}
    \kappa(g',\beta)=\int_{F D_{U,z'}} \varphi(z,g',\bx),
\end{equation}
and $F D_{U,z'}$ is any fiber of the fibration $\Gamma_U \backslash D\rightarrow \Gamma_U \backslash D_{U,z'}$. Notice that the integration in  \eqref{kappaintegralintro} only depends on $\beta=(\bx,\bx)$. The key to proving \eqref{keyidentity} is to show that the norm of $ \varphi(z,g'.\bx)$ is fast decreasing on the fiber $F D_{U,z'}$, which will be done in Section \ref{rapiddecreasechapter}.
Then the standard unfolding lemma tells us that 
\[\int_{\Gamma \backslash D} \eta\wedge\sum_{\by\in \Gamma \cdot \bx}\varphi(z,g'.\by)=\int_{\Gamma_U \backslash D} \eta\wedge\varphi(z,g'.\bx)=\kappa(g',\beta) \int_{\Gamma_U \backslash D_{U,z'}} \eta.\]
This identity is nothing but Theorem \ref{PDthm}.

As we have mentioned, in general $\kappa(g',\beta)$ is hard to compute. One of the reasons is that $F D_{U,z'}$ is not a sub symmetric space of $D$ except in the Kudla-Millson cases.
So instead we show that the asymptotic value of $\kappa(m'(a(t)),\beta)$
when $t\rightarrow \infty$ (see Section \ref{methodofLaplacesection} for the meaning of $m'(a(t))$) is nonzero  using the method of Laplace. This will imply Theorem \ref{thmB}.

\subsection{Outline of the paper}
Section \ref{locallysymmetricspace} reviews the definition of $D$ and constructs compact arithmetic quotients of $D$. Section \ref{generalizedspecial} defines the generalized special cycle $C_{\bx,z'}$ and show that they are algebraic subvarieties. Section \ref{dualpair} reviews some fact about the Weil representation and set up coordinate functions for later use. Section \ref{Andersoncocyle} reviews the results of \cite{Anderson}, constructs the special class $\varphi$ and proves that it is closed. In Section \ref{thomformsection} we will prove Theorem \ref{PDthm} assuming the rapid decrease of ${\varphi}$ on the fiber $F D_{U,z'}$. Section \ref{rapiddecreasechapter} proves the rapid decrease of ${\varphi}$ on the fiber $F D_{U,z'}$. Section \ref{methodofLaplacesection} proves Theorem \ref{thmB} by the method of Laplace. Appendix \ref{appendix} described the $\tilde{K}'$-type of $\varphi$ in terms of highest weight theory. Readers who are familiar with arithmetic groups and the Weil representation can pick up the definition of the generalized special cycles in Section \ref{generalizedspecial} and then proceed to section \ref{Andersoncocyle} directly, only go back to Section \ref{dualpair} if necessary. For Section \ref{dualpair}, Section \ref{Andersoncocyle}, Section \ref{rapiddecreasechapter} and Section \ref{methodofLaplacesection}, one can focus only on the case A for first reading as the other two cases are similar.\\
\subsection{Acknowledgements} First I would like to thank my thesis advisor John Millson for introducing me to the subject, for studying the closed form constructed by \cite{Anderson} together with me, and for asking valuable questions and checking some of the proofs in the paper. I would like to thank Jeffrey Adams for teaching me useful knowledge of the Weil representation and for carefully reading the paper and provide valuable suggestions. I would also like to thank Michael Rapoport and Tonghai Yang for helpful suggestions on the definition of generalized special cycles. I would like to thank Patrick  Daniels and Hanlong Fang for helpful discussions. Lastly, I would like to thank the referee of the paper for valuable comments. The work is partially supported by NSF grant DMS-1518657.

\subsection{Notations and conventions}
In Section \ref{locallysymmetricspace} and Section \ref{thomformsection} we will let $k$ be a totally real number field, $F$ be a CM extension of $k$, $B$ be a division algebra with center $k$, $V$ be a vector space over $k$ and $G$ be an algebraic group over $k$. In case A, $V$ will be equipped with a Hermitian form.

In Section \ref{dualpair}, Section \ref{Andersoncocyle}, Section \ref{rapiddecreasechapter}, Section \ref{methodofLaplacesection} and the appendix, $V$ will denote a real vector space and $G$ will denote a real group. In order to relate case A, B and C using seesaw dual pairs and have uniform statements of results, we equip $V$ with a skew Hermitian form in case A. The identification between Hermitian and skew Hermitian forms is not canonical. In this paper we multiply a skew Hermitian form by $i=\sqrt{-1}$ to get a Hermitian form if such an identification is necessary and adjust the statements of our theorems correspondingly.

\noindent
{\bf Data availability statement:}
The paper has no associated data.

\part{Generalized special cycles on Hermitian locally symmetric spaces}
\section{Hermitian locally symmetric spaces}\label{locallysymmetricspace}
In this section we recall the construction of the Hermitian locally symmetric manifolds that are relevant to us. These manifolds are compact arithmetic quotients of Hermitian symmetric domains associated to the groups $\rU(p,q)$, $\Sp(2n,\R)$ and $\rO^*(2n)$ and are projective algebraic varieties.

\subsection{}\label{compactquotient}
Let $k$ be a totally real number field with $\ell$ distinct embeddings $\sigma_1, \ldots,\sigma_\ell$ into $\R$. Let $S_\infty=\{v_1,\ldots,v_\ell\}$ be the set of Archimedean places and $k_{v_1},\ldots,k_{v_\ell}$ be the corresponding completions. Let $F$ be a CM field whose maximal real subfield is $k$. There are $\ell$ pairs of conjugate embeddings of $F$ into $\C$. We choose one inside each pair, denote them by $\sigma_1,\ldots,\sigma_\ell$ by abusing notation. To obtain compact quotients of the symmetric spaces, we assume $k\neq \Q$. 

Let $(B,\sigma)$ be a $k$-algebra with involution of one of the following types:
		\begin{equation}\label{B}
		    (B,\sigma)=\begin{cases}
		       (\text{the CM field }$F$,\text{the generator of }$Gal(F/k)$), \\
		       (\text{a quaternion algebra with center $k$, the main involution}).
		\end{cases}
		\end{equation}
Let $V=B^n$ regarded as a right $B$ vector space. Let $(,)$ be a non-degenerate skew Hermitian form or Hermitian form on $V$ satisfying 
\begin{equation}\label{(,)}
		    (v b, \tilde{v} \tilde{b})=b^\sigma (v,\tilde{v}) \tilde{b}
\end{equation}
for $v,\tilde{v}\in V$ and $b,\tilde{b}\in B$. Let $G$ be the group defined by the equation
\begin{equation}\label{G}
  G=\{g\in \GL_B(V)\mid (gv,g\tilde{v})=(v,\tilde{v}), \forall v,\tilde{v}\in V\}.
\end{equation}
Define $V_v=V\otimes_k k_v$ and $G_v=G(k_v)$, where $v\in S_\infty$. Extend $(,)$ to $V_{v}$ and denote the new form by $ (,)_{v}$. Also define
		\[V_\infty=\prod_{v\in S_\infty} V_v,\ G_\infty=\prod_{v\in S_\infty} G_v.\]
We require the form $(,)$ to be anisotropic, which is to say that there is no nonzero vector $v\in V$ such that $(v,v)=0$. As a set, the symmetric space $D$ is defined by 
\[D=G_\infty/ K_\infty,\]
where $K_\infty$ is a maximal compact subgroup of $G_\infty$. Later in this section we will see case by case that $D$ can be regarded as an open subset of a (Lagrangian) Grassmannian. Hence $D$ is a complex variety.

Let $\mathcal{O}_k$ be the ring of integers of $k$ and $\Oo_B$ be the integral closure of $\mathcal{O}_k$ in $B$. Let
\[\cL=\Oo_B^n\subset B^n=V,\]
and define
\[G(\mathcal{O}_k)=\{g\in G\mid g \cL = \cL\}.\]
For an ideal $I$ in $\Oo_B$, define $G(I)$ to be the congruence subgroup 
\begin{equation}\label{G(I)}
    G(I)=\{g\in G(\mathcal{O}_k)\mid g \equiv I_n \ (\mathrm{mod} \ I \cL)\}.
\end{equation}
By  Theorem 17.4 of \cite{Borel}, we can choose an ideal $J$ of $\Oo_B$ such that $G(J)$ is neat. In particular, $G(J)$ acts simply on the symmetric space $D$. 
Fix an ideal $I\subseteq J$ and let $\Gamma=G(I)$. Since we assume that $(,)$ is anisotropic, $\Gamma$ is a co-compact subgroup in $G_\infty$. Moreover by a theorem of Baily and Borel (\cite{BailyBorel}), $\Gamma\backslash D$ is a complex projective variety.

We need a lemma.
		\begin{lemma}\label{choiceofc}
		For each $1\leq i\leq \ell$, let $U_i$ is an open subset of $k_{v_i}\cong\R$.
		There is a $c\in k$ such that $\sqrt{c} \notin k $ and $\sigma_i(c)\in U_i$ for all $1\leq i \leq \ell$.
		\end{lemma}
		\begin{proof}
		Choose a prime ideal $\mathfrak{p}$ of $\Oo_k$ such that $\mathcal{O}_k/ \mathfrak{p}$ is not a field of characteristic two. As taking square is a two to one map on $\mathcal{O}_k/ \mathfrak{p}$, there exists a $b\in \mathcal{O}_k$ such that  the equation $x^2\equiv b \ (\mathrm{mod} \ \mathfrak{p})$ has no solution in $\mathcal{O}_k/\mathfrak{p}$. Thus $x^2=b$ has no solution in $\mathcal{O}_k$ and $k$. Now choose $\epsilon$ small enough such that $x^2=a$ has a solution when $|a-1|_{\mathfrak{p}}\leq \epsilon $. By the weak approximation theorem, there exists a $c\in k$ such that 
		\begin{enumerate}
		    \item $\sigma_i (c)\in U_i \ (1\leq i \leq \ell)$, \
		    \item $|c-b|_{\mathfrak{p}}< |b|_{\mathfrak{p}} \cdot \epsilon$.
		\end{enumerate}
		Then $c$ satisfies the assumption of the lemma.
		\end{proof}
	
Now we carry out the above construction and give a more detailed description of $D$ in each case. 
		
		\subsection{Case A}

		Choose $d_1,\ldots, d_{p+q}\in k$ such that
		\begin{enumerate}
		    \item $\sigma_1(d_\alpha)>0$ and $ \sigma_1(d_\mu)<0$ for $1\leq \alpha \leq p$ and $p+1\leq \mu \leq p+q$,\
		    \item $\sigma_i(d_j)>0$ for $2\leq i \leq \ell$ and $1\leq j \leq p+q$.
		\end{enumerate}
		This is possible by the weak approximation theorem. Let V be a $p+q$ dimensional right $F$ vector space and $(,)$ be the Hermitian form defined by the diagonal matrix with diagonal entries $d_1,\ldots,d_{p+q}$. If G is defined by equation \eqref{G}, we have 
		\begin{enumerate}
		    \item  $G_{v_1}\cong \Uni(p,q)$,\
		    \item $G_{v_i}\cong \Uni(p+q)$ for $2\leq i \leq \ell$.
		\end{enumerate}
		Since $(,)_{v_i}$ is definite for $2\leq i \leq \ell$, $V$ is anisotropic. The symmetric space $D$ can be defined by
		\[D=\{z\text{ is a subspace of } V_{v_1}\mid \mathrm{dim}_{\C} z=q, (,)_{v_1}|_{z} \text{ is negative definite}\}.\]
		
		\subsection{Case B}\label{locallysymmetricspaceSp}
		By Lemma \ref{choiceofc}, we can choose $c_1,c_2$ such that
		\begin{enumerate}
		    \item $\sigma_1(c_j)>0$ and $\sqrt{c_j}\notin k$ for $j=1 \text{ or }2$,
		    \item $\sigma_i(c_j)<0$  for $j=1,2$ and $2\leq i \leq \ell$.
		\end{enumerate} 
		Let $B=\H_k(c_1,c_2)$ be the quaternion algebra over $k$ generated by $\epsilon_1,\epsilon_2$ with relations
		\begin{equation}\label{definitionrelationofB}
		    \epsilon_1^2=c_1,\epsilon_2^2=c_2,\epsilon_1 \epsilon_2=- \epsilon_2 \epsilon_1.
		\end{equation}
		We put $\epsilon_3=\epsilon_1 \epsilon_2$. Then an element $\xi\in B$ can be written as $\xi=\xi_0+ \xi_1 \epsilon_1+ \xi_2 \epsilon_2+\xi_3 \epsilon_3$, where $\xi_j\in k$ for $0\leq j \leq 3$. We define an anti-involution $\sigma$ on $B$ by 
		\[\sigma(\xi)=\xi_0- \xi_1 \epsilon_1- \xi_2 \epsilon_2-\xi_3 \epsilon_3.\]
		With the given assumption we know that
	    \[ B\otimes_k k_{v_1}\cong M_2(\R) , \ 
		B\otimes_k k_{v_i}\cong \H\] 
        for $2\leq i \leq \ell$, where $\H$ is the classical Hamiltonian quaternions. Now let V be a $n$-dimensional right $B$ vector space and $(,)$ be a Hermitian form on $V$ satisfying \eqref{(,)}. Let $G$ be the group defined by \eqref{G}. 
        
		Let $E$ be any field that contains $k[\sqrt{c_1}]$. We define an anti-involution $\sigma'$ on $M_2(E)$ by
		\[\sigma'(x)=J  x^t J^{-1},\]
		where $J=\left(\begin{array}{cc}
		    0 & -1 \\
		    1 & 0
		\end{array}\right)$. We can embed $B$ into $M_2(E)$ as follows
		\[i(\xi_0+ \xi_1 \epsilon_1+ \xi_2 \epsilon_2+\xi_3 \epsilon_3)=
		\left(\begin{array}{cc}
		   \xi_0+\xi_1 \sqrt{c_1}  & c_2(\xi_2+\xi_3 \sqrt{c_1}) \\
		    \xi_2-\xi_3 \sqrt{c_1} & \xi_0- \xi_1 \sqrt{c_1}
		\end{array}
		\right).\]
		It is easy to check that $\sigma'\circ i=i\circ \sigma$, so from now on we abuse notation and denote both involutions by $\sigma$. 
		Let $e_{i j}$ be the matrix with the $(i,j)$-th entry 1 and all the other entries zero. Let $e=e_{11}$. As $B\otimes_k E\cong M_2(E)$, we get a decomposition 
		\[V_{E}=V_{E}e+V_{E} e^\sigma\]
		as a $E$ vector space, where $V_E=V\otimes_k E$. Let $S_E$ be the $E$-bilinear form on $ V_E e$ defined by 
		\begin{equation}
		    S_E(x e, y e)e_{21}=(x e,y e).
		\end{equation}
		Following an argument like that of Page 368 of \cite{LiMillson}, we see that $S_E$ is skew symmetric and
		\[G_E\cong  \Sp((V_E e,S_E))\cong \Sp(2n,E).\]
		In particular since $k_{v_1}\cong \R$ and $\sigma_1(c_1) >0$, we know that 
		\[G_{v_1}\cong \Sp((V_{v_1}e,S_{v_1}))\cong \Sp(2n,\R).\]
		And we also have
		\[G_{v_i}\cong \Sp(p_{v_i},q_{v_i})\]
		for $2\leq i \leq \ell$, where $p_{v_i}+q_{v_i}=n$. Moreover we choose the form $(,)$ to be defined by a diagonal matrix with diagonal entries $d_1,\ldots,d_n \in k$ satisfying
		$\sigma_i(d_j)>0$
		for $2\leq i \leq \ell$ and $1\leq j \leq n$. We then have
		$G_{v_i}\cong \Sp(n)$ for $2\leq i \leq n$.
		With this choice $G$ is anisotropic as $\Sp(n)$ is. We have
		\[G_\infty \cong \Sp(2n,\R)\times \Sp(n)^{\ell-1}.\]
		The symmetric space $D$ is then the set of $n$ dimensional complex subspace $z$ of the $2n$ dimensional complex vector space $(V_{v_1} e)\otimes_{\R} \C$ such that
		\begin{enumerate}
		    \item $S_{k_{v_1}}(,)|_z$ is zero ($z$ is Lagrangian for $S_{k_{v_1}}(,)$. 
		    \item The Hermitian form $i S_{k_{v_1}}(\bar{\ },)$ is negative definite on $z$, where $\bar{\ }$ is the complex conjugation on $(V_{v_1} e)\otimes_{\R} \C$.
		\end{enumerate}
		
		\subsection{Case C}\label{locallysymmetricspaceOstar}
		The construction of $G$ in this case is similar to that in case B. Let $k$ be the same totally real number field. By Lemma \ref{choiceofc}, we can choose $c_1,c_2$ such that 
		\begin{enumerate}
		    \item $\sigma_1(c_1)<0$, $\sigma_1(c_2)<0$ and $\sqrt{c_1},\sqrt{c_2}\notin k$,
		    \item $\sigma_i(c_1)>0$, $\sigma_i(c_2)<0$ for $2\leq i \leq \ell$.
		\end{enumerate}
        Let $B=\H_k(c_1,c_2)$ as in equation \eqref{definitionrelationofB}. This time we have
		\[ B\otimes_k k_{v_1}\cong \H,\ B\otimes_k k_{v_i}\cong M_2(\R)\] 
        for $2\leq i \leq \ell$.  Now let V be a $n$-dimensional right $B$ vector space and $(,)$ be a skew Hermitian form on $V$ satisfying \eqref{(,)}. Let $G$ be the group defined as in \eqref{G}. Then we see that \[G_{v_1}\cong \Orth^*(2n).\]
		Following a result on Page 368 of \cite{LiMillson}, we also have
		\[G_{v_i}\cong \Orth(p_{v_i},q_{v_i})\]
		for $2\leq i \leq \ell$, where $p_{v_i}+q_{v_i}=2n$. Moreover we choose $d_1,\ldots,d_n \in k$ such that $\sigma_i(d_j)>0$ for all $2 \leq i \leq \ell$, $1\leq j \leq n $, and let $(,)$ be the skew Hermitian form defined by the diagonal matrix with diagonal entries $d_1 \epsilon_2,\ldots, d_n \epsilon_2 $. By Lemma 2.1 of \cite{LiMillson}, we know that $G_{v_i}$ ($2\leq i \leq \ell$) is the orthogonal group defined by the block diagonal matrix
		$S_{v_i}$ with 2 by 2 diagonal blocks $-J (d_1 \epsilon_2),\ldots ,-J (d_n \epsilon_2)$. Since
		\[-J (d_j \epsilon_2)=\left(\begin{array}{cc}
		   0  & 1 \\
		    -1 & 0
		\end{array}\right)\cdot \left(
		\begin{array}{cc}
		  0   &  c_2 d_j\\
		   d_j  & 0
		\end{array}\right)=\left(
		\begin{array}{cc}
		  d_j   &  0\\
		   0  & -c_2 d_j
		\end{array}\right),\]
		by our assumptions $S_{v_i}$ will be positive definite for $2\leq i \leq \ell$. Thus $G_{v_i}\cong \Orth(2n)$ for $2\leq i \leq \ell$. This implies that $G$ is anisotropic and 
		\[G_\infty \cong \Orth^*(2n,\R)\times \Orth(2n)^{\ell-1}.\]
		We define a skew Hermitian form $H(,)$ on $V_{v_1}$ regarded as a $2n$ dimensional complex vector space by the equation
		\begin{equation}
		    H(v,w)=a+bi \text{ if } (v,w)=a+bi+cj +d k.
		\end{equation}
		Also define a symmetric form $S(,)$ on $V_{v_1}$ regarded as a $2n$ dimensional complex vector space by the equation
		\begin{equation}
		    S(v,w)=H(v j, w).
		\end{equation}
		Then the symmetric space $D$ is the set of $n$ dimensional complex subspace $z$ of $V_{v_1}$ such that 
		\begin{enumerate}
		    \item $S(,)|_z$ is zero ($z$ is Lagrangian for $S(,)$),
		    \item the Hermitian form $ iH(,)|_z$ is negative definite. 
		\end{enumerate}

\section{Generalized special cycles}\label{generalizedspecial}
In this section we define generalized special cycles in $\Gamma\backslash D$ , where $D$ is the symmetric space associated to $G_\infty=\prod_{v\in S_\infty} G_v$ or equivalently $G_{v_1}$ as all the other factors in $G_\infty$ are compact. To ease notations, in this section we write $G$ for $G_{v_1}$ and assume that 
\begin{enumerate}
    \item In case A: $B=\C$, $V\cong\C^{p+q}$, $(,)$ is Hermitian of signature $(p,q)$ on V,
    \item In case B: $B=\R$, $V\cong \R^{2n}$. $(,)$ is symplectic on $V$,
    \item In case C: $B=\H$, $V\cong \H^{n}$. $(,)$ is non-degenerate skew Hermitian on $V$.
\end{enumerate}
In particular $G=\rU(p,q)$ in case A, $\Sp(2n,\R)$ in case B and $\rO^*(2n)$ in case C. On the symmetric space level, a generalized special cycle is the fixed point set of a pair of involutions while a special cycle in the sense of Kudla and Millson is the fixed point set of a single involution. The definition of a generalized special cycle depends on a choice of a non-degenerate $U\subseteq V$ together with a point in the symmetric space associated to $U$.

As we have seen in the previous section, in each case of our interests, the symmetric space $D$ can be described as an open subset of a (Lagrangian) Grassmannian. Meanwhile $D$ can also be described as the set of Cartan involutions of $G$. The relation between these two descriptions is as follows.
For any $z\in D$, we can find an element $r_z\in G$ such that 
\begin{equation}
    D^{r_z}:=\{x\in D\mid r_z x=x\}=\{z\}.
\end{equation}
In case A we can simply take $r_z=Id|_z\oplus (-Id)|_{z^\bot}$. In the other two cases we can take $r_z$ to be an element in the center of $G_z=\{g\in G\mid gz=z\}$ not equal to $\pm 1$. Define $\sigma_z=\mathrm{Ad}(r_z)\in \mathrm{Inn}(G)$. Then $\sigma_z$ is the Cartan involution associated to $z$ and we have
\[G_z=G^{\sigma_z}.\]

Let 
\[\bx=(x_1,\ldots,x_m)\in \cL^m.\]
We require $U=\mathrm{span}_{B}(\bx)$ to be non-degenerate with respect to $(,)$.
We then have the decomposition $V=U+U^{\bot}$. For any non-degenerate $U$ we define 
\begin{equation}
  G(U)=\{g \in G\mid g v=v,\forall v\in U^\bot\}.
\end{equation}
Let $D(U)$ be the symmetric space associated to $G(U)$. 

\begin{remark} In order to be consistent with the notation of the work of Kudla and Millson, we will also use the notations $G_U$ and $ D_U$ :
\begin{equation}
G_U=G(U^\bot), \ D_U= D(U^\bot).
\end{equation}
\end{remark}

Define $\Gamma_U=\Gamma\cap G_U$.  By our construction $\Gamma_{U} \backslash D_U$ is a compact locally symmetric manifold. We want to map $\Gamma_{U} \backslash D_U$ to $ \Gamma \backslash D$ to get an algebraic cycle. This requires the choice of a point in $D(U)$ as we will see.

For any $\sigma\in \mathrm{Aut}(G)$ and $H
\leq G$ such that $\sigma(H)=H $, define
\[H^{\sigma}=\{g\in H \mid   \sigma( g ) =g \}.  \]
Let $r_U=Id|_U \oplus (-Id)|_{U^\bot}$ and $\sigma_U=\mathrm{Ad}(r_U)\in \mathrm{Inn}(G)$.
It is easy to see that
\[G^{\sigma_U}=G(U)\times G(U^\bot).\]
Consequently the symmetric space of $G^{\sigma_U}$ is the product $D(U)\times D(U^\bot)$.
There is a natural embedding $\rho_U$ defined by
\[\rho_U:D(U)\times D(U^\bot)\hookrightarrow D:(z_1,z_2)\mapsto z_1\oplus z_2.\]
For a vector $v\in z $, we have $r_U(v)\in z$ if and only if both $\mathrm{proj}_U(v)$ and $\mathrm{proj}_{U^\bot}(v)$ are in $z$. Hence we have 
\begin{equation}
    z\in D^{r_U}:=\{z\in D\mid r_U(z)=z\}\Leftrightarrow z=z\cap U \oplus z\cap U^\bot.
\end{equation}
Hence $\rho_{U}$ defines an isomorphism from $D(U)\times D(U^\bot)$ onto $D^{r_U}$. From now on we denote $D^{r_U}$ by $D(U,U^\bot)$.

Now let $z'\in D(U)$. 
We define an embedding $i_{z'}$:
\[i_{z'}:D_U\hookrightarrow D(U)\times D(U^\bot), z \mapsto (z',z).\]
Let $\rho_{U,z'}$ be the composition $\rho_U\circ i_{z'} $. In other words
\begin{equation}\label{rhoUz'}
    \rho_{U,z'}(z)=z'\oplus z.
\end{equation}
\begin{definition}\label{specialsubsymmetricspace}
Denote by $ D_{U,z'}$ (or $D_{\bx,z'}$) the image of $D_U$ under the map $\rho_{U,z'}$. We call it a generalized special sub symmetric space of $D$.
\end{definition}

As we have explained, it is possible to choose a $r_{z'}\in G(U)$ such that
\[D(U)^{r_{z'}}=\{z'\},\  G(U)^{\mathrm{Ad}(r_{z'})}=\{g\in G(U)\mid gz'=z'\}.\]
We now define $\sigma_{z'}=\mathrm{Ad}(r(z'))\in \mathrm{Inn}(G) $, where $r(z')=r_{z'}\oplus Id_{U^\bot}$.
Apparently $\sigma_U \sigma_{z'}=\sigma_{z'} \sigma_U$, hence we know that $G^{\sigma_U}$ is  $\sigma_{z'}$-stable. Then we define
\[G^{\sigma_U,\sigma_{z'}}:=(G^{\sigma_U})^{\sigma_{z'}}=(G^{\sigma_{z'}})^{\sigma_U}.\]
Since $G^{\sigma_U}=G(U)\times G(U^\bot)$, we know that 
\[G^{\sigma_U,\sigma_{z'}}=G(U)^{\sigma_{z'}}\times G(U^\bot).\]
Notice that $G(U)^{\sigma_{z'}} $ is the maximal compact subgroup of $G(U)$ fixed by $\sigma_{z'}$. The symmetric space of $G^{\sigma_U, \sigma_{z'}}$ is $D_U$ which we identify with $D_{U,z'}$. Since $D^{r_U}=D(U,U^\bot)$ and $D(U,U^\bot)^{r(z')}=D_{U,z'}$, we have
\begin{proposition}\label{fixpointofpairofinvolution}
\[D_{U,z'}=D^{r_U, r(z')}:=(D^{r_U} )^{r(z')}.\]
In particular $D_{U,z'}$ is totally geodesic.
\end{proposition}

$ i_{z'}$, $\rho_U$ and $\rho_{U,z'}$ induce  maps (still denoted as  $ i_{z'}$, $\rho_U$ and $\rho_{U,z'}$) of locally symmetric spaces
\[\Gamma_U\backslash D_U\rightarrow \Gamma^{\sigma_U}\backslash D(U,U^\bot),\ \Gamma^{\sigma_U}\backslash D(U,U^\bot)\rightarrow \Gamma \backslash D ,\ \Gamma_U\backslash D_U\rightarrow \Gamma \backslash D\]
respectively.
Apparently the induced map $ i_{z'}$ is always  an embedding. But in general $\rho_U$ (hence $\rho_{U,z'}$) will not be injective and if this is the case the image of $\rho_{U,z'}$ will have self intersections and not be a manifold. The following lemma is Lemma 2.1 of \cite{KM90} which "resolves the singularity" of the image.
\begin{lemma}\label{resolutionofsingularity}
There is an arithmetic subgroup $\Gamma'\subseteq \Gamma$ of finite index such that the following diagram commutes, and $\rho'_U$ is an embedding.
\[ \begin{tikzcd}
(\Gamma')^{\sigma_U} \backslash D(U,U^\bot) \arrow{r}{\rho'_U} \arrow[swap]{d}{} & \Gamma' \backslash D \arrow{d}{} \\%
\Gamma^{\sigma_U} \backslash D(U,U^\bot) \arrow{r}{\rho_U}& \Gamma \backslash D.
\end{tikzcd}
\]
In particular, $\rho'_{U,z'}:\Gamma'_U \backslash D_U \rightarrow \Gamma'\backslash D$ is also an embedding.
\end{lemma}

We denote the image of $\rho_{U,z'}$ by $C_{U,z'} $. We sometimes also use $C_{\bx,z'}$ to denote $C_{U,z'} $ for convenience but it really just depends on $U$.

\begin{theorem}\label{specialcyclesaresubvarieties}
$C_{U,z'}$ is an algebraic subvariety in $M=\Gamma\backslash D$.
\end{theorem}
\begin{proof}
First let us assume that the map $\rho_{U,z'}: \Gamma_U \backslash D_U\rightarrow \Gamma \backslash D$ is an embedding.
The subsymmetric space $D_{U,z'}$ is a complex analytic subvariety of $D(U,U^\bot)$ since it is a factor. We know that $r_U$ is an isometry of $D$, hence automatically an holomorphic map (Lemma 4.3 of \cite{Helgason}). Hence $D(U,U^\bot)=D^{r_U}$ is a complex analytic subvariety of $D$. So $D_{U,z'}$ is a complex analytic subvariety of $D$. Since $\rho_{U,z'}$ is an embedding, locally $D_{U,z'}$ and $C_{U,z'}$ are defined by the same analytic equations. In particular, $C_{U,z'}$ is complex analytic. So it is a complex algebraic subvariety of $\Gamma \backslash D$ as well by the main theorem of \cite{Chow}.

In the general case, we apply Lemma \ref{resolutionofsingularity}. By the previous argument we see that  $ C'_{U,z'}$ (the image of $\Gamma'_U \backslash D_U$ under $\rho'_{U,z'}$) is a complex algebraic subvariety of $\Gamma' \backslash D$. Since the map $ f:\Gamma'\backslash D \rightarrow \Gamma \backslash D$ is an analytic covering between complex projective varieties, it is automatically a regular map of complex projective algebriac varieties by \cite{Serre}. Hence $f$ is projective by Lemma 28.41.15 of Stack Project. Being a finite covering map, $f$ is automatically quasi-finite, hence a finite morphism. Hence $f$ is proper and in particular closed.  Then the image $C_{U,z'}=f(C'_{U,z'})$ is a closed subvariety of $\Gamma \backslash D$.
\end{proof}

\begin{definition}\label{definitionofspecialcycle}
 We call $C_{U,z'} $ (or $C_{\bx,z'}$) a generalized special cycle.
\end{definition}

We will need the following lemma later.
\begin{lemma}\label{resolutionofsingularity2}
The map $\Gamma_U \backslash D_U\overset{\rho_{U,z'}}{\longrightarrow} \Gamma \backslash D$ is a finite birational morphism onto its image. 
\end{lemma}
\begin{proof}\
Since $g D_{U,z'}=D_{gU,gz'}$, we know that the stabilizer of $D_{U,z'}$ is exactly $G^{\sigma_U,\sigma_{z'}}=G(U)^{\sigma_{z'}}\times G_U$.
$\forall \gamma \in \Gamma-\Gamma^{\sigma_U,\sigma_{z'}}$, define
\[D_\gamma=D_{U,z'} \cap \gamma D_{U,z'}.\]
Then $D_\gamma$ is an analytic subset of $D_{U,z'}$. We claim that it is a proper subset. Otherwise $\gamma$ is in the stabilizer of $D_{U,z'}$. Hence $\gamma\in \Gamma^{\sigma_U,\sigma_{z'}}$, a contradiction. 

The image $V_\gamma$ of $D_\gamma$ under the natural quotient map $D_U\rightarrow \Gamma_U \backslash D_U$ is a proper analytic sub variety of $\Gamma_U \backslash D_U$. Define
\[V=\bigcup_{\gamma \in \Gamma-\Gamma^{\sigma_U,\sigma_{z'}}}V_\gamma.\]
Then $\rho_{U,z'}:\Gamma_U \backslash D_U\rightarrow \Gamma\backslash D$ is injective outside $V$.

The map $\rho_{U,z'}:\Gamma_U\backslash D_U\rightarrow \Gamma\backslash D$ factors through $\phi:\Gamma_U\backslash D_U\rightarrow C_{U,z'}$. As in the situation of Lemma \ref{resolutionofsingularity}, we have the following  commutative diagram 
\[ \begin{tikzcd}
\Gamma'_U \backslash D_U \arrow{r}{\phi'} \arrow[swap]{d}{} & C'_{U,z'} \arrow{d}{} \\%
\Gamma_U \backslash D_U \arrow{r}{\phi}& C_{U,z'}.
\end{tikzcd}
\]
By the above commutative diagram and the fact that $\phi'$ is an isomorphism $\phi$  is quasifinite. It is a projective morphism by Lemma 28.41.15 of Stack Project, hence is a finite morphism. By the argument in the previous paragraph, it is injective outside a set of measure $0$ with respect to the measure defined by the K\"ahler metric on $D_U$. Hence the degree of the finite morphism $\rho_{U,z'}$ must be $1$. It must be a birational morphism.
\end{proof}

\begin{proposition}\label{homologyindependent}
The homology class $[C_{U,z'}]\in H_*(\Gamma \backslash D)$ does not depends on the choice of $z'\in D(U)$.
\end{proposition}
\begin{proof}
For any two $z', z'' \in D(U)$, there is a continuous path 
\[c:[0,1]\rightarrow D(U)\]
such that $c(0)=z'$, $c(1)=z''$. Thus we can define a map $D_U \times [0,1]\rightarrow D $ by
\[(z,t)\mapsto  c(t)\oplus z.\]
Since $\Gamma_U$ fixes $U$, this map defines a map $\Gamma_U \backslash D_U \times [0,1]\rightarrow  \Gamma\backslash D $  which is a homotopy equivalence between two different embeddings of $\Gamma_U \backslash D_U$.
\end{proof}
\begin{remark}
From now on we specify the choice of embedding if necessary, otherwise we use the notation $[C_U]$ or  $[C_{\bx}]$ to refer the homology class of $C_{U,z'}$.
\end{remark}

We now illustrate the above abstract construction in case A.
\subsection{Example: Case A}
Recall that the symmetric space $D$ can be identified with the set of negative $q$-planes in $V$
		\[D=\{z\in \Gr_q(V)\mid (,)|_z <0 \}.\]
Let $U$ be a $F$-subspace of $V$. If $(,)|_{U}$ has signature $(r,s)$, then we have $G(U^\bot)\cong \Uni(p-r,q-s)$ and
		\[D_U=\{z\in \Gr_{q-s}(U^\bot)\mid   (,)|_z <0 \}.\]
The choice of a point $z'\in D(U)$ is the same as a choice of an orthogonal decomposition of $U$
\[U=U^{+}\oplus U^{-}\]
with respect to $(,)$ such that $(,)|_{U^{+}}$ is positive definite and $(,)|_{U^{-}}$ is negative definite. Given such a decomposition of $U$ we have $z'=U^-$ and is $s$ dimensional. 
Under the embedding $\rho_{U,z'}$ defined in \eqref{rhoUz'} we have 
\begin{equation}\label{Dlambda0}
    D_{U,z'}=\{z\in D\mid U^{-}\subseteq z \subseteq (U^{+})^\bot\}.
\end{equation}

\begin{proposition}\label{DU=DXcapDY}
When $U$ is positive or negative definite, the above embedding is canonically defined. To be more precise, the choice of $z'$ is unnecessary and we have $D_{U,z'}=D(U,U^\bot)$. Moreover we have
\begin{enumerate}
    \item When $U$ is positive definite, $ D_U=\{z\in D\mid  z \subseteq U^\bot\}$.
    \item When $U$ is negative definite, $ D_U=\{z\in D\mid U \subseteq z  \}$.
\end{enumerate}
\end{proposition}
\begin{proof}
When $U$ is positive (negative resp.) definite, the group $G(U)$ is compact and  the symmetric space $D(U)$ consists of a single point $z'$. Since
$D(U,U^\bot)=D(U)\times D_U$, the first statement is proved.
The statements in the enumeration follow from equation (\ref{Dlambda0}).
\end{proof}

\begin{remark}\label{definitionofKM}
When $U$ is positive (resp. negative) definite, we simply denote $C_{U,z'}$ by $C_U$. This is the situation in the work of Kudla and Millson (\cite{KMCompo} \cite{KMI},\cite{KMII} and \cite{KM90}). $C_U$ is called a special cycle there. 
\end{remark}

\subsection{}\label{fibrationpi}
Now we construct a fibration $\pi:D \to D_{U,z'}$. 
For each $z_0\in D_{U,z'}$, we define
\[N_{z_0} D_{U,z'}=\{v\in T_{z_0} D\mid v\perp T_{z_0} D_{U,z'} \}.\]
Then $N D_{U,z'}=\bigcup_{z_0\in D_{U,z'}} N_{z_0} D_{U,z'}$ is the normal bundle of $D_{U,z'}$ in $D$.
Then the Riemannian exponential map induces a map $F: N D_{U,z'} \to D$ by the formula
\begin{equation}
    F(z_0,v) = \exp_{z_0}(v).
\end{equation}
The image of the line through $v\in N_{z_0} D_{U,z'}$ under the exponential map is a geodesic through $z_0$ orthogonal to $D_{U,z'}$. Since $D_{U,z'}$ is totally geodesic in $D$ which is negatively curved, Theorem 14.6 of \cite{Helgason} tells us that 
$D$ is the disjoint union of the geodesics which are perpendicular to $D_{U,z'}$.  Moreover by a standard Jacobi field calculation we can show that
\begin{lemma} \label{expofnormalbundlefibering}
$F:N D_{U,z'} \to D$ is a diffeomorphism.
\end{lemma}
For $g\in G_U$ one can check that the two geodesics $g F(z_0,vt)$ and $F(gz_0,(L_g)_* vt)$ are the same by checking that they have the same starting point (at $t=0$) and have the same derivative there. Hence we know that 
\begin{equation}\label{FisGequivariant}
    g F(z_0,v)=F(gz_0,(L_g)_* v), \forall g\in G_U.
\end{equation}

By Lemma \ref{expofnormalbundlefibering} we can define $\pi:D\rightarrow D_{U,z'}$ by the equation
\begin{equation}\label{eq:D to D_U projection}
    \pi (F(z_0,v))=z_0.
\end{equation}
$\pi:D\rightarrow D_{U,z'}$ is isomorphic to $ND_{U,z'}\rightarrow D_{U,z'} $ as a fiber bundle. 
We denote by
\begin{equation}\label{fiber}
    F_{z_0} D_{U,z'}:= \pi^{-1}(z_0),
\end{equation}
the fiber of $\pi$ for any $z_0\in D_{U,z'}$.
By equation \eqref{FisGequivariant} and the definition of $\pi$, we see that if $z=F(z_0,v)$ then
\begin{equation}\label{eq:equivariant fibration}
    \pi(g z)=\pi(g F(z_0,v))=\pi(F(g z_0,(L_g)_* v))=g z_0=g\pi(z)
\end{equation}
for any $g\in G_U$. Thus $\pi$ induces a fibration $\pi:\Gamma_U \backslash D\rightarrow \Gamma_U \backslash D_{U,z'}$ which we still denote by $\pi$.

\begin{remark}
In case A when $U$ is positive or negative definite, $F_{z_0} D_{U,z'}$ as in \eqref{fiber} is a sub symmetric space of $D$. We refer the readers to \cite{KM90} for an explanation of this fact as we do not need it in this paper. Except for these cases, $F_{z_0} D_{U,z'}$ is not a sub symmetric space.
\end{remark}

\part{Relative Lie Algebra cohomology of the
Weil representation}

\section{The Weil representation and dual pairs}\label{dualpair}
Let $G$ be one of the algebraic group over $k$ introduced in Section \ref{compactquotient}. Then we can define another algebraic group $G'$ over $k$ such that $(G,G')$ forms a dual pair in the sense of \cite{Howe1}. Let $\bA$ be the ring of Adeles of $k$, then the Weil representation is a certain function space on which (a double cover of) $(G\times G')(\bA)$ acts. 
The main tool of this paper is relative Lie algebra cohomology with values in the Weil representation. However we need two different models of the Weil representation. One is called the Schr\"odinger model, where $G$ acts "geometrically", which is the natural model when we do differential geometry on $D$.
The other model is the Fock model, where the maximal compact group of $(G\times G')_\infty$ acts in a nice way. The Fock model is indispensable for the construction of Anderson's holomorphic forms (\cite{Anderson}). In this section we briefly recall knowledge of the two models and write down formulas of the intertwining operators between the two models. We also review seesaw dual pairs \eqref{seesawone} and \eqref{seesawtwo}.

\subsection{Dual reductive pairs and the Schr\"odinger model}\label{dualreductivepairs}
Let $(B,\sigma)$ be defined as in \eqref{B}. Let $\epsilon$ be $-1$ or $1$. In subsection \ref{compactquotient} we define a non-degenerate $\epsilon $-Hermitian form $(,)$ on a right $B$ vector space $V$. We can also define a non-degenerate $-\epsilon$-Hermitian form $\langle,\rangle$ on a left $B$ vector space $W$ which satisfy 
\[\langle b_1 w_1, b_2 w_2\rangle=b_1 \langle w_1,w_2\rangle b^\sigma_2,\ \langle w_1,w_2\rangle=-\epsilon \langle w_2,w_1\rangle^\sigma,\]
for  $w_1,w_2\in W$ and $b_1,b_2\in B$. Let $G$ be defined as in \eqref{G} and
\begin{equation}
G'=\{g' \in \GL_B(W)\mid \langle w_1 g', w_2 g'\rangle=\langle w_1,w_2\rangle \forall w_1,w_2 \in W \}.    
\end{equation}
We view $G$ and $G'$ as algebraic groups over $k$. Let 
\[n=dim_B V,\ 2m=dim_B W.\]
Also let
\[\W =V\otimes_B W \text{ and }  \langle \langle,\rangle \rangle=\mathrm{tr}_{B/k} ((,)\otimes \langle,\rangle^\sigma).\]

So that $\W$ is a $k$-vector space with the non-degenerate symplectic $\langle \langle,\rangle \rangle $. Then $(G,G')$ is a dual reductive pair in the sense of \cite{Howe2}.
		
Now we assume that $(W,\langle,\rangle)$ is split over $B$. i.e. there is a decomposition
		\[W=X+Y\]
with $B$ subspaces $X$ and $Y$ which are isotropic for $\langle,\rangle$. We can choose a standard symplectic basis with respect to this decomposition. This choice of basis gives rise to an isomorphism 
\[G'\cong \left\{ \left( \begin{array}{cc}
		   a  & b \\
		   c  & d
\end{array}\right)\in \GL_{2n}(B) \mid ad^*-b c^*=1, ab^*=ba^*,cd^*=dc^*\right\}, \]
where $a^*= (a^\sigma)^t$ and to isomorphisms 
\[\W\cong V^{2m},\ \mathbb{X}:=V\otimes_B X\cong V^m,\ \mathbb{Y}:=V\otimes_B Y\cong V^m.\]
The parabolic subgroup $P'\subset G'$ which stabilizes $Y$ then has the form 
		\[P'=\left\{\left( \begin{array}{cc}
		   a  & b \\
		    0 & d
		\end{array}\right)\in G'\right\}\]
and has unipotent radical
\begin{equation}\label{N'}
		  N'=\left\{ n'(b)=\left(\begin{array}{cc}
		   1  & b \\
		   0  & 1
		\end{array}\right)\mid b\in M_m(B) \text{ with } b^* =\epsilon b\right\}  
\end{equation}
and Levi factor
\begin{equation}\label{M'}
		    M'=\left\{m'(a)=\left(\begin{array}{cc}
		   a  & 0 \\
		    0 & \hat{a}
		\end{array}\right)\mid a\in \GL_m(B)  \text{ and } \hat{a}=(a^{-1})^*   \right\}.
\end{equation}
Such a parabolic subgroup is called a Siegel parabolic in the literature. 		
		
Fix a non-trivial additive chracter $\psi$ of $k_\mathbb{A}$ trivial on $k$ and let 
\[(\omega, L^2(\mathbb{X}))\cong (\omega, L^2(V(\mathbb{A})^{m}))\]
be the Schr{\"o}dinger model of the global Weil representation of $\widetilde{\Sp(\W(\mathbb{A}))}$, the two fold metaplectic cover of $\Sp(\W(\mathbb{A}))$, corresponding to $\psi$ and the polarization (1.7) in \cite{Weil} and \cite{Howe2}. Let $\tilde{G}'(\mathbb{A})$ denote the inverse image of $G'(\mathbb{A})$ in $\widetilde{\Sp(\W(\mathbb{A}))}$. Then the action of $\tilde{G}'(\mathbb{A})$ in $L^2(V(\mathbb{A})^m)$ defined by the restriction of $\omega$ to $\tilde{G}'(\mathbb{A})$ commutes with the natural action of $G(\mathbb{A})$ defined by 
\[\omega(g)\varphi(\bx)=\varphi(g^{-1}\bx),\]
where $g\in G(\mathbb{A}),\varphi \in L^2(V(\mathbb{A})^m)$.
		
The action of the parabolic subgroup $ \tilde{P}'(\mathbb{A})$ of $\tilde{G}'(\mathbb{A})$ is easy to describe. Fix a section of the covering $\widetilde{\Sp(\W(\mathbb{A}))}\rightarrow \Sp(\W(\mathbb{A}))$ and hence an identification:
\[\widetilde{\Sp(\W(\mathbb{A}))} \cong  \Sp(\W(\mathbb{A})) \times \mu_2.\]
Then 
\begin{equation}\label{M'action}
		   \omega(m'(a),\zeta)\varphi(\bx)=\zeta |a|^{\frac{m}{2}} \varphi(\bx a) ,
\end{equation}
where $|a|$ is the modulus of multiplying by $a$ on $B(\bA)^n$. Also
\begin{equation}\label{n'action}
		    \omega(n'(b),\zeta) \varphi(\bx)=\zeta \psi(\frac{1}{2}tr(b(\bx,\bx)))\varphi(\bx).
\end{equation}


For the rest of the section, we study real groups. So we switch notation and let $B=\C$ or $\H$, $V$ be a real vector space and $G$ be a real Lie group, etc.

\subsection{}\label{infinitesimalFockmodel} 
In this subsection we recall the construction of the infinitesimal Fock model together with the intertwining operator from the infinitesimal Fock model to the Schr\"odinger model by Section 6 of \cite{KM90}. The key to the construction of the intertwining operator is the action of Weyl algebra. 
In later subsections, we specialize to different dual pairs.

Let $\W$ be a be a vector space over $\R$ with a non-degenerate skew-symmetric form $\langle \langle,\rangle \rangle$ and $J_0$ be a positive definite almost complex structure (i.e. the form $\langle \langle ,  J_0 \rangle \rangle$ is a positive definite symmetric form) on $\W$. We may decompose $\W\otimes \C$ according to
\[\W\otimes \C= \W'+\W'',\]
where $\W'$ is the $+i$ eigenspace of $J_0$ and $\W''$ is the $-i$ eigenspace of $J_0$. Notice that both $\W'$ and $\W''$ are isotropic for $\langle \langle,\rangle \rangle$.

Choose a nonzero complex parameter $\lambda$. Define $\mathcal{W}_\lambda$ to be the quotient of the tensor algebra $T^{\bullet}(\W\otimes \C)$ of the complexification of $\W$ by the ideal generated by the elements $x\otimes y-y\otimes x-\lambda \langle\langle x,y \rangle \rangle 1$, where $x,y\in \W$.
Then $\mathcal{W}_\lambda$ is called the Weyl algebra in the literature.
Let $p:T^{\bullet}(\W\otimes \C)\rightarrow \mathcal{W}_\lambda $ be the quotient map. Clearly $p(T^{\bullet}(\W'))=\Sym^{\bullet}(\W') $ and $p(T^{\bullet}(\W''))=\Sym^{\bullet}(\W'') $. Let $\mathcal{I}$ be the left ideal in $\mathcal{W}_\lambda$ generated by $\W'$. Then $\mathcal{P}_\lambda:=\mathcal{W}_\lambda / \mathcal{I}$ is a $\mathcal{W}_\lambda$-module by left multiplication. The natural projection $\Sym^\bullet (\W'')\rightarrow \mathcal{W}_\lambda$ induces an isomorphism onto $\mathcal{P}_\lambda$ and we obtain an action of $\mathcal{W}_\lambda$ by left multiplication.
The infinitesmal Fock model is 
\begin{equation}
    \mathcal{P}_\lambda\cong\Sym^\bullet (\W")\cong \Pol(\W'),
\end{equation}
where the second isomorphism is induced by the non-degenerate pairing $\langle\langle,\rangle\rangle$ between $\W"$ and $\W'$.

There is an embedding of Lie algebra $j:\mathfrak{sp}(\W\otimes \C) \rightarrow \mathcal{W}_\lambda$, so $ \mathcal{P}_\lambda$ is a $\mathfrak{sp}(\W\otimes \C) $ module (see page 150 of \cite{KM90}). This induces an action $\omega_\lambda$ of $\mathfrak{sp}(\W\otimes \C) $ on $\mathcal{P}_\lambda$.

Explicitly let $\{e_1,\ldots,e_N,f_1,\ldots f_N \}$ be a symplectic  basis for $\W$ such that
\begin{equation}\label{symplecticbasis}
		    J_0 e_j=f_j \text{ and } J_0 f_j=-e_j
\end{equation}
for $1\leq j \leq N$. Define 
\[w'_j=e_j -  f_j  i\text{ and } w''_j=e_j + f_j i\]
for $1\leq j \leq N$. Then $\{w'_1,\ldots, w'_N \}$ (resp. $\{w''_1,\ldots, w''_N \} $) is a basis for $\W'$ (resp. $\W''$). Let $u_j$ be the linear functional on $\W'$ given by 
		\[u_j(w')=\langle \langle w',w''_j \rangle \rangle.\]
Then $\mathcal{P}_\lambda=\Sym^\bullet (\W'')$ can be identified with $\Pol(\W')\cong \Pol(\C^N)=\C[u_1,\ldots,u_N]$. Denote by $\rho_\lambda$ the action of $\mathcal{W}_\lambda $ on $\Pol(\C^N)$. We have (Lemma 6.1 of \cite{KM90})
\begin{lemma}\label{wjaction}
\
	\begin{enumerate}
		    \item $\rho_\lambda(w''_j)=u_j$. \
		    \item $\rho_\lambda (w'_j)=2 i \lambda \frac{\partial}{\partial u_j}$.
	\end{enumerate}
\end{lemma}
From now on we specialize to the case $\lambda= 2\pi i$ and let $\mathcal{P}=\mathcal{P}_{2 \pi i}$, $\omega=\omega_{2\pi i}$ and $\rho=\rho_\lambda$.
If we decompose $\W$ as 
\[\W=\mathbb{X}\oplus \mathbb{Y},\]
where $\mathbb{X}$ and $\mathbb{Y}$ are Lagrangian subspaces of $\W$. The Schr\"odinger model can be viewed as the set of Schwartz functions $\mathcal{S}(\mathbb{X})$ on $\mathbb{X}$. Explicitly if we assume 
\[\mathbb{X}=\Span\{e_1,\ldots,e_N\},\ \mathbb{Y}=\Span\{f_1,\ldots,f_N\},\]
then we have
\begin{equation}\label{ejfjaction}
    \rho(e_j)=\frac{\partial}{\partial x_j},\ \rho(f_j)=2\pi i x_j,
\end{equation}
where $\{x_1,\ldots,x_N\}$ are coordinate functions with respect to the basis $\{e_1,\ldots,e_N\}$. We define
\begin{equation}
    \varphi_0=\exp(-\pi \sum_{i=1}^N x_i^2).
\end{equation}
$\varphi_0$ is the unique vector in $\mathcal{S}(\mathbb{X})$ that is annihilated by $\rho(w_j')$ for all $1\leq j \leq N$. Then under the Weil representation $\varphi_0$
is fixed by $\widetilde{\Uni(N)}$ which is a maximal compact subgroup of $\Mp(2N,\R)$. 	There a unique $\mathcal{W}_\lambda$-intertwining (thus $\mathfrak{sp}(\W)$-intertwining) operator
\begin{equation}\label{eq:intertwining operator}
    \iota:\mathcal{P}\rightarrow\mathcal{S}(\mathbb{X}).
\end{equation}
The image $\iota(\mathcal{P})$ is exactly the $\widetilde{\Uni(N)}$-finite vectors in the Schrodinger model which consists of functions on $\mathbb{X}$ of the form $p({\bx})\varphi_0({\bx}) $, where $p({\bx})$ is a polynomial function on $\mathbb{X}$. 
More specifically, we know that
\begin{equation}\label{iotaof1}
    \iota(1)=\varphi_0.
\end{equation}
Using Lemma \ref{wjaction} and equation \eqref{ejfjaction} one can see immediately that (Lemma 6.3 of \cite{KM90})
\begin{lemma}
\[\iota(u_j)\iota^{-1}=\frac{\partial}{\partial x_j}-2 \pi x_j,\ \iota(-4\pi \frac{\partial}{\partial u_j})\iota^{-1}=\frac{\partial}{\partial x_j}+2 \pi x_j,\]
where $u_j$ and $\frac{\partial}{\partial u_j}$ are regarded as operators in $\mathcal{W}_\lambda$.
\end{lemma}
Since $\iota$ intertwines the $\mathcal{W}_\lambda $ action, the above lemma and \eqref{iotaof1} determine the map $\iota$ completely (see Lemma \ref{intertwiner}).

In the rest of the section by using the above framework we are going to write down coordinate functions of the Fock and the Schr\"odinger model for the dual pairs 
\begin{equation}
    (\Uni(n,n),\Uni(m,m')),\ (\Sp(2n,\R),\Orth(2m,2m')) \text{ and } (\Orth^*(2n),\Sp(m,m')).
\end{equation}
The Fock or Schr\"odinger model of the three dual pairs are the same since they share the same $(\W,\langle\langle,\rangle \rangle)$. In fact the dual pairs can be put into two seesaw dual pairs:
\begin{equation}\label{seesawone}
		    \xymatrixrowsep{0.3in}
		\xymatrixcolsep{0.3in}
		\xymatrix{ \Uni(n,n) &\Orth(2m,2m')\\
			\Sp(2n,\R)\ar[u] \ar[ur] & \Uni(m,m'), \ar[u]\ar[ul]}
\end{equation}
and
\begin{equation}\label{seesawtwo}
		    \xymatrixrowsep{0.3in}
		\xymatrixcolsep{0.3in}
		\xymatrix{ \Uni(n,n) &\Sp(m,m')\\
			\Orth^*(2n)\ar[u] \ar[ur] & \Uni(m,m'). \ar[u]\ar[ul]}
\end{equation}

\subsection{Case A: the $(\Uni(p,q),\Uni(m,m'))$ dual pair}\label{FockmodelUpq}
Let $V$ be a $p+q$ dimensional right complex vector space and $(,)$ be a non-degenerate skew Hermitian form on $V$ with signature $(p,q)$ satisfying \eqref{(,)}. Choose an orthogonal basis $\{v_1,\ldots,v_p,v_{p+1},\ldots, v_{p+q}\}$ of $V$ such that 
\begin{equation}\label{standardbasisUpq}
    (v_\alpha,v_\alpha)=-i \text{ and }  (v_\mu,v_\mu)=i
\end{equation}
for $1\leq \alpha \leq p, p+1 \leq \mu \leq p+q$ (in this subsection we keep this convention of index). 
		
Let $W$ be a $m+m'$ dimensional complex vector space with a non-degenerate  Hermitian form $\langle,\rangle$ of signature $(m,m')$ satisfying 
\[\langle hw,h'w'\rangle=h\langle w_1,w_2\rangle\overline{h'}\]
for $h,h'\in \C$ and $w,w' \in W$. We assume that choose an orthogonal basis $\{w_1,\ldots,w_{m+m'}\}$ of $W$ such that 
\begin{equation}
    \langle w_a,w_a\rangle=1 ,\ \langle w_k,w_k\rangle=-1 
\end{equation}
for $1\leq a \leq m, m+1\leq k \leq m+m'$ (in this subsection we keep this convention of index).
		
Define $\W=V\otimes_\C W$ and $\langle\langle,\rangle\rangle_h$ on $\W$ by
\begin{equation}\label{bigunitaryform}
  \langle\langle v\otimes w, \tilde{v}\otimes \tilde{w}\rangle\rangle_h=(v,\tilde{v})\langle\tilde{w},w\rangle.  
\end{equation}
One checks easily that $\langle\langle,\rangle\rangle_h$ is a skew Hermitian form that is anti-linear in the first variable and linear in the second variable. Define $\langle\langle,\rangle\rangle=\mathrm{Re}\langle\langle,\rangle\rangle_h$, then $\langle\langle,\rangle\rangle$ is a symplectic form on the underlying real vector space of $\W$.
		
Define $J_0=i I_{p,q}\otimes I_{m,m'}$, where $I_{a,b}$ is the matrix
\[\left(\begin{array}{cc}
		    I_a &  0\\
		    0 & -I_b
\end{array}\right).\]
Then $J_0$ is a positive definite complex structure for the symplectic form $\langle \langle,\rangle \rangle$.
		
Now define $\W_\C=\W\otimes_\R \C  \cong V\otimes_\C(W\otimes_\R \C)$. Denote the new complex structure by right multiplication by $i$. Define 
\begin{align}\label{wprime}
    w'_a=w_a-i w_a i, \ & w''_a=w_a+i w_a i,\\
    w'_k=w_k+i w_k i, \ & w''_k=w_k-i w_k i. \nonumber
\end{align}
$\W_\C=\W'\oplus \W''$, where $\W'$($\W''$ resp.) is the $+i$($-i$ resp.) eigenspace of $J_0$. Then we have
\begin{align*}
    &\W'=\Span_\C\{v_\alpha \otimes w'_a, v_\mu \otimes w''_a, v_\alpha \otimes w'_k, v_\mu \otimes w''_k \}, \\
    &\W''=\Span_\C\{v_\alpha \otimes w''_a, v_\mu \otimes w'_a, v_\alpha \otimes w''_k, v_\mu \otimes w'_k \}.
\end{align*}
	
Define linear functionals on $\W'$:
\begin{align}
    u_{\alpha a}^+(x)=\langle\langle x,v_\alpha \otimes w''_a\rangle\rangle,\  &
    u_{\mu a}^+(x)=\langle\langle x, v_\mu \otimes w'_a \rangle\rangle ,\\ \nonumber
    u_{\alpha k}^-(x)=\langle\langle x,v_\alpha \otimes w''_k\rangle\rangle,\  &
    u_{\mu k}^-(x)=\langle\langle x, v_\mu \otimes w'_k\rangle\rangle
\end{align}
for $x\in \W'$. We can now identify $\Sym^\bullet(\W'')\cong \Pol(\W')$ with the space of polynomials in complex variables $\{u_{ i a}^+,u_{i k}^-\mid 1\leq i \leq p+q,1\leq a \leq m,m+1\leq k\leq m+m'\}$ and this will be the Fock model $\mathcal{P}_\lambda$.

Now we assume $m=m'$. Define
\[e_a=\frac{1}{\sqrt{2}}(w_a-w_{a+m}),\ f_a=\frac{1}{\sqrt{2}}(w_a+w_{a+m})\]
for $1\leq a \leq m$. Then $E=\Span\{e_1,\ldots ,e_m\}$
is a Lagrangian subspace of $W$.
The Schr\"odinger model of the Weil representation is given by the space of Schwartz functions on $\mathcal{S}(V\otimes_\C E)\cong\mathcal{S}(V^m)$ on $V^m$. We use complex coordinates ${\bf z}=(z_1,\ldots,z_m)$
with $z_j= (z_{1 j},\ldots,z_{p+q, j})^t $,
where $z_{kj}=x_{kj}+i y_{kj} $ ($1\leq k \leq p+q$) is the coordinate function of the $j$-th copy of $V$ with respect to the basis $\{v_1\otimes e_j,\ldots,v_{p+q}\otimes e_j\}$.

The Weil representation of $\mathfrak{sp}(\W,\langle\langle,\rangle\rangle)$ now arises from the action of Weyl algebra $\mathcal{W}_\lambda$. Using Lemma 2.2 of \cite{kudlanotes}, it is easy to derive the following formulas.
\begin{align}\label{weylalgebraaction}
	\rho_\lambda(v_\alpha\otimes w'_a)=& \frac{1}{\sqrt{2}}(-\lambda i \bar{z}_{\alpha a}+2 \frac{\partial}{\partial z_{\alpha a}}), & \rho_\lambda(v_\mu \otimes w''_a)=& \frac{1}{\sqrt{2}}(-\lambda i z_{\mu a}+2\frac{\partial}{\partial \bar{z}_{\mu a}}),\\ \nonumber
	\rho_\lambda(v_\alpha\otimes w'_{a+m})=& \frac{1}{\sqrt{2}}(\lambda i z_{\alpha a}-2\frac{\partial}{\partial \bar{z}_{\alpha a}}), & \rho_\lambda(v_\mu \otimes w''_{a+m})=& \frac{1}{\sqrt{2}}(\lambda i \bar{z}_{\mu a}-2\frac{\partial}{\partial z_{\mu a}}),\\ \nonumber
	\rho_\lambda(v_\alpha \otimes w''_a)=&\frac{1}{\sqrt{2}}(\lambda i z_{\alpha a}+2\frac{\partial}{\partial \bar{z}_{\alpha a}}), & \rho_\lambda(v_\mu \otimes w'_a)=& \frac{1}{\sqrt{2}}(\lambda i \bar{z}_{\mu a}+2\frac{\partial}{\partial z_{\mu a}}), \\ \nonumber
	\rho_\lambda(v_\alpha \otimes w''_{a+m})=& \frac{1}{\sqrt{2}}(-\lambda i \bar{z}_{\alpha a}-2 \frac{\partial}{\partial z_{\alpha a}}) , & \rho_\lambda(v_\mu \otimes w'_{a+m})=&\frac{1}{\sqrt{2}}(-\lambda i z_{\mu a}-2\frac{\partial}{\partial \bar{z}_{\mu a}}),
\end{align}
where $1\leq \alpha \leq p,p+1 \leq \mu \leq p+q$, $1\leq a \leq m $ and 
\[\frac{\partial}{\partial z_{j k}}=\frac{1}{2}(\frac{\partial}{\partial x_{j k}}-\frac{\partial}{\partial y_{j k}}i),\ \frac{\partial}{\partial \bar{z}_{j k}}=\frac{1}{2}(\frac{\partial}{\partial x_{j k}}+\frac{\partial}{\partial y_{j k}}i).\]
If we fix the parameter $\lambda=2 \pi i$, then we have
\[\varphi_0({\bf z})=\exp(-\pi \sum_{k=1}^{p+q} \sum_{a=1}^m |z_{k a}|^2).\]
$\iota:\mathcal{P}_\lambda\rightarrow \mathcal{S}(V^m)$ maps $1\in \mathcal{P}_\lambda$ to $\varphi_0$. From \eqref{weylalgebraaction} we have the following lemma.
\begin{lemma}\label{intertwiner}
	\begin{align*}
		    \iota ( u_{\alpha a}^+ )\iota^{-1}=&\frac{1}{\sqrt{2}}(-2\pi z_{\alpha a}+2\frac{\partial}{\partial \bar{z}_{\alpha a}}), & \iota ( u_{\mu a}^+)\iota^{-1}=& \frac{1}{\sqrt{2}}(-2\pi \bar{z}_{\mu a}+2\frac{\partial}{\partial z_{\mu a}}), \\
		    \iota ( u_{\alpha k}^- )\iota^{-1}=& \frac{1}{\sqrt{2}}(2\pi \bar{z}_{\alpha a}-2 \frac{\partial}{\partial z_{\alpha a}}) , & \iota ( u_{\mu k}^-)\iota^{-1}=&\frac{1}{\sqrt{2}}(2\pi z_{\mu a}-2\frac{\partial}{\partial \bar{z}_{\mu a}})
	\end{align*}
for $1\leq \alpha \leq p,p+1 \leq \mu \leq p+q$, $1\leq a \leq m $, $k=a+m$.
In particular, if $p$ is a monomial in the variables $\{u_{\alpha a}^+,u_{\mu a}^+,u_{\alpha k}^-,u_{\mu k}^-\mid 1\leq \alpha \leq p,p+1 \leq \mu \leq p+q,1\leq a \leq m,m+1\leq k \leq 2m\}$:
\[p=\prod_{\alpha=1}^p \prod_{\mu=p+1}^{p+q} \prod_{a=1}^m \prod_{k=m+1}^{2m} (u_{\alpha a}^+)^{d_{\alpha a}} (u_{\mu a}^+)^{d_{\mu a}} (u_{\alpha k}^-)^{d_{\alpha k}} (u_{\mu k}^-)^{d_{\mu k}},\]
then we have 
	\begin{align*}
		    \iota(p)=&\prod_{\alpha=1}^p \prod_{\mu=p+1}^{p+q} \prod_{a=1}^m \prod_{k=m+1}^{2m} (\sqrt{2})^{d_{\alpha a}+d_{\mu a}}(-\sqrt{2})^{d_{\alpha k}+d_{\mu k}} \\
		    &\cdot (\frac{\partial}{\partial \bar{z}_{\alpha a}}-\pi z_{\alpha a})^{d_{\alpha a}} (\frac{\partial}{\partial z_{\mu a}}-\pi \bar{z}_{\mu a})^{d_{\mu a}}(\frac{\partial}{\partial z_{\alpha a}}-\pi \bar{z}_{\alpha a})^{d_{\alpha k}}(\frac{\partial}{\partial \bar{z}_{\mu a}}-\pi z_{\mu a})^{d_{\mu a}} \varphi_0.
	\end{align*}
\end{lemma}	
We will need the following lemma.
\begin{lemma}\label{highesttermofpolynomial}
Suppose $p$ is the same as in the previous lemma.
\[\iota(p)=\tilde{p}\varphi_0,\]
where $\tilde{p}$ is a polynomials of the variables $\{z_{\alpha a},z_{\mu a},\bar{z}_{\alpha a},\bar{z}_{\mu a}\mid 1 \leq \alpha \leq p,p+1 \leq \mu \leq p+q,1\leq a \leq m\}$ whose unique highest degree term (in every variable) is 
\[\prod_{\alpha=1}^p \prod_{\mu=p+1}^{p+q} \prod_{a=1}^m \prod_{k=m+1}^{2m} (-2\sqrt{2}\pi z_{\alpha a})^{d_{\alpha a}} (-2\sqrt{2}\pi \bar{z}_{\mu a})^{d_{\mu a}} (2\sqrt{2}\pi \bar{z}_{\alpha, k-m})^{d_{\alpha k}} (2 \sqrt{2}\pi z_{\mu, k-m})^{d_{\mu k}}.\]
\end{lemma}
\begin{proof}
Since
\[\varphi_0=\prod_{\alpha=1}^{p+q} \prod_{a=1}^m \exp(-\pi |z_{\alpha a}|^2)\]
and the operators $\{\iota(u_{\alpha a}^+),\iota(u_{\mu a}^+),\iota(u_{\alpha k}^-),\iota(u_{\mu k}^-)\mid 1\leq \alpha \leq p,p+1 \leq \mu \leq p+q,1\leq a \leq m,m+1\leq k \leq 2m\}$ commute with each other, it suffices to prove the lemma for the case of one variable. That is 
\[(\frac{\partial}{\partial \bar{z}}-\pi z)^{d_1} (\frac{\partial }{\partial z}-\pi \bar{z})^{d_2}\cdot \exp(-\pi |z|^2)=\tilde{p} \exp(-\pi |z|^2) ,\]
where the highest degree term of $ \tilde{p}$ is $(-2\pi z)^{d_1} (-2\pi \bar{z})^{d_2}$. But this follows from an easy induction on the bi-dgree $(d_1,d_2)$.
\end{proof}

\subsection{The Dual Pair $(\Sp(2n,\R),\Orth(2r,2s))$}\label{twoWforcaseB}
The purpose of this subsection is to explain the relation between two different constructions of the fundamental module $\W$ of the dual pair $(\Sp(2n,\R),\Orth(2r,2s))$. One of them comes directly from our global construction of the algebraic group in Section \ref{locallysymmetricspace}. The other is what we actually use when studying (the relative Lie algebra cohomology of) the Weil representation. 

In this subsection let $B$ be $M_2(\R)$ and $V_B$ be a free right $B$ module of rank $n$. Let $\sigma$ be the anti-involution of $M_2(\R)$ defined by 
		\[\sigma(x)=J x^t J^{-1},\]
where $J=\left(\begin{array}{cc}
		   0  & -1 \\
		   1  & 0
\end{array}\right)$. Let $(,)_B$ be a Hermitian form on $V_B$ satisfying \eqref{(,)}.
Define $G$ to be the isometry group of $(V_B,(,)_B)$.
Let $e=e_{11}$, then $V_B e$ is a $2n$ dimensional real vector space. Recall from Section \ref{locallysymmetricspaceSp} that we can define a symplectic form $(,)_1$ on $V_B e$ by
\[(v e, v'e)_1 e_{21}=(v e,v' e)_B.\]
This implies that $G\cong \Sp(2n,\R)$.
Let $W_B$ be a free left $B$ module of rank $m$ and $\langle,\rangle_B$ be a skew Hermitian form on $W_B$ satisfying
		\[\langle b w, \tilde{b}\tilde{w}\rangle_B=b \langle w,\tilde{w}\rangle_B \tilde{b}^\sigma.\]
Define $G'$ to be the isometry group of $(W_B,\langle,\rangle_B)$.
$eW_B$ is a $2m$ dimensional real vector space. We can define a symmetric form $\langle,\rangle_\R$ on $eW_B$ by
		\[\langle e w,e w'\rangle_\R e_{12}=\langle e w, e w'\rangle_B.\]
This implies that $G'\cong \Orth(2r,2s)$.
		
Now let $\W=V_B\otimes_B W_B$. Then we have $\W\cong V_B e\otimes_\R e W_B$ as a real vector space and as a $\GL_B(V_B)\times \GL_B(W_B)$ module. In fact if one think of $v\in V_B$ as a $2n$ by $2$ matrix and $w\in W_B$ as a $2$ by $2m$ matrix, the tensor product $v\otimes_B w$ is just the matrix multiplication $vw$. Obviously one get the same space by tensoring $V_B e$ (the set of $2n$ by $1$ matrices) with $eW_B$ (the set of $1$ by $2m$ matrices). Moreover, $\GL_B(V_B)\cong \GL_{2n}(\R)$ (resp. $\GL_B(W_B)\cong \GL_{2m}(\R)$) acts by left (resp. right) multiplication on $\W$. 
		
One can define $\langle\langle,\rangle\rangle$ on $\W$ by 
		\[\langle\langle v\otimes w, \tilde{v}\otimes \tilde{w}\rangle\rangle=\mathrm{tr}_\R[(v,\tilde{v})_B\langle\tilde{w},w\rangle_B].\]
Then $\langle\langle,\rangle\rangle_\R$ is symplectic. One can also define $\langle\langle,\rangle\rangle_\R$ on $\W$ (regarded as $ V_B e\otimes_\R e W_B $) by 
\[\langle\langle v e\otimes e w, \tilde{v} e\otimes e \tilde{w}\rangle\rangle_\R=(v e,\tilde{v} e)_1 \langle e \tilde{w},e w\rangle_\R.\]
We have the following interesting fact
\begin{lemma}\label{twoformsareequal}
There exists a nonzero constant $c\in \R$ such that
		\[\langle\langle,\rangle \rangle=c\langle\langle,\rangle \rangle_\R\]
\end{lemma}
\begin{proof}
We have the following fact (for example see equation (19), page 121 of \cite{fulton1997young} or verify directly)
\[\wedge^2(V_B e\otimes e W_B)\cong [\wedge^2(V_B e)\otimes \Sym^2(e W_B)]\oplus [\Sym^2(V_B e) \otimes \wedge^2(e W_B )].\]
Both $ \langle\langle,\rangle \rangle$ and $\langle\langle,\rangle \rangle_\R$ are non-degenerate, skew symmetric and invariant under $G\times G'$. By classical invariant theory the invariant subspace of $\wedge^2(V_B e)\otimes \Sym^2(e W_B)$ under $G\times G'$ is one dimensional and the invariant subspace in  $\Sym^2(V_B e) \otimes \wedge^2(e W_B)$ is trivial. This proves that the two forms are the same up to a constant multiple. 
\end{proof}

\subsection{Case B: the seesaw dual pairs \ref{seesawone}}\label{seesawonesection}
Let $V_0$ be a $2n$ dimensional real vector space with a non-degenerate skew symmetric form $(,)_1$. Let $V=V_0\otimes \C$ and we extend $(,)_1$ from $V_0$ to $V$  anti-linearly in the first variable and linearly in the second variable. Denote the resulting skew Hermitian form by $(,)$. The Hermitian form $i(),$ has signature $(n,n)$.
In fact let $E_1,\ldots,E_n,F_1,\ldots,F_n$ be a symplectic basis of $(V_0,(,)_1)$. Define
\begin{enumerate}
            \item $v_\alpha=\frac{1}{\sqrt{2}}(E_\alpha-i F_\alpha )$ for $1\leq \alpha \leq n$,\
            \item $v_\mu=\frac{1}{\sqrt{2}}(E_{\mu-n}+i F_{\mu-n})$ for $n+1\leq \mu \leq 2n$.
\end{enumerate}
Then $\{v_1,\ldots,v_{2n}\}$ is an orthogonal basis of $(V,(,))$ such that 
\[(v_\alpha,v_\alpha)=-i,\ (v_\mu,v_\mu)=i \]
for $1\leq \alpha \leq n, n+1 \leq \mu \leq 2n$.
		
Let $W$ be a $m+m'$ dimensional complex vector space with a Hermitian form $\langle,\rangle$ of signature $(m,m')$ which is linear in the first variable and anti-linear in the second variable. And define $\langle,\rangle_\R=\mathrm{Re}\langle,\rangle$. Then $\langle,\rangle_\R$ is a symmetric form of signature $(2m,2m')$. We have 
\begin{equation}\label{V_0^2r=V^r}
    V^{m+m'}\cong V\otimes_\C W=(V_0\otimes_\R\C)\otimes_\C W\cong V_0\otimes_\R W \cong V_0^{2m+2m'}.
\end{equation}
Recall that we can define $\W=V\otimes_\C W$ and a skew Hermitian form $\langle\langle,\rangle\rangle_h$ on $\W$ by \eqref{bigunitaryform}.
Also define a symplectic form $\langle\langle,\rangle\rangle_\R$ on $\W$ (regarded as $V_0\otimes_\R W$) by
\[\langle\langle v\otimes w, \tilde{v}\otimes \tilde{w}\rangle\rangle=(v,\tilde{v})_1\langle\tilde{w},w\rangle_\R.\]
It is easy to check directly that
\begin{lemma}
\[\mathrm{Re}\langle\langle,\rangle\rangle_h=\langle\langle,\rangle\rangle.\]
\end{lemma}
\begin{remark}
The above lemma shows that the seesaw dual pairs in \ref{seesawone} share the same underlying symplectic module $(\W,\langle\langle,\rangle\rangle)$. Thus they give rise to the same Weil representation. Thus the Fock or Schr\"odinger model of $(\Uni(n,n),\Uni(m,m'))$ can serve as the Fock or Schr\"odinger model of $(\Sp(2n,\R),\Orth(2m,2m'))$ as well.
\end{remark}
\begin{remark}
Combine the above remark with the discussion in Section \ref{twoWforcaseB}, the Schr\"odinger model of $(\Sp(2n,\R),\Orth(2m,2m))$ is 
\[\mathcal{S}(V_0^{2m})\cong \mathcal{S}(V^{m}) \cong \mathcal{S}((M_2(\R)^{n})^m). \]
\end{remark}

\subsection{Case C: the seesaw dual pair \ref{seesawtwo}}\label{seesawtwosection}
Let $V$ be a $n$-dimensional right $\H$ vector space with skew Hermitian form $(,)_2$ satisfying
\begin{equation}\label{(,)_2}
		    (v b, \tilde{v} \tilde{b})_2=b^\sigma (v,\tilde{v})_2 \tilde{b}
\end{equation}
for $v,\tilde{v}\in V$ and $b,\tilde{b}\in \H$.
Define a complex skew Hermitian form $(,)$ on $V$ by 
		\[(v,\tilde{v})=a+bi \text{ if } (v,\tilde{v})_2=a+bi+cj+dk.\]
Then the Hermitian form $i(,)$ has signature $(n,n)$. In fact choose an $\H$-basis $\{v_1,\ldots,v_n\}$ of $V$ such that 
\[(v_\alpha,v_\beta)_2=-i \delta_{\alpha \beta}, \ 1\leq \alpha, \beta \leq n.\]
Define 
\[v_\mu := v_{\mu-n} j, \ n+1 \leq \mu \leq 2n.\]
Then $\{v_1,\ldots,v_n,v_{n+1},\ldots,v_{2n}\}$ is an orthogonal basis of $(V,(,))$ such that 
\[(v_\alpha,v_\alpha)=-i \text{ and }  (v_\mu,v_\mu)=i\]
for $1\leq \alpha \leq n, n+1 \leq \mu \leq 2n$.

Let $W$ be a $m+m'$ dimensional left $\C$ vector space with non-degenerate Hermitian form $\langle,\rangle$ of signature $(m,m')$ that is complex linear in the first variable and anti-linear in the second variable. Define $W_\H=\H\otimes_\C W$, extend $\langle,\rangle$ to a form on $W_\H$ denoted as $\langle,\rangle_\H$ satisfying  
\[\langle h v,\tilde{h}\tilde{v}\rangle_\H=h\langle v,\tilde{v}\rangle_\H {\tilde{h}}^\sigma\]
for $h,\tilde{h}\in \H$. Then we have a canonical isomorphism 
\[V\otimes_\H W_\H\cong V\otimes_\C W.\]

Let $\W=V\otimes_\H W_\H$. Define $\langle\langle,\rangle\rangle$ on $\W$ by
\[\langle\langle v\otimes w, \tilde{v}\otimes \tilde{w} \rangle\rangle=\mathrm{Re}[(v,\tilde{v})_2 \langle\tilde{w},w\rangle_\H].\]
$\langle\langle,\rangle\rangle$ is well-defined on $\W$ and is a symplectic form. Also define $\langle\langle,\rangle\rangle_h$ on $\W$ (regarded as $V\otimes_\C W$) by \eqref{bigunitaryform}.
Then we have 
\begin{lemma}
\[\langle\langle,\rangle\rangle=\mathrm{Re}\langle\langle,\rangle\rangle_h.\]
\end{lemma}
		
\begin{remark}
	The above lemma shows that the seesaw dual pairs in \ref{seesawone} share the same underlying symplectic module $(\W,\langle\langle,\rangle\rangle)$. Thus they give rise to the same Weil representation. Thus the Fock or Schr\"odinger model of $(\Uni(n,n),\Uni(m,m'))$ can serve as the Fock or Schr\"odinger model of $(\Orth^*(2n),\Sp(m,m'))$ as well.
\end{remark}

\section{Special Schwartz classes in the relative Lie Algebra cohomology of the Weil representation}\label{Andersoncocyle}
In this section we review the construction of holomorphic differential forms in \cite{Anderson}. We use this result to construct the special canonical class $\varphi$ as in Theorem \ref{thmA0}. We will prove that $\varphi$  is closed. In this section, $G$ denotes a real Lie group.

Let $\mathfrak{g}_0$ be the Lie algebra of $G$ and $\mathfrak{g}$ be its complexification. Fix a maximal compact subgroup $z_0=K$ of $G$ and the corresponding Cartan decomposition $\mathfrak{g}_0=\mathfrak{k}_0+\mathfrak{p}_0$. Identify $T_{z_0} D$ with $\mathfrak{p}_0$, where $D=G/K$.  By parallel translating the trace form on $\mathfrak{p}_0$ by the group $G$, we endow $D$ with a Riemannian metric denoteed by $\tau$. We assume that $D$ is Hermitian symmetric. Decompose $\mathfrak{p}=\mathfrak{p}_0\otimes\C$ into holomorphic and anti-holomorphic tangent vectors
\[\mathfrak{p}=\mathfrak{p}^{+}+\mathfrak{p}^{-}.\]

For a  $(\mathfrak{g},K)$ module $\mathcal{M}$, define
\[C^{\bullet}(\mathfrak{g},K;\mathcal{M})=\Hom_{K}(\wedge^\bullet(\mathfrak{g}/\mathfrak{k}),\mathcal{M})\cong \Hom_K (\wedge^\bullet\mathfrak{p},\mathcal{M})\cong ( \wedge^\bullet \mathfrak{p}^* \otimes \mathcal{M})^K.\]
It is a cochain complex and gives rise to the relative Lie algebra cohomology $H^\bullet (\mathfrak{g},K;\mathcal{M})$ (see \cite{BW}) .
In his thesis \cite{Anderson}, Anderson constructed cochains $\phi^+$ in $\Hom_{K}(\wedge^\bullet \mathfrak{p}_+,\mathcal{P}_-^{\mathfrak{p}_-})$, where $\mathcal{P}_-$ is the Fock  model of a certain Weil representation. Here, the notation $\mathcal{P}_-^{\mathfrak{p}_-}$ denotes the subspace of $\mathcal{P}_-$ annihilated by $\mathfrak{p}_-$. We can  construct a mirror element $\phi^-\in \Hom_{K}(\wedge^\bullet \mathfrak{p}_-,\mathcal{P}_+^{\mathfrak{p}_+})$. Then we take $\phi$ to be the out wedge product of $\phi^+$ and $\phi^-$ and $\varphi$ to be $\iota(\phi)$ where $\iota$ is as in \eqref{eq:intertwining operator}.
We now carry out this construction case by case and prove some lemmas (Lemma \ref{onlyonetermUpq}, Lemma \ref{onlyonetermSp} and Lemma \ref{onlyonetermOstar}) along the way.

\subsection{Case A}\label{varphiUni}
We follow the assumptions and notations of Section \ref{FockmodelUpq}.
Recall that $V$ is a $p+q$ dimensional complex vector space with a skew Hermitian form $(,)$ such that $i(,)$ has signature $(p,q)$. Let $\bar{V}$ be the conjugate complex vector space of $V$: $\bar{V}$ has the same underlying abelian group as $V$  but the complex multiplication of $\bar{V}$ is the conjugate of $V$.

Let $V^*=\Hom_\C(V,\C)$. There is a complex linear isomorphism $\bar{V}\rightarrow V^*$ given by
\[v\mapsto (v,\cdot).\]
One can also identify $V\otimes_\C \bar{V}$ with $\Hom_\C(V,V)$ by the map
\[v\otimes {\tilde{v}} \mapsto v(\tilde{v},\cdot).\]
Let $\Sym(V\otimes_\C  \bar{V})$  be the symmetric tensor inside $V\otimes_\C \bar{V}$ (this makes sense since $V$ and $\bar{V}$ have the same underlying Abelian group). By the above identification, $\Sym(V\otimes_\C  \bar{V})$ acts on $V$ by 
\[(v\circ \tilde{v}) (x)=v(\tilde{v},x)+\tilde{v}(v,x),\forall x\in V,\]
where $v\circ \tilde{v}=v\otimes \tilde{v}+\tilde{v} \otimes v$. One can check that this action satisfies
\[((v\circ \tilde{v})(x),y)+(x, (v\circ \tilde{v})(y))=0.\]
In fact by \cite[Lemma 3.6]{BMM2}, we have
\begin{equation}
    \Sym(V\otimes_\C  \bar{V})\cong \mathfrak{u}(V,(,))=\mathfrak{u}(p,q).
\end{equation}
Define $z_0=\Span_\C \{v_1,\ldots,v_p\}\in D$. Its stabilizer is $K=\rU(z_0)\times \rU(z_0^\bot)\cong \rU(p)\times \rU(q)$.
We have a corresponding Cartan decomposition of $\mathfrak{g}_0=\mathfrak{u}(V,(,))$:
\[\mathfrak{g}_0=\mathfrak{k}_0+\mathfrak{p}_0,\]
where 
\[\mathfrak{k}_0=(\Hom(z_0^\bot,z_0^\bot)\oplus \Hom(z_0,z_0)) \cap \mathfrak{g}_0,\ \mathfrak{p}_0= (\Hom(z_0^\bot,z_0)\oplus \Hom(z_0,z_0^\bot))\cap \mathfrak{g}_0.\]
More explicitly define 
\begin{equation}\label{LiealgebrabasisofUpq}
		    E_{m n}=v_m\circ v_n \text{ and } F_{m n}=i v_m\circ v_n
\end{equation}
for $1\leq m,n \leq p+q$. Then we have
\begin{enumerate}
		    \item $ \mathfrak{k}_0=\Span_\R\{E_{\alpha \beta}, F_{\alpha \beta},E_{\mu \nu},F_{\mu \nu}\}$, \
		    \item $ \mathfrak{p}_0=\Span_\R\{E_{\alpha \mu},F_{\alpha \mu}\}$,
\end{enumerate}
where $1 \leq \alpha,\beta \leq p$ and $p+1 \leq \mu,\nu \leq p+q$ (in this subsection we keep this convention of index). In terms of matrices we have
\[\mathfrak{k}_0=\left\{\left(\begin{array}{cc}
        A & 0 \\
        0 & B
\end{array}\right)\mid A+A^*=0, B+B^*=0
\right\},
\mathfrak{p}_0=\left\{\left(\begin{array}{cc}
        0 & A \\
        A^* & 0
\end{array}\right)\mid A\in M_{p\times q}(\C)\right\}.\]
		
We now describe the $\mathrm{Ad}(K)$-invariant almost complex structure $J_\mathfrak{p}$ acting on $\mathfrak{p}$ that induces the structure of Hermitian symmetric domain on $D$ (c.f. page 14 of \cite{BMM2}). Let $\zeta=e^{\frac{\pi i}{4}}$. Define $a(\zeta)$ by 
\begin{equation}\label{azeta}
		    a(\zeta)|_{z_0}=  \zeta \cdot \mathrm{Id}_{z_0} \text{ and } a(\zeta)|_{z_0^\bot}=  \zeta \cdot \mathrm{Id}_{z_0^\bot}.
\end{equation}
Now we define 
\begin{equation}\label{J_p}
  J_\mathfrak{p}=\mathrm{Ad}(a(\zeta)).  
\end{equation}
It is easy to check under the identification of $V\otimes_\C \bar{V}\cong \Hom_\C(V,V)$ we have
\[\mathrm{Ad}(a(\zeta))(v\otimes \tilde{v})=(a(\zeta)v)\otimes (a(\zeta)\tilde{v}).\]
This implies that
\[J_\mathfrak{p}(E_{\alpha \mu})=F_{\alpha \mu},\ J_\mathfrak{p}(F_{\alpha \mu})=-E_{\alpha \mu}.\]
Define 
\[X_{\alpha \mu}=E_{\alpha \mu} -i F_{\alpha \mu}=2 e_{\alpha \mu}, \ Y_{\alpha \mu}=E_{\alpha \mu} +i F_{\alpha \mu}=2 e_{\mu \alpha},\]
where $e_{ab}$ is the matrix whose $(a,b)$-th entry is $1$ and all other entries are zero.
Let $\mathfrak{p}_+$(resp. $\mathfrak{p}_-$) be the $+i$ (resp. $-i$) eigenspace of $J_\mathfrak{p}$. We then have
\[\mathfrak{p}_+=\Span_\C\{X_{\alpha \mu}\mid 1\leq \alpha \leq p ,p+1\leq \mu \leq p+q\}, \]
\[ \mathfrak{p}_-=\Span_\C\{Y_{\alpha \mu}\mid 1\leq \alpha \leq p,p+1\leq \mu \leq p+q\}.\]
We also let $\{\xi'_{\alpha \mu}\mid 1\leq \alpha \leq p,p+1\leq \mu \leq p+q\}$ (resp. $\{\xi''_{\alpha \mu}\mid 1\leq \alpha \leq p,p+1\leq \mu \leq p+q\}$) be the basis of $\mathfrak{p}_+^*$ (resp. $\mathfrak{p}_-^*$ ) that is dual to $\{X_{\alpha \mu}\}$ (resp. $\{Y_{\alpha \mu}\}$).
	
Now fix $1\leq r \leq p, 1\leq s \leq q$.
\begin{lemma}\label{lem:K acts on U}
$K$ acts transitively on the set of subspaces $U\subset V$ such that  $U$ has signature $(r,s)$ and $z_0\subset D(U)\times D(U^\bot)$.
\end{lemma}
\begin{proof}
Recall that the condition that $z_0\subset D(U)\times D(U^\bot)$ is equivalent to the condition
\begin{equation}\label{eq:z_0 in D(U)times D(Ubot)}
    z_0=(z_0\cap U) \oplus (z_0\cap U^\bot).
\end{equation}
Let $z_1$ be the perpendicular of $z_0\cap U$ in $U$. Since $U$ has signature $(r,s)$ and $U^\bot$ has signature $(p-r,q-s)$, \eqref{eq:z_0 in D(U)times D(Ubot)} implies that $\mathrm{dim}_\C(z_0\cap U)=s$ and $\mathrm{dim}_\C(z_0\cap U^\bot)=q-s$. It follows that $\mathrm{dim}_\C(z_1)=r$. A simple dimension count shows 
\[U=(z_0\cap U )\oplus (z_0^\bot \cap U).\]
It follows that $K=\rU(z_0)\times \rU(z_0^\bot)$ acts transitively on all such $U$.
\end{proof}

Fix a $U=U^++U^-$ satisfying the conditions in Lemma \ref{lem:K acts on U} where
\[U^+=\Span_\C\{v_1,\ldots,v_r\},\ U^-=\Span \{v_{p+1},\ldots,v_{p+s}\}, \]
and let $z'_0:=U^-\in D(U)$. Recall from equation \eqref{Dlambda0} that 
\[D_{U,z'_0}=\{z\in D\mid U^- \subset z \subset U^+\}\cong D(U^\bot),\]
where $D_{U,z'_0}$ is the generalized special sub symmetric space in Definition \ref{specialsubsymmetricspace}.
Identify $T_{z_0}D$ with $\mathfrak{p}_0$, then
\[T_{z_0} D_{U,z'_0}=\Span_{\R}\{E_{\alpha \mu},F_{\alpha \mu}\mid r+1 \leq \alpha \leq p, p+s+1 \leq \mu \leq p+q\}.\]

Recall that we define a fiber bundle  $\pi:D\rightarrow  D_{U,z'_0}$ in Subsection \ref{fibrationpi}. The tangent space of the fiber $F_{z_0} D_{U,z'_0}=\pi^{-1}(z_0)$ at $z_0$ can be described as
\begin{equation}\label{normalfibreUpq}
		    N_{z_0} {D_{U,z'_0}}=T_{z_0} D_{U,z'_0}^\bot=\Span_\R\{E_{\alpha \mu},F_{\alpha \mu}\mid(\alpha,\mu)\in I\},
\end{equation}
where $I$ is the index set 
\begin{equation}\label{indexI}
		    I=\{(\alpha,\mu)\mid1\leq \alpha \leq r,p+1 \leq \mu \leq p+q\}\cup \{(\alpha,\mu)\mid r+1 \leq \alpha \leq p, p+1\leq \mu\leq p+s\}.
\end{equation}
We also have
\[N_{z_0}^+ {D_{U,z'_0}}=\{X_{\alpha \mu}\mid (\alpha,\mu)\in I\}.\]

Next let $\mathcal{P}_-$ be the infinitesimal Fock model defined in Section \ref{FockmodelUpq} for the dual pair $(\Uni(p,q),\Uni(0,r+s))$. Recall that $\mathcal{P}_-$ is the polynomial space in the variables $\{u_{ik}^-\mid 1\leq i \leq p+q, 1\leq k \leq r+s\}$.
We now define polynomials $f_{1}^-,f_{2}^-\in \mathcal{P}_-$ by 
\begin{definition}\label{specialharmonics}
		\[f_1^-=\mathrm{det}\left(\begin{array}{cccc}
		    u_{1 1}^- & u_{1 2}^- &\ldots & u_{1 r}^- \\
		    \ldots &  \ldots & \ldots & \ldots \\
		    u_{r 1}^- & u_{r 2}^- & \ldots &  u_{r r}^-
		\end{array}
		\right),
		f_{2}^-=\mathrm{det}\left(\begin{array}{cccc}
		    u_{p+1 \ r+1}^- & u_{p+1 \ r+2}^- &\ldots & u_{p+1\ r+s}^- \\
		    \ldots &  \ldots & \ldots & \ldots \\
		    u_{p+s \ r+1}^- & u_{p+s \ r+2}^- & \ldots &  u_{p+s\  r+s}^-
		\end{array}\right).\]
\end{definition}
		
We define an element of $(\wedge^\bullet \mathfrak{p}_+\otimes\mathcal{P}_-^{\mathfrak{p}_-})^{\mathrm{SU}(p,q)}$ following the construction of \cite{Anderson}. To be more precise, let  $\tilde{K}$ be  the preimage of $K$ under the map $\Mp(\mathbb{W})  \rightarrow \Sp(\mathbb{W})$ and $\tilde{K}^0$ be the identity component of $\tilde{K}$. Then $\tilde{K}$ is the $\mathrm{det}^{-\frac{r+s}{2}}$-cover of $K$ (c.f. \cite[Section 1.2]{Paul}):
\[\tilde{K}\cong\{(g,z)\in K\times \C^\times\mid z^2=\mathrm{det}(g)^{-\frac{r+s}{2}} \}.\]
Define
\[e_{D_{U,z'_0}}=\bigwedge_{(\alpha,\mu)\in I} X_{\alpha \mu}.\]
Also define 
\begin{equation}\label{fD_UUpq}
f_{D_{U,z'_0}}=(f_1^-)^{q-s} (f_2^-)^{p-r}.    
\end{equation}
The polynomials $f_1^-$, $f_2^-$ and $f_{D_{U,z'_0}}$ are special cases of the harmonic polynomials studied in \cite{KV}, hence are automatically annihilated by $\mathfrak{p}_-$, see the comment after \cite[(3.14)]{Anderson}.
It can be shown that $\mathfrak{k}$ acts on $\mathcal{P}_-$ by (c.f. equation (3.5) of \cite{Anderson})
\[\omega (e_{\alpha \beta})= \sum_{k=1}^{r+s}u_{\alpha k}^- \frac{\partial}{\partial u_{\beta k}^-}+\frac{1}{2}\delta_{\alpha \beta}(r+s),
\ \omega (e_{\mu \nu})=- \sum_{k=1}^{r+s}u_{\nu k}^- \frac{\partial}{\partial u_{\mu k}^-}-\frac{1}{2}\delta_{\mu \nu}(r+s).\]
The adjoint action of $\mathfrak{k}$ on $\mathfrak{p}_+$ induces an action on $\wedge^\bullet \mathfrak{p}_+$. Define 
\[\mathfrak{b}=\Span_\C\{e_{\alpha \beta}\mid 1\leq \alpha\leq \beta \leq p\}\oplus \Span_\C\{e_{\mu \nu}\mid p+1 \leq \nu \leq \mu \leq p+q\}.\]
Then $\mathfrak{b}$ is a Borel sub-algebra of $\mathfrak{k}$. One can verify that both $e_{D_{U,z'_0}}$ and $f_{D_{U,z'_0}}$ are highest weight vectors with respect to $\mathfrak{b}$. The weight of $e_{D_{U,z'_0}}$ with respect to $\mathfrak{b}$ is
\[(\underbrace{q, \ldots,q}_r,\underbrace{s,\ldots,s}_{p-r},\underbrace{-p,\ldots,-p}_s,\underbrace{-r,\ldots,-r}_{q-s}).\]
The weight of $f_{D_{U,z'_0}}$ with respect to $\mathfrak{b}$ is
the above weight plus  \[(\underbrace{\frac{1}{2}(r-s),\ldots,\frac{1}{2}(r-s)}_{p+q}).\]

Now denote the irreducible representation of $\tilde{K}^0$ generated by $e_{D_{U,z'_0}}$ as $V(U)$. By the theory of highest weight we have a $\tilde{K}^0$-equivariant map $\psi_{r,s}^+:V(U)\rightarrow \mathcal{P}_-^{\mathfrak{p}_-}\otimes  \mathrm{det}^{-\frac{1}{2}(r-s)}$ such that
\[\psi_{r,s}^+(e_{D_{U,z'_0}})=f_{D_{U,z'_0}}\otimes 1.\]
Then Theorem A of \cite{Anderson}  can be rephrased as
\begin{theorem}
		\[\psi_{r,s}^+ \in \Hom_{\tilde{K}^0}(\wedge^{rq+ps-rs} \mathfrak{p}_+,\mathcal{P}_-^{\mathfrak{p}_-}\otimes  \mathrm{det}^{-\frac{1}{2}(r-s)}) .\]
\end{theorem}

Let $\{\epsilon_1,\ldots,\epsilon_d\}$ be a basis of $V(U)\subset \wedge^\bullet \mathfrak{p}_+$ such that each $\epsilon_i$ is a weight vector of $\mathfrak{t}$. Extend $\{\epsilon_1,\ldots,\epsilon_d\}$ to a basis of $\wedge^{rq+ps-rs} \mathfrak{p}_+$, take the dual basis inside $\wedge^{rq+ps-rs} \mathfrak{p}_+^*$ and denote the first $d$ basis vectors by $\Omega_1,\ldots,\Omega_d$. We have an isomorphism
\[\Hom_{\tilde{K}^0}(\wedge^\bullet \mathfrak{p}_+,\mathcal{P}_-^{\mathfrak{p}_-}\otimes  \Det^{-\frac{1}{2}(r-s)})\cong (\wedge^\bullet \mathfrak{p}_+^*\otimes\mathcal{P}_-^{\mathfrak{p}_-}\otimes  \Det^{-\frac{1}{2}(r-s)})^{\tilde{K}^0}. \]
Under this isomorphism $\psi_{r,s}^+$ maps to an element $\phi_{r,s}^+\in (\wedge^{rq+ps-rs} \mathfrak{p}_+^*\otimes\mathcal{P}_-^{\mathfrak{p}_-}\otimes  \Det^{-\frac{1}{2}(r-s)})^{\tilde{K}^0}$:
\begin{equation}\label{varphiUpq}
		    \phi_{r,s}^+=\sum_{i=1}^d \psi_{r,s}^+(\epsilon_i) \Omega_i.
\end{equation}
The element thus defined is independent of the choice of the basis $\{\epsilon_1,\ldots, \epsilon_d\}$ and is actually in $(\wedge^\bullet \mathfrak{p}_+\otimes\mathcal{P}_-^{\mathfrak{p}_-})^{\mathrm{SU}(p,q)}$.
		
Let 
\[i:F_{z_0} D_{U,z'_0}\rightarrow D\]
be the natural embedding.
We have the following crucial lemma which states that when restricted to the fiber $F_{z_0} D_{U,z'_0}$ at $z_0$ there is only one term left in $\phi_{r,s}^+$.
		
\begin{lemma}\label{onlyonetermUpq}
		\[i^*(\phi_{r,s}^+(\bx))|_{z_0}=(f_1^-)^{q-s}(f_2^-)^{p-r} i^*(\bigwedge_{(\alpha,\mu)\in I} \xi'_{\alpha,\mu})|_{z_0}.\]
\end{lemma}
\begin{proof}
Recall that  $\{X_{\alpha \mu}\mid (\alpha,\mu)\in I\}$ ($I$ is defined in Equation \eqref{indexI}) span the holomorphic tangent space $N_{z_0}^+ {D_{U,z'_0}}$ of $F_{z_0} D_{U,z'_0}$ at $z_0$ and $X_{\alpha \mu}$ is perpendicular to $N_{z_0}^+ {D_{U,z'_0}}$ if $(\alpha,\mu)\notin I$. Hence $\{\xi'_{\alpha \mu}\mid (\alpha,\mu)\in I\}$ span the holomorphic cotangent space  of $F_{z_0} D_{U,z'_0}$ at $z_0$ and $i^*(\xi'_{\alpha \mu})|_{z_0}=0$  if $(\alpha,\mu)\notin I$.

Now let $S=\{(\alpha,\mu)\mid 1\leq \alpha \leq p, p+1\leq \mu \leq p+q\}$. Then 
\[\{\bigwedge_{(\alpha,\mu)\in T} \xi'_{\alpha \mu}\mid T\subseteq S, |T|=rq+ps-rs\}\]
is a basis of $\wedge^{rq+ps-rs} \mathfrak{p}_+^* $. If $T\neq I$, by the argument in the last paragraph we have 
\[i^*(\bigwedge_{(\alpha,\mu)\in T} \xi'_{\alpha \mu})|_{z_0}=0.\]
So the only term left in $i^*(\phi_{r,s}^+(\bx))|_{z_0}$ is the first term in \eqref{varphiUpq} which is the right hand side of the lemma by the definition of $\phi_{r,s}^+$.
\end{proof}

Similarly let $\mathcal{P}_+$ be the Fock model for the dual pair $(\Uni(p,q),\Uni(r+s,0))$. 
Recall that $\mathcal{P}_+$ is the polynomial space in the variables $\{u_{ia}^+\mid 1\leq i \leq p+q, 1\leq a \leq r+s\}$.
We now define polynomials $f_{1}^+,f_{2}^+\in \mathcal{P}_+$ by 
\[f_1^+=\Det\left(\begin{array}{cccc}
		    u_{1 1}^+ & u_{1 2}^+ &\ldots & u_{1 r}^+ \\
		    \ldots &  \ldots & \ldots & \ldots \\
		    u_{r 1}^+ & u_{r 2}^+ & \ldots &  u_{r r}^+
		\end{array}
		\right),
f_{2}^+=\Det\left(\begin{array}{cccc}
		    u_{p+1 \ r+1}^+ & u_{p+1 \ r+2}^+ &\ldots & u_{p+1\ r+s}^+ \\
		    \ldots &  \ldots & \ldots & \ldots \\
		    u_{p+s \ r+1}^+ & u_{p+s \ r+2}^+ & \ldots &  u_{p+s\  r+s}^+
\end{array}\right).\]
Then $(f_1^+)^{q-s}(f_2^+)^{p-r}$ is a lowest weight vector with respect to $\mathfrak{b}$ of weight
\[(\underbrace{\frac{1}{2}(s-r)-q,.,\frac{1}{2}(s-r)-q}_r,\underbrace{-\frac{1}{2}(r+s),.,-\frac{1}{2}(r+s)}_{p-r},\underbrace{p+\frac{1}{2}(s-r),.,p+\frac{1}{2}(s-r)}_s,\underbrace{\frac{1}{2}(r+s),.,\frac{1}{2}(r+s)}_{q-s}).\]
$\bigwedge_{(\alpha,\mu)\in I} Y_{\alpha \mu}$ is a lowest weight vector  with respect to $\mathfrak{b}$ of weight
\[(\underbrace{-q, \ldots,-q}_r,\underbrace{-s,\ldots,-s}_{p-r},\underbrace{p,\ldots,p}_s,\underbrace{r,\ldots,r}_{q-s}).\]
Then there is a unique element  
\[\psi_{r,s}^- \in \Hom_{\tilde{K}^0}(\wedge^{rq+ps-rs} \mathfrak{p}_-,\mathcal{P}_+^{\mathfrak{p}_+}\otimes  \Det^{\frac{1}{2}(r-s)}), \]
which maps $\bigwedge_{(\alpha,\mu)\in I} Y_{\alpha \mu}$ to $(f_1^+)^{q-s}(f_2^+)^{p-r}$ and a corresponding element
\[\phi_{r,s}^- \in (\wedge^{rq+ps-rs} \mathfrak{p}_-^*\otimes\mathcal{P}_+^{\mathfrak{p}_+}\otimes  \Det^{\frac{1}{2}(r-s)})^{\tilde{K}^0}.\]

Now let $\mathcal{P}$ be the infinitesimal Fock model for the dual pair $(\Uni(p,q),\Uni(r+s,r+s))$. We have
\[\mathcal{P}=\mathcal{P}_-\otimes \mathcal{P}_+\cong \mathcal{P}_-\otimes \Det^{-\frac{1}{2}(r-s)}\otimes \mathcal{P}_+\otimes \Det^{\frac{1}{2}(r-s)}.\]
We define the following outer wedge product
\[\phi_{r,s}=\phi_{r,s}^+\wedge \phi_{r,s}^-.\]
It is immediate that
\[\phi_{r,s}\in (\wedge^{2rq+2ps-2rs}\mathfrak{p}^*\otimes \mathcal{P})^{K}.\]
We say $(r,s)$ is the signature of $\phi_{r,s}$.

\begin{remark}
A priori the definition of $\psi_{r,s}$ depends on the tuple $(z_0,U,z'_0)$ subject to the condition that $z_0\in D_{U,z'_0}$ or equivalently $z_0\subset D(U)\times D(U^\bot)$ and $z'_0=z\cap U$. However it follows from Lemma \ref{lem:K acts on U} and the $K$-invariance of $\psi_{r,s}$ that the definition of $\psi_{r,s}$ depends only on $z_0$ instead of the tuple  $(z_0,U,z'_0)$.
\end{remark}

\begin{remark}
The form $\phi^+_{r,0}$ and $\phi^-_{r,0}$ are also constructed in \cite{BMM2}. The form $\phi_{r,0}$ is constructed in both \cite{KM90} and \cite{BMM2} and is called the Kudla-Millson form in the literature. 
\end{remark}

\subsection{Case B}\label{Spcocylesection}
We use the fact that $G=\Sp(2n,\R)\cong \Sp(2n,\C)\cap \Uni(n,n)$ (\cite{Anderson}). More precisely let $V_0$ be a $2n$-dimensional real vector space with a skew symmetric form $(,)_1$. Then we can extend $(,)$ linearly to a skew symmetric form $S(,)$ on $V=V_0\otimes_\R \C$. We can also extend $(,)_1$ to a skew Hermitian form $(,)$ on $V$ satisfying \eqref{(,)}. 
Then
\[\Sp(V_0,(,)_1)=\Sp(V,S(,))\cap \rU(V,(,)).\]
Moreover let $\mathfrak{g}_0=\mathfrak{sp}(V_0,(,)_1)$. Then we have (Section 7 of \cite{KM90})
\[\mathfrak{g}_0\cong \Sym^2(V_0),\]
and
\[\mathfrak{g}=\mathfrak{g}_0\otimes \C\cong \mathfrak{sp}(V,S(,))\cong \Sym^2 (V),\]
where $\Sym^2(V)$ acts on $V$ by 
\[(v\otimes \tilde{v}+\tilde{v}\otimes v) (x)=v S(\tilde{v},x)+\tilde{v} S(v,x), \forall x\in V.\]
We will denote $(v\otimes \tilde{v}+\tilde{v}\otimes v)\in \Sym^2(V)$ by $v\diamond \tilde{v}$. The linear transformation $a(\zeta)$ introduced in equation (\ref{azeta}) sits inside $ \Sp(V,S(,))\cap \Uni(V,(,))$. Thus the almost complex structure $J_\mathfrak{p}$ introduced in \eqref{J_p} stabilize $\mathfrak{g}_0$ and induced an almost complex structure on $\mathfrak{g}_0$.

The map $\Sp(V_0,(,)_1) \hookrightarrow \rU(V,(,))$ induces an embedding of symmetric spaces $D\hookrightarrow \tilde{D}$, where $D$ is the symmetric space of $\Sp(V_0,(,)_1)$ and $\tilde{D}$ is the symmetric space of $\rU(V,(,))$. To be more precise, recall that $\tilde{D}$ is the set of $n$-dimensional subspaces of $V$ such that $z\in \tilde{D}$ if and only if the Hermitian form $i (,)$ is negative definite on $z$. Then we have
\[D=\{z\in \tilde{D}\mid S(,)|_z \text{ is zero}\}.\]

Choose a symplectic basis $\{E_1,\ldots, E_n,F_1,\ldots, F_n \}$ of $V_0$ and let 
\[v_\alpha=\frac{1}{\sqrt{2}}(E_\alpha- F_\alpha i), \ v_\mu=\frac{1}{\sqrt{2}}(E_{\mu-n}+ F_{\mu-n} i)\]
for $1\leq \alpha  \leq n, n+1\leq \mu \leq 2n$ (in this subsection we keep this convention of index). Let 
\[z_0=\mathrm{span}_\C \{v_{n+1},\ldots ,v _{2n}\}\in D.\]
Its stablizer is $K\cong \rU(n)$. We have the corresponding Cartan decomposition 
\[\mathfrak{g}_0=\mathfrak{k}_0+\mathfrak{p}_0 \text{ and }\mathfrak{p}=\mathfrak{p}_+ +\mathfrak{p}_-,\]
where $ \mathfrak{k}$ is the $0$ eigenspace of $J_\mathfrak{p}$ and $\mathfrak{p}_+$ (resp. $\mathfrak{p}_-$) is the $+i$ (resp. $-i$) eigenspace of $J_\mathfrak{p}$.
In terms of matrices we have
\begin{equation}\label{pofSp}
           \mathfrak{k}_0=\left\{\left(\begin{array}{cc}
		    A & 0 \\
		   0  & -A^t
\end{array}\right)\mid A \in M_n(\C), A^*=-A\right\},
		    \mathfrak{p}_0=\left\{\left(\begin{array}{cc}
		    0 & A \\
		   A^*  & 0
\end{array}\right)\mid A\in M_n(\C), A^t=A\right\}.
\end{equation}
Define $V^+=\mathrm{span}_\C \{v_{1},\ldots ,v _{n}\}$ and $V^-\mathrm{span}_\C \{v_{n+1},\ldots ,v _{2n}\}$, we have the identification
\[\mathfrak{k}=\Span_\C\{v\diamond \tilde{v}\mid v\in V^+,\tilde{v}\in V^-\}.\]
Define 
\[X_{\alpha \beta}=\frac{1}{i}v_\alpha \diamond v_\beta= e_{\beta , \alpha+n}+e_{\alpha,\beta+n},\ Y_{\alpha \beta}=i v_{\alpha+n} \diamond v_{\beta+n}= e_{\alpha+n, \beta}+e_{\beta+n,\alpha}\]
for $1\leq \alpha \leq \beta \leq n$. Then
\[\mathfrak{p}_+=\Span_\C\{X_{\alpha \beta}\mid  1\leq \alpha \leq \beta \leq n\},\ \mathfrak{p}_-=\Span_\C\{Y_{\alpha \beta}\mid  1\leq \alpha \leq \beta \leq n\}.\]
We also let $\{\xi'_{\alpha \beta}\mid 1\leq \alpha\leq \beta \leq n\}$ (resp.  $\{\xi''_{\alpha \beta}\mid 1\leq \alpha\leq \beta \leq n\}$) be the basis of $\mathfrak{p}_+^*$ (resp. $\mathfrak{p}_-^*$ ) dual to the above basis.
			
Now fix $1\leq r \leq n$. Define
\[U=\Span_\R\{E_1,\ldots,E_r,F_1,\ldots,F_r\}.\]
Let 
\[z'_0=\mathrm{span}_\C \{v_{n+1}, \ldots, v_{n+r}\}\in D(U).\]
Recall that in Definition \ref{specialsubsymmetricspace} we define a sub symmetric space $D_{U,z'_0}$ of $D$. The holomorphic tangent space of  $D_{U,z'_0}$ at $z_0$ is
\[T_{z_0}^{+} {D_{U,z'_0}}=\Span_\C\{X_{\alpha \beta}\mid r+1\leq \alpha \leq \beta\leq n\}.\]
The holomorphic tangent space of the fiber $F_{z_0} D_{U,z'_0}$ (see Subsection \ref{fibrationpi}) at $z_0$ is
\begin{equation}\label{normalspaceSp}
		    N_{z_0}^{+} {D_{U,z'_0}}=\Span_\C\{X_{\alpha \beta}\mid(\alpha,\beta)\in I\},
\end{equation}
where $I$ is the index set
\begin{equation}\label{indexISp}
		  I=\{(\alpha,\beta)\mid 1\leq \alpha \leq r,\alpha  \leq \beta \leq n\}.
\end{equation}
Define $e_{D_{U,z'_0}} \in \wedge^{\bullet} \mathfrak{p}_+$ by
\[e_{D_{U,z'_0}}=\bigwedge_{(\alpha,\beta)\in I} X_{\alpha \beta}.\]

Let $\mathcal{P}_-$ be the infinitesimal Fock model for the dual pair  $(\Sp(2n,\R),\Orth(0,2r))$ defined in Section \ref{seesawonesection}. Recall that $\mathcal{P}_-$ is the same as the Fock model  for the dual pair  $(\rU(n,n),\rU(0,r))$ and is the polynomial space in the variables $\{u_{i k}^-\mid 1\leq i \leq 2n, 1\leq k \leq r\}$.
Define $f^-\in \mathcal{P}_-$ by
		\[f^-=\Det\left(\begin{array}{cccc}
		    u_{1 1}^- & u_{1 2}^- &\ldots & u_{1 r}^- \\
		    \ldots &  \ldots & \ldots & \ldots \\
		    u_{r 1}^- & u_{r 2}^- & \ldots &  u_{r r}^-
		\end{array}
		\right).\]
Also define
\begin{equation}\label{f_DUSp}
    f_{D_{U,z'_0}}=(f^-)^{n-r+1}.
\end{equation}
It can be shown that $\mathfrak{k}$ acts on $\mathcal{W}_-$ by (c.f. formula (4.2) of \cite{Anderson})
\begin{equation}
    \omega(v_\alpha \diamond v_{\beta+n})=\omega(-ie_{\alpha,\beta}+ie_{\beta+n,\alpha+n})=-i\sum_{k=1}^{r}(u_{\alpha+n,k}\frac{\partial}{\partial u_{\beta+n,k}}+u_{\alpha,k}\frac{\partial}{\partial u_{\beta,k}})-i r \delta_{\alpha \beta}.
\end{equation}
The adjoint action of $\mathfrak{k}$ on $\mathfrak{p}_+$ induces an action on $ \wedge^\bullet \mathfrak{p}_+$. Define 
\[\mathfrak{b}=\Span_\C\{v_\alpha \diamond v_{\beta+n}\mid 1\leq \alpha \leq \beta\leq n\}.\]
$\mathfrak{b}$ is a Borel sub-algebra of $\mathfrak{k}$. Both $e_{D_{U,z'_0}}$ and $f_{D_{U,z'_0}}$ are highest weight vectors of $\mathfrak{b}$. Moreover they have the same weight 
		\[(\underbrace{n+1,\ldots,n+1}_r,\underbrace{r,\ldots,r}_{n-r}).\]
Let  $\tilde{K}$ the preimage of $K$ under the map $\Mp(\mathbb{W})  \rightarrow \Sp(\mathbb{W})$. Using the seesaw pair 
\begin{equation}
		    \xymatrixrowsep{0.3in}
		\xymatrixcolsep{0.3in}
		\xymatrix{ \Uni(n,n) &\Orth(0,2r)\\
			\Sp(2n,\R)\ar[u] \ar[ur] & \Uni(0,r) \ar[u]\ar[ul]}
\end{equation}
and facts about $\widetilde{\rU(n,n)}$ (see the last subsection), we can see that 
\[\tilde{K}=K\times \{\pm 1\}.\]
		
Now denote the irreducible representation of $K$ generated by $e_{D_{U,z'_0}}$ as $V(U)$. By the theory of highest weight we have a $K$-equivariant map $\psi_{2r}^+:V(U)\rightarrow \mathcal{P}_-^{\mathfrak{p}_-}$ such that
		\[\psi_{r}^+(e_{D_{U,z'_0}})=f_{D_{U,z'_0}}.\]
Theorem B of \cite{Anderson}  is
\begin{theorem}
		\[\psi_{r}^+ \in \Hom_{K}(\wedge^{\frac{1}{2}n(n+1)-\frac{1}{2}(n-r)(n-r+1)} \mathfrak{p}_+,\mathcal{P}_-^{\mathfrak{p}_-}).\]
\end{theorem}
As in case A, the definition of $\psi_{r}^+$ is independent of the choice of $U$. We have an isomorphism
\[\Hom_{K}(\wedge^\bullet \mathfrak{p}_+,\mathcal{P}_-^{\mathfrak{p}_-})\cong (\wedge^\bullet \mathfrak{p}_+^*\otimes\mathcal{P}_-^{\mathfrak{p}_-})^{K},\]
under which $\psi_r^+$ is mapped to an element $\phi_r^+\in (\wedge^{\frac{1}{2}n(n+1)-\frac{1}{2}(n-r)(n-r+1)} \mathfrak{p}_+^*\otimes\mathcal{P}_-^{\mathfrak{p}_-})^K$.

Let 
	    \[i:F_{z_0} D_{U,z'_0}\rightarrow D\]
be the natural embedding. 
Using the definition of $\phi_{r}^+$, we can prove the following lemma in a similar way as Lemma \ref{onlyonetermUpq}.
		
\begin{lemma}\label{onlyonetermSp}
		\[i^*(\phi_{r}^+(\bx))|_{z_0}=(f^-)^{n-r+1} i^*(\bigwedge_{(\alpha,\mu)\in I} \xi'_{\alpha,\mu})|_{z_0}.\]
\end{lemma}
		
Let $\mathcal{P}_+$ be the infinitesimal Fock model for the dual pair $(\Sp(2n,\R),\Orth(2r,0))$.
Recall that $\mathcal{P}_+$ is the polynomial space in the variables $\{u_{i a}^+\mid 1\leq i \leq 2n, 1\leq a \leq r\}$.
Define $f^+\in \mathcal{P}_+$ by
		\[f^+=\mathrm{det}\left(\begin{array}{cccc}
		    u_{1 1}^+ & u_{1 2}^+ &\ldots & u_{1 r}^+ \\
		    \ldots &  \ldots & \ldots & \ldots \\
		    u_{r 1}^+ & u_{r 2}^+ & \ldots &  u_{r r}^+
		\end{array}
		\right).\]
Then both $(f^+)^{n-r+1}$ and $\bigwedge_{(\alpha,\beta)\in I} Y_{\alpha \beta}$ are lowest weight vector with respect to $\mathfrak{b}$ of weight 
\[(\underbrace{-n-1,\ldots,-n-1}_r,\underbrace{-r,\ldots,-r}_{n-r}).\]
There is a unique element
\[\psi_{r}^- \in \Hom_{K}(\wedge^{\frac{1}{2}n(n+1)-\frac{1}{2}(n-r)(n-r+1)} \mathfrak{p}_-,\mathcal{P}_+^{\mathfrak{p}_+}),\]
which maps $(f^+)^{n-r+1}$ to $\bigwedge_{(\alpha,\beta)\in I} Y_{\alpha \beta}$ and a corresponding element
\[\phi_{r}^- \in (\wedge^{\frac{1}{2}n(n+1)-\frac{1}{2}(n-r)(n-r+1)} \mathfrak{p}_-^*\otimes\mathcal{P}_+^{\mathfrak{p}_+})^K.\]
		
Now let $\mathcal{P}$ be the infinitesimal Fock model for the dual pair $(\Sp(2n,\R),\Orth(2r,2r))$. We have
		\[\mathcal{P}=\mathcal{P}_-\otimes \mathcal{P}_+.\]
Define 
		\[\phi_{r}=\phi_{r}^+\wedge \phi_{r}^-.\]
Then it is immediate that 
\[\phi_{r}\in (\wedge^{n(n+1)-(n-r)(n-r+1)}\mathfrak{p}^*\otimes \mathcal{P})^K.\]

\subsection{Case C}\label{Ostarcocyclesection}
Let $V$ be a $n$-dimensional right $\H$ vector space with skew Hermitian form $(,)_2$ satisfying \eqref{(,)_2}. Let $V_\C$ denote the underlying complex vector space of $V$.
Define $V^\sigma$ the conjugate left $\H$ vector space of $V$ as follows. The underlying Abelian group of $V^\sigma$ is the same with that of $V$. The scalar multiplication of $V^\sigma$ is defined by 
		\[h v= v h^\sigma, \forall h \in \H,\]
where the left hand side is scalar multiplication in $V^\sigma$ while the right hand side is scalar multiplication in $V$. There is an isomorphism $V^\sigma\rightarrow \Hom_\H(V,\H)$ as left $\H$ module given by the form $(,)_2$
\[v\mapsto (v,\cdot)_2.\]
One can also identify $V\otimes_\H V^\sigma$ with $\Hom_\H(V,V)$ by the map
		\[v\otimes {\tilde{v}} \mapsto v(\tilde{v},\cdot)_2.\]
Define
\[\Sym(V\otimes_\H  V^\sigma)=\Span_\R\{v\otimes \tilde{v}+\tilde{v} \otimes v\mid v,\tilde{v}\in V\}\subset V\otimes_\H  V^\sigma.\]
By the above identification, $\Sym(V\otimes_\H  V^\sigma)$ acts on $V$ by 
\[(v\diamond \tilde{v}) (x)=v(\tilde{v},x)_2+\tilde{v}(v,x)_2, \forall x\in V,\]
where $v\diamond \tilde{v}=v\otimes \tilde{v}+\tilde{v} \otimes v$. One can check that this action is $\H$-linear and satisfies
\[((v\diamond \tilde{v})(x),y)_2+(x, (v\diamond \tilde{v})(y))_2=0.\]
Moreover we have
\begin{lemma}
\[\Sym(V\otimes_\H  V^\sigma)\cong \mathfrak{o}^*(V,(,)_2)=\mathfrak{o}^*(2n).\]
\end{lemma}
		
Define $(,)$ on $V_\C$ by
		\[(v,\tilde{v})=a+bi \text{ if } (v,\tilde{v})_2=a+bi+cj+dk.\]
Then $i (,)$ is (complex) Hermitian of signature $(n,n)$. We also define $S(,)$ on $V_\C$ by 
		\[S(v,\tilde{v})=(vj,\tilde{v}).\]
One can check that $S(,)$ is symmetric complex bilinear. We have the following fact (\cite{Anderson})
		\[\Orth^*(V,(,)_2)=\Uni(V, (,))\cap \Orth(V,S(,)).\]
It can also be shown that
\[\mathfrak{g}=\mathfrak{g}_0\otimes \C\cong \mathfrak{o}(V,S(,))\cong \bigwedge^2 (V_\C),\]
where $\mathfrak{g}_0=\mathfrak{o}^*(V,(,))$ and $\bigwedge^2 (V_\C)$ acts on $V$ by 
\[v \wedge \tilde{v} (x)=-v S(\tilde{v},x)+\tilde{v} S(v,x), \forall x\in V.\]

The map $\rO^*(V,(,)_2) \hookrightarrow \rU(V,(,))$ induces an embedding of symmetric spaces $D\hookrightarrow \tilde{D}$, where $D$ is the symmetric space of $\rO^*(V,(,)_2)$ and $\tilde{D}$ is the symmetric space of $\rU(V,(,))$. To be more precise, recall that $\tilde{D}$ is the set of $n$-dimensional subspaces of $V_\C$ such that $z\in \tilde{D}$ if and only if the Hermitian form $i (,)$ is negative definite on $z$. Then we have
\[D=\{z\in \tilde{D}\mid S(,)|_z \text{ is zero}\}.\]
		
Explicitly choose an orthogonal basis $\{v_1,\ldots,v_n\}$ of $V$ such that
		\[(v_\alpha,v_\beta)_2=-i \delta_{\alpha \beta}\]
for $1\leq \alpha,\beta \leq n$. Then $\{v_1,\ldots,v_n,v_{n+1} ,\ldots,v_{2n} \}$ is a basis of $V_\C$, where $v_\mu=v_{\mu-n} j $
for $n+1\leq \mu \leq 2n$. The $\C$-linear transformation $a(\zeta)$ introduced in equation (\ref{azeta}) sits inside $ \Orth(V,S(,))\cap \Uni(V,(,))$. Thus the almost complex structure \eqref{J_p} stabilize $\mathfrak{g}_0$ and induces an almost complex structure on  $\mathfrak{g}_0$. Let
\[z_0=\Span_\C\{v_{n+1},\ldots,v_{2n}\}\in D. \]
Its stablizer is $K\cong \rU(n)$. We have the corresponding Cartan decomposition 
\[\mathfrak{g}_0=\mathfrak{k}_0+\mathfrak{p}_0 \text{ and }\mathfrak{p}=\mathfrak{p}_+ +\mathfrak{p}_-,\]
where $ \mathfrak{k}$ is the $0$ eigenspace of $J_\mathfrak{p}$ and $\mathfrak{p}_+$ (resp. $\mathfrak{p}_-$) is the $+i$ (resp. $-i$) eigenspace of $J_\mathfrak{p}$. 
In terms of matrices we have
\begin{equation}\label{pofOstar}
            \mathfrak{k}_0=\big\{\left(\begin{array}{cc}
		    A & 0 \\
		    0  & -A^t
		\end{array}\right)\mid A\in M_n(\C), A^*=-A\big\},
		    \mathfrak{p}_0=\big\{\left(\begin{array}{cc}
		    0 & A \\
		    A^*  & 0
		\end{array}\right)\mid A\in M_n(\C), A^t=-A\big\}.
\end{equation}
If we define $V^+=\Span_\C \{v_1,\ldots,v_n\}$ and $V^-=\Span_\C \{v_{n+1},\ldots,v_{2n}\}$, we have the identification
\[\mathfrak{k}=\Span_\C\{v\wedge \tilde{v}\mid v\in V^+,\tilde{v}\in V^-\}.\]
Define 
\[X_{\alpha \beta}=i v_\alpha \wedge v_\beta=e_{\alpha,\beta+n}-e_{\beta,\alpha+n}, \ Y_{\alpha \beta}=\frac{1}{i}v_{\alpha+n} \wedge v_{\beta+n}= e_{\beta+n,\alpha}- e_{\alpha+n,\beta}\]
for $1\leq \alpha < \beta \leq n$. Then
\[\mathfrak{p}_+=\Span_\C\{X_{\alpha \beta}\mid 1\leq \alpha < \beta \leq n \},\ \mathfrak{p}_-=\Span_\C\{Y_{\alpha \beta}\mid 1\leq \alpha < \beta \leq n \}.\]
We also let $\{\xi'_{\alpha \beta}\mid 1\leq \alpha< \beta \leq n\}$ (resp. $\{\xi''_{\alpha \beta}\mid 1\leq \alpha< \beta \leq n\}$) be the basis of $\mathfrak{p}_+^*$ (resp. $\mathfrak{p}_-^*$ ) dual to the above basis.
		
Now fix $1\leq r \leq n$. Define 
\[U=\Span_\H\{v_1,\ldots,v_r\}.\]
Define 
\[z'_0=\Span_\C\{v_{n+1},\ldots,v_{n+r}\} \in D(U).\]
In Definition \ref{specialsubsymmetricspace} we define a sub symmetric space $D_{U,z'_0}$ of $D$.  The holomorphic tangent space of  $D_{U,z'_0}$ at $z_0$ is
\[T_{z_0}^{+} {D_{U,z'_0}}=\Span_\C\{X_{\alpha \beta}\mid r+1\leq \alpha < \beta\leq n\}.\]
The holomorphic tangent space of the fiber $F_{z_0} D_{U,z'_0}$ (see Subsection \ref{fibrationpi} ) at $z_0$ is
\begin{equation}\label{normalspaceOstar}
		    N_{z_0}^{+} {D_{U,z'_0}}=\Span_\C\{X_{\alpha \beta}\mid (\alpha,\beta)\in I\},
\end{equation}
where $I$ is the index set
\begin{equation}\label{indexIOstar}
   I=\{(\alpha,\beta)\mid 1\leq \alpha \leq r,\alpha  < \beta \leq n\}. 
\end{equation}
Define $e_{D_{U,z'_0}} \in \wedge^{\bullet} \mathfrak{p}_+$ by
\[e_{D_{U,z'_0}}=\bigwedge_{(\alpha,\beta)\in I} X_{\alpha \beta}.\]
		
Let $\mathcal{P}_-$ be the Fock model defined in the last section for the dual pair $(\Orth^*(2n),\Sp(0,r))$ in Section \ref{seesawtwosection}. Recall that $\mathcal{P}_-$ is the same as the Fock model  for the dual pair  $(\rU(n,n),\rU(0,r))$ and is the polynomial space in the variables $\{u_{i k}^-\mid 1\leq i \leq 2n, 1\leq k \leq r\}$. Define $f^-\in \mathcal{P}_-$ by
		\[f^-=\Det\left(\begin{array}{cccc}
		    u_{1 1}^- & u_{1 2}^- &\ldots & u_{1 r}^- \\
		    \ldots &  \ldots & \ldots & \ldots \\
		    u_{r 1}^- & u_{r 2}^- & \ldots &  u_{r r}^-
		\end{array}
		\right).\]
Also define $f_{D_{U,z'_0}}\in \mathcal{W}_-$ by
\begin{equation}\label{f_DUOstar}
    f_{D_{U,z'_0}}=(f^-)^{n-r-1}.
\end{equation}
It can be shown that $\mathfrak{k}$ acts on $\mathcal{W}_-$ by (c.f. equation (5.2) of \cite{Anderson})
\begin{equation}
    \omega(v_\alpha \wedge v_{\beta+n})=\omega(-ie_{\alpha \beta}+ie_{\beta+n,\alpha+n})=-i\sum_{k=1}^{r}(u_{\alpha+n,k}\frac{\partial}{\partial u_{\beta+n,k}}+u_{\alpha,k}\frac{\partial}{\partial u_{\beta,k}})-i r \delta_{\alpha \beta}.
\end{equation}
The adjoint action of $\mathfrak{k}$ on $\mathfrak{p}_+$ induces an action on $ \wedge^\bullet \mathfrak{p}_+$. Define 
\[\mathfrak{b}=\Span_\C\{v_\alpha \wedge v_{\beta+n}\mid 1\leq \alpha \leq \beta\leq n\}.\]
$\mathfrak{b}$ is a Borel sub-algebra of $\mathfrak{k}$. Both $e_{D_{U,z'_0}}$ and $f_{D_{U,z'_0}}$ are highest weight vectors of $\mathfrak{b}$. Moreover they have the same weight 
\[(\underbrace{n-1,\ldots,n-1}_r,\underbrace{r,\ldots,r}_{n-r}).\]
Let $\tilde{K}$ be  the preimage of $K$ under the map $\Mp(\mathbb{W})  \rightarrow \Sp(\mathbb{W})$. Using the seesaw pair 
\begin{equation}
		    \xymatrixrowsep{0.3in}
		\xymatrixcolsep{0.3in}
		\xymatrix{ \Uni(n,n) &\Sp(0,r)\\
			\rO^*(2n)\ar[u] \ar[ur] & \Uni(0,r) \ar[u]\ar[ul]}
\end{equation}
and facts about $\widetilde{\rU(n,n)}$ (see subsection \ref{varphiUpq} or \cite{Paul}), we can see that 
		\[\tilde{K}=K\times \{\pm 1\}.\]
Now denote the irreducible representation of $K$ generated by $e_{D_{U,z'_0}}$ as $V(U)$. By the theory of highest weight we have a $K$-equivariant map $\psi_+:V(U)\rightarrow \mathcal{P}_-^{\mathfrak{p}_-}$ such that
\[\psi_r^+(e_{D_{U,z'_0}})=f_{D_{U,z'_0}}.\]
Then Theorem C of \cite{Anderson}  is
\begin{theorem}
		\[\psi_r^+ \in \Hom_{K}(\wedge^\bullet \mathfrak{p}_+,\mathcal{P}_-^{\mathfrak{p}_-}).\]
\end{theorem}
As in case A, the definition of $\psi_{r}^+$ is independent of the choice of $U$. We have an isomorphism
\[\Hom_{K}(\wedge^\bullet \mathfrak{p}_+,\mathcal{P}_-^{\mathfrak{p}_-})\cong (\wedge^\bullet \mathfrak{p}_+^*\otimes\mathcal{P}_-^{\mathfrak{p}_-})^{K},\]
under which $\psi_r^+$ is mapped to an element $\phi_r^+\in (\wedge^{\frac{1}{2}n(n-1)-\frac{1}{2}(n-r)(n-r-1)} \mathfrak{p}_+^*\otimes\mathcal{P}_-^{\mathfrak{p}_-})^K$.

Let 
	    \[i:F_{z_0} D_{U,z'_0}\rightarrow D\]
be the natural embedding (see equation \eqref{fiber} for the definition of $F_{z_0} D_{U,z'_0}$). 
Using the definition of $\phi_r^+$, we can prove the following lemma in a similar way as Lemma \ref{onlyonetermUpq}: 
		
\begin{lemma}\label{onlyonetermOstar}
		\[i^*(\phi_+(\bx))|_{z_0}=(f^-)^{n-r-1} i^*(\bigwedge_{(\alpha,\mu)\in I} \xi'_{\alpha,\mu})|_{z_0}.\]
\end{lemma}
		
Let $\mathcal{P}_+$ be the infinitesimal Fock model for the dual pair $(\rO^*(2n),\Sp(r,0))$.
Recall that $\mathcal{P}_+$ is the polynomial space in the variables $\{u_{i a}^+\mid 1\leq i \leq 2n, 1\leq a \leq 2r\}$.
Define $f^+\in \mathcal{P}_+$ by
		\[f^+=\Det\left(\begin{array}{cccc}
		    u_{1 1}^+ & u_{1 2}^+ &\ldots & u_{1 r}^+ \\
		    \ldots &  \ldots & \ldots & \ldots \\
		    u_{r 1}^+ & u_{r 2}^+ & \ldots &  u_{r r}^+
		\end{array}
		\right).\]
Then both $(f^+)^{n-r-1}$ and $\bigwedge_{(\alpha,\beta)\in I} Y_{\alpha \beta}$ are lowest weight vector with respect to $\mathfrak{b}$ of weight 
\[(\underbrace{-n+1,\ldots,-n+1}_r,\underbrace{-r,\ldots,-r}_{n-r}).\]
There is a unique element
\[\psi_{r}^- \in \Hom_{K}(\wedge^{\frac{1}{2}n(n-1)-\frac{1}{2}(n-r)(n-r-1)} \mathfrak{p}_-,\mathcal{P}_+^{\mathfrak{p}_+}),\]
which maps $(f^+)^{n-r+1}$ to $\bigwedge_{(\alpha,\beta)\in I} Y_{\alpha \beta}$ and a corresponding element
\[\phi_{r}^- \in (\wedge^{\frac{1}{2}n(n-1)-\frac{1}{2}(n-r)(n-r-1)} \mathfrak{p}_-^*\otimes\mathcal{P}_+^{\mathfrak{p}_+})^K.\]
		
Now let $\mathcal{P}$ be the infinitesimal Fock model for the dual pair $(\Orth^*(2n),\Sp(r,r))$. We have
		\[\mathcal{P}=\mathcal{P}_-\otimes \mathcal{P}_+.\]
Define 
		\[\phi_r=\phi_r^+\wedge \phi_r^-.\]
It is immediate that
\[\phi_r\in (\wedge^{n(n-1)-(n-r)(n-r-1)}\mathfrak{p}^*\otimes \mathcal{P})^K.\]
We say $r$ is the rank of $\phi_r$.

\begin{remark}\label{K'actionremark}
We have in each case the following subgroups of the dual group $G'$ of $G$:
\begin{enumerate}
    \item Case A: $G'=\rU(r+s,r+s)$, $K'=\rU(r+s)\times \rU(r+s)$, $K_+=\rU(r+s,0)$, $K_-=\rU(0,r+s)$.
    \item Case B: $G'=\rO(2r,2r)$, $K'=\rO(2r)\times \rO(2r)$, $K_+=\rO(2r,0)$, $K_-=\rO(0,2r)$.
    \item Case C: $G'=\Sp(r,r)$, $K'=\Sp(r)\times \Sp(r)$, $K_+=\Sp(r,0)$, $K_-=\Sp(0,r)$.
\end{enumerate}
In all cases $K'=K_-\times K_+$ and as a $K'$-representation $\mathcal{P}=\mathcal{P}_- \boxtimes \mathcal{P}_+$, where $K_+$ acts trivially on $ \mathcal{P}_-$ and $K_-$ acts trivially on $ \mathcal{P}_+$.
\end{remark}
		
\subsection{Closedness of holomorphic differentials}
In this subsection we simply write $\phi^+$ for $\phi^+_{r,s}$, $\phi^-$ for $\phi_{r,s}^-$ and $\phi$ for $\phi_{r,s}$ in case A and similarly in other cases. 
We will prove that the cochains $\phi^+$, $\phi^-$ and $\phi$ are closed hence cocycles.
First we recall the following well-known fact.
\begin{lemma}\label{holomorphicformoncompactKahler}
A holomorphic form on a compact K\"{a}hler manifold $M$ is closed.
\end{lemma}
\begin{proof}
On a compact K\"{a}hler manifold we have the following identity of Laplacians
\[\Delta_d=2\Delta_\partial=2\Delta_{\bar{\partial}}.\]
A holomorphic form $\varphi$ is $\bar{\partial}$-closed. It is also $\bar{\partial}^*$-closed because it has Hodge-type $(p,0)$. Hence $\varphi$ is $\Delta_{\bar{\partial}} $-harmonic, hence $\Delta_d$-harmonic. Hence $\varphi$ is closed.
\end{proof}

\begin{theorem}\label{closedness}
The form  $\phi^+$ ($\phi^-$ resp.) constructed in the previous subsections is closed as an element of $C^{\bullet}(\mathfrak{g},K;\mathcal{P}_-)$ ($C^{\bullet}(\mathfrak{g},K;\mathcal{P}_+)$ resp).
\end{theorem}
\begin{proof}
We prove the holomorphic case, the anti-holomorphic case is similar.

Recall that in all three cases $\phi^+$ takes values in the $K$-representation generated by $f_{D_{U,z'_0}}$ (see \eqref{fD_UUpq}, \eqref{f_DUSp} and \eqref{f_DUOstar}), a special harmonic polynomial considered in \cite{KV}.
By \cite{KV}, $f_{D_{U,z'_0}}$ is in a representation $\mathcal{A}\boxtimes \theta(\mathcal{A})$, where $\mathcal{A}$ is the an irreducible $(\mathfrak{g},K)$ module and $\theta(\mathcal{A})$ is an irreducible representation of the compact group dual $\tilde{K}_-$, where $K_-$ is as in Remark \ref{K'actionremark}. 
Hence $\phi^+$ is a holomorphic cochain in $H^{(\bullet,0)}(\mathfrak{g},K;\mathcal{A})$. 

By the proof of Proposition 2.3 in \cite{Anderson}, there is a cocompact lattice $\Gamma$ of $G$ and a $(\mathfrak{g},K)$-map 
\[I:\mathcal{P}_-\rightarrow C^\infty (\Gamma \backslash G),\]
such that $I(f)\neq 0$. Since $f\in\mathcal{A}$ and $\mathcal{A}$ is irreducible, we know that $I$ is injective when restricted on $\mathcal{A}$.
Hence the map on the cochains
\[I_*: C^\bullet(\mathfrak{g},K;\mathcal{A})\rightarrow C^\bullet(\mathfrak{g},K;C^\infty (\Gamma \backslash G)) \]
is also injective.

Now $I_*(\phi^+)$ is a holomorphic form on a compact K\"{a}hler manifold. So it is closed by Lemma \ref{holomorphicformoncompactKahler}. Since $I_*$ is a map of chain complexes we know that
\[I_*(d \phi^+)=d I_* (\phi^+ )=0.\]
Because $I_*$ is injective on $C^\bullet(\mathfrak{g},K;\mathcal{A})$, we know that 
\[d\phi^+=0.\]
This finishes the proof of the theorem.
\end{proof}
\begin{remark}
There is an alternative proof of the above theorem.
Let $C$ be the Casimir element of the universal enveloping algebra of $\mathfrak{g}$. Then one can show by explicit computation that $C$ acts trivially on $\mathcal{A}$ in all three cases. Then Proposition 3.1 in Chapter II of \cite{BW} guarantees that any element in $C^\bullet(\mathfrak{g},K;\mathcal{A})$ is closed.
\end{remark}

\begin{corollary}
The cochain $\phi$ is closed.
\end{corollary}	
\begin{proof}
Recall that $\phi=\phi^+\wedge \phi^-$. The differential operator $d$ for the chain complex $C^{\bullet}(g,K;\mathcal{P})$ satisfies
\[d=d_-\otimes 1+1\otimes d_+,\]
where $d_-$ (resp. $d_+$) is the differential operator for the chain complex $C^{\bullet}(g,K;\mathcal{P}_-)$ (resp. $C^{\bullet}(g,K;\mathcal{P}_+$)). The corollary now follows from Theorem \ref{closedness}.
\end{proof}
		
\subsection{Cocycles in the Schr\"odinger Model}
Recall that we have defined a map $\iota:\mathcal{P}\rightarrow \mathcal{S}(V^{m})$, where $m=r+s$ in case A, $m=r$ in case B and C.  We define
\begin{equation}\label{definitionofvarphirs}
    \varphi_{r,s}=\iota(\phi_{r,s}) \text{ in case A},\quad \varphi_r=\iota(\phi_r) \text{ in case B and C}
\end{equation}
in $(\wedge^\bullet\mathfrak{p}^*\otimes \mathcal{S}(V^m))^K$ (recall from \eqref{V_0^2r=V^r} that $V^{r}\cong V_0^{2r}$ in case B). 
Since $\iota$ is an isomorphism of $(\mathfrak{g},K) $  modules between $\mathcal{P}$ and its image in $\mathcal{S}(V^m)$, by the corollary to Theorem \ref{closedness}, $\varphi_{r,s}$ (resp. $\varphi_{r}$ in case B and C) is closed.
More explicitly we have in the Schrodinger Model (see equation (\ref{varphiUpq})):

\begin{equation}\label{varphiSchrodinger}
	        \varphi_{r,s} \text{ or }\varphi_{r}=\varphi_0\sum_{i,j=1}^d p_{i j} \Omega_i \wedge \bar{\Omega}_j,
\end{equation}
where $\Omega_i$ are of Hodge type $(d',0)$ where
\begin{equation}\label{eq:d'}
    d'=\begin{cases}
    pq-(p-r)(q-s) & \text{in case A},\\
   \frac{1}{2}n(n+1)-\frac{1}{2}(n-r)(n-r+1) & \text{in case B},\\
    \frac{1}{2}n(n-1)-\frac{1}{2}(n-r)(n-r-1) & \text{in case C},
    \end{cases}
\end{equation}
and $p_{ij}$ is a polynomial in the variables $\{z_{\alpha a},\bar{z}_{\alpha a},z_{\mu a},\bar{z}_{\mu a}\mid 1\leq \alpha  \leq p,p+1\leq \mu \leq p+q,1\leq a\leq m\}$ (in case B and C, $p=q=n$).
Define 
\[f'_{+}=(-2\sqrt{2}\pi)^r \mathrm{det}\left(\begin{array}{cccc}
		    z_{1 1} & z_{1 2} &\ldots & z_{1 r} \\
		    \ldots &  \ldots & \ldots & \ldots \\
		    z_{r 1} & z_{r 2} & \ldots &  z_{r r}
		\end{array}
		\right),\]
\[f'_{-}=(-2\sqrt{2}\pi)^s \mathrm{det}\left(\begin{array}{cccc}
		    \bar{z}_{p+1 \ r+1} & \bar{z}_{p+1 \ r+2} &\ldots & \bar{z}_{p+1\ r+s} \\
		    \ldots &  \ldots & \ldots & \ldots \\
		    \bar{z}_{p+s \ r+1} & \bar{z}_{p+s \ r+2} & \ldots &  \bar{z}_{p+s\  r+s}
\end{array}\right).\]

Recall that
	    $i:F_{z_0} D_{U,z'_0}\rightarrow D$
is the an embedding (Section \ref{fibrationpi}).
\begin{lemma}\label{highestermSchrodinger}
Let $\varphi$ be $\varphi_{r,s}$ in case A and $\varphi_r$ in case B and C.
The highest term (in terms of the degree of the polynomial in front of $\varphi_0$) of $i^*(\varphi(\bx))|_{z_0}$ is 
\begin{align*}
    (f'_+ \bar{f}'_+)^{q-s} (f'_- \bar{f}'_-)^{p-r} &i^*(\bigwedge_{(\alpha,\mu)\in I} \xi'_{\alpha,\mu}\wedge  \xi''_{\alpha,\mu})|_{z_0}, \\
    (f'_+ \bar{f}'_+)^{n-r+1}  &i^*(\bigwedge_{(\alpha,\mu)\in I} \xi'_{\alpha,\mu}\wedge  \xi''_{\alpha,\mu})|_{z_0}, \\
    (f'_+ \bar{f}'_+)^{n-r-1}  &i^*(\bigwedge_{(\alpha,\mu)\in I} \xi'_{\alpha,\mu}\wedge  \xi''_{\alpha,\mu})|_{z_0},
\end{align*}
respectively in case A, B and C, where $I$ is specified in \eqref{indexI}, \eqref{indexISp} and \eqref{indexIOstar} respectively. Notice that in each case $i^*(\bigwedge_{(\alpha,\mu)\in I} \xi'_{\alpha,\mu}\wedge  \xi''_{\alpha,\mu})|_{z_0}$ is a constant multiple of the volume form of $F_{z_0}D_{U,z'_0}$ at $z_0$.
\end{lemma}	    
\begin{proof}
It follows from combining Lemma \ref{onlyonetermUpq} (resp. Lemma \ref{onlyonetermSp} or Lemma \ref{onlyonetermOstar}), the analogous result for $\phi^-$ and Lemma \ref{highesttermofpolynomial}.
\end{proof}

\part{Proof of the main theorems}

\section{Poincar\'e dual and Thom form}\label{thomformsection}
In this section we start to prove Theorem \ref{PDthm} and Theorem \ref{thmB} in the introduction. We resume the notations of Section \ref{locallysymmetricspace}.
We give the symmetric space $D=G_\infty /K_\infty$ the $G_\infty$-invariant Riemannian metric $\tau$ induced by the trace form on $\mathfrak{p}_0$. $D$ is then a negatively curved symmetric K\"ahler manifold whose sectional curvatures are automatically bounded as it is homogeneous. We suppose the sectional curvature of $D$ is bounded below by $-\rho^2$.

Choose $\bx\in V^m $ ($1\leq m\leq n$) satisfying the following assumptions in the three cases of our interests respectively:
\begin{enumerate}
    \item Case A: the Hermitian form $(,)_{v_1}$ restricted to $\mathrm{span}_B\{\bx\}\otimes k_{v_1}$ is non-degenerate and has signature $(r,s)$. In particular $1\leq r \leq p, 1\leq s \leq q$ and $r+s=m$.\
    \item Case B: the Hermitian form $(,)$ restricted to $\mathrm{span}_B\{\bx\}$ is non-degenerate.\
    \item Case C: the skew-Hermitian form $(,)$ restricted to $\mathrm{span}_B\{\bx\}$ is non-degenerate. 
\end{enumerate}

In each case we define a $2m$ dimensional $B$-vector space $W$, a form $\langle,\rangle$ on $W$ that is non-degenerate and split. We assume that it is Hermitian, skew Hermitian and Hermtian respectively in the three cases. Let $G'$ be the group of $B$-linear transformations on $W$ preserving $\langle,\rangle$. For each Archimedean place $v$ of $k$ define $G'_v=G'(k_v)$ and let $G'_\infty=\prod_{v\in S_\infty}G'_v$. We have 
\begin{enumerate}
    \item $G'_\infty=\Uni(r+s,r+s)^{\ell}$ in case A,\
    \item  $G'_\infty=\Orth(2m,2m)^{\ell}$ in case B,\
    \item  $G'_\infty=\Sp(m,m)^{\ell}$ in case C.
\end{enumerate}

We fix a point $z_0$ in the symmetric space $D$ of $G_\infty$, or equivalently a maximal compact group $K_\infty$ of $G_\infty$. By our assumptions, $K_\infty=K_{v_1}\times \prod_{v\in S_\infty,v\neq v_1} G_v$. Let
		\[\mathfrak{g}_0=\mathfrak{p}_0\oplus \mathfrak{k}_0\]
be the corresponding Cartan decomposition on the Lie algebra $\mathfrak{g}_0$ of $G_\infty$ (we drop the subscript $0$ to indicate complexification). Let $ \Omega^\bullet (D)$ be the space of smooth differential forms on $D$ with values in $\C$. 
	    
Let $\mathcal{S}(V_\infty^m)$ be the set of Schwartz functions on $V_\infty^m$. There is an isomorphism given by evaluation at $z_0$:
\begin{equation}
	        (\Omega^\bullet (D)\otimes \mathcal{S}(V_\infty^m))^{G_\infty}\rightarrow (\wedge^\bullet \mathfrak{p}\otimes \mathcal{S}(V_\infty^m))^{K_\infty}.
\end{equation}

Let $\tilde{G}'_{\infty}$ be the metaplectic cover of $G'_{\infty}$.
$\tilde{G}'_\infty$ acts on $\mathcal{S}(V_\infty^n)$ by the Weil representation $\omega$ and the action commutes with that of $G_\infty$.	    
Any form $\psi \in (\wedge^\bullet\mathfrak{p}^*\otimes \mathcal{S}(V_\infty^m))^{K_\infty}$ give rise to the form $\tilde{\psi}$ defined by
\begin{equation}\label{translateform}
    \tilde{\psi}(z,g',\bx)=L^*_{g_z^{-1}}(\omega(g_z,g')\psi(\bx)),
\end{equation}
where $g_z\in G_\infty$ such that $g_z z_0=z$, $g'\in \tilde{G}'_\infty$ and $L_{g}$ denotes the left action by $g$ on $D$. 
Then we have 
\[\tilde{\psi}\in (\Omega^\bullet (D)\otimes \mathcal{S}(V_\infty^m))^{G_\infty}.\]

Let $\varphi_{v}$ be the vacuum vector (Gaussian function) of $\mathcal{S}(V_v)$ for any Archimedean place $v\neq v_1$. Define
\begin{equation}\label{definitionofvarphiinfty}
    \varphi_\infty=\varphi \otimes \prod_{v\in S_\infty,v\neq v_1} \varphi_v  \in (\wedge^{\bullet} \mathfrak{p}\otimes \mathcal{S}(V_\infty^m))^{K_\infty},
\end{equation}
where $\varphi$ is the form defined in Section \ref{Andersoncocyle} such that 
\begin{enumerate}
    \item $\varphi=\varphi_{r,s}$ in case A,
    \item $\varphi=\varphi_{m}$ in case B and C.
\end{enumerate}
Applying equation \eqref{translateform} to $\varphi_\infty$ we get a form $\tilde{\varphi}_\infty\in (\Omega^\bullet (D)\otimes \mathcal{S}(V_\infty^m))^{G_\infty}$.
\begin{remark}\label{rmk:independence of z0}
The definition of the form $\tilde{\varphi}_\infty$ is independent of the choice of the base point $z_0$. Let us focus on case A. Suppose we have two base points $z_0, z'_0\in D$ corresponding to two different maximal compact groups $K_{v_1}$ and $K'_{v_1}$. The choice of $z_0$ (resp. $z'_0$) corresponds to an orthogonal splitting $V_{v_1}=V_+\oplus V_-$ (resp. $V_{v_1}=V'_+\oplus V'_-$) such that $z_0=V_-$ (resp. $z'_0=V_-$). We can
choose a basis $\{v_1,\ldots,v_{p+q}\}$ (resp. $\{v'_1,\ldots,v'_{p+q}\}$) of $V$ such that $V_+=\Span\{v_1,\ldots,v_p\}$ (resp. $V'_+=\Span\{v'_1,\ldots,v'_p\}$) and $V_-=\Span\{v_{p+1},\ldots,v_{p+q}\}$ (resp. $V'_-=\Span\{v'_{p+1},\ldots,v'_{p+q}\}$). This basis determines  holomorphic coordinate functions $(z_{ij})_{1\leq i \leq p+q, 1\leq j \leq m}$ (resp. $(z'_{ij})_{1\leq i \leq p+q, 1\leq j \leq m}$) of $V^m$. Let $g\in G_{v_1}$ such that $g v_i=v'_i$. Then $g z_0=z'_0$ and $L_g^* (z'_{ij})=z_{ij}$. By the construction of $\varphi$ and $\varphi'$ we have \begin{equation}\label{eq:independence of z0}
    L_g^*(\varphi')=\varphi.
\end{equation}
By the $K_{v_1}$-invariance of $\varphi$ and $K'_{v_1}$-invariance of $\varphi$, \eqref{eq:independence of z0} is true for any $g$ such that $g z_0=z'_0$. This implies the desired independence.
\end{remark}

Choose an $\Oo_B$ lattice $\cL$ of $V$ and choose $\Gamma\in G$ as in Section \ref{locallysymmetricspace}, then by \cite{Weil}  we can choose an arithmetic subgroup $\Gamma'\subset G'(k)$ such that the theta distribution
	    \[\theta_{\cL} (\psi)=\sum_{\bx\in \cL^m} \psi(\bx), \forall \psi \in \mathcal{S}(V_\infty^m)\]
is 	$ \Gamma\times \tilde{\Gamma}'$-invariant where $\tilde{\Gamma}'=\Gamma'\times\{\pm 1\}$ is the metaplectic cover of $\Gamma'$.
	    
We now apply $\theta_{\cL}$ to $\tilde{\varphi}_\infty$ to get 
	    \[\theta_{\cL,\tilde{\varphi}_\infty}=\theta_{\cL} (\tilde{\varphi}_\infty)\in \Omega^\bullet(\Gamma\backslash D)\otimes C^\infty(\tilde{\Gamma}' \backslash \tilde{G}'_\infty).\]
We also define 
\begin{equation}\label{fourierexpansionoftheta}
    \theta_{\cL,\beta,\tilde{\varphi}_\infty}(z,g')=\sum_{\bx\in \cL^m,(\bx,\bx)=\beta} \tilde{\varphi}_\infty(z,g',\bx)
\end{equation}
for a matrix $\beta \in M_m(B)$. We have the following Fourier expansion of $\theta_{\cL,\tilde{\varphi}_\infty}$:
	    \begin{align*}
	        \theta_{\cL,\tilde{\varphi}_\infty}(z,g')=&\sum_{\beta} \sum_{\bx\in \cL^m,(\bx,\bx)=\beta} \tilde{\varphi}_\infty(z,g',\bx)\\
	        =&\sum_{\beta} \theta_{\cL,\beta,\tilde{\varphi}_\infty}(z,g'),
	    \end{align*}
where $\beta$ runs over all possible inner product matrices $(\bx,\bx)$.

By the construction of $\varphi_\infty$ we know that it has Hodge bi-degree $(d',d')$ where $d'$ is as in \eqref{eq:d'}.
Let $\eta$ be any closed differential form on $\Gamma\backslash D$, define a smooth function $ \theta_{\cL,\tilde{\varphi}_\infty}(\eta)$ on $\tilde{G}'_\infty$ by
	    \[\theta_{\cL,\tilde{\varphi}_\infty}(\eta)(g')=\int_{\Gamma\backslash D}\eta\wedge \theta_{\cL,\tilde{\varphi}_\infty}(g').\]
We call the above map the {\bf geometric theta lift} defined by $\tilde{\varphi}_\infty$. This geometric theta lift is a map 
	    \[H^{(d-d',d-d')}(\Gamma\backslash D, \C)\rightarrow \calA(\Gamma'\backslash \tilde{G}'_\infty),\]
where $\calA(\Gamma'\backslash \tilde{G}'_\infty)$ is the space of analytic functions on $\Gamma'\backslash \tilde{G}'_\infty$ (see the proof of Lemma \ref{analyticity} for the statement of analyticity). 
Also define $a_{\cL,\beta,\tilde{\varphi}_\infty}(\eta)$ to be the $\beta$-coefficient of $\theta_{\cL,\tilde{\varphi}_\infty}(\eta)$:
\begin{equation}\label{a_beta}
a_{\cL,\beta,\tilde{\varphi}_\infty}(\eta)=\int_{\Gamma\backslash D}\eta\wedge \theta_{\cL,\beta,\tilde{\varphi}_\infty}(z,g')=\int_{\Gamma\backslash D}\eta\wedge\sum_{\bx\in L^m,(\bx,\bx)=\beta} \tilde{\varphi}_\infty(z,g',\bx).
\end{equation}
Then we have the Fourier expansion:
\[\theta_{\cL,\tilde{\varphi}_\infty}(\eta)=\sum_{\beta} a_{\cL,\beta,\tilde{\varphi}_\infty}(\eta).\]

We assume from now on that $\beta$ is non-degenerate (a nonsingular matrix) and $ \sigma_1(\beta)$ has signature $(r,s)$ in case A (recall that $\sigma_i$ is the map $k\rightarrow k_{v_i}$).	    
By a theorem of Borel (\cite{Borel}, Theorem 9.11), the set 
	    \[\{\bx\in \mathcal{L}^m\mid (\bx,\bx)=\beta\}\]
consists of finitely many $\Gamma$-orbits. We choose $\Gamma$-orbit representatives $\{\bx_1,\ldots,\bx_o\}$ and define 
	    \[U_i=\Span \{\bx_i\}, 1\leq i \leq o.\]
For each $1\leq i \leq o$ choose a base point $z_i \in D(U_i)$. Let $C_{\bx_i,z_i}$ be the generalized special cycle (Definition \ref{definitionofspecialcycle}). Then each $C_{\bx_i,z_i}$ has complex codimension $d'$.
Let
\[{\bf z}=\{z_1,z_2,\ldots, z_o\}.\]
Define
	    \[C_{\beta,{\bf z}}=\sum_{i=1}^o C_{\bx_i, z_i}.\]
$C_{\beta,{\bf z}}$ is a cycle in the Chow group of $\Gamma \backslash D$. By remark \ref{homologyindependent}, the homology class $[C_{\beta,{\bf z}}]$ is independent of the choice of ${\bf z}$, so we simply denote by $[C_\beta]$ its homology class. Whenever we take the period of a closed differential form $\eta$ on $\Gamma \backslash D$ we can write $\int_{C_\beta} \eta$ (resp. $\int_{C_\bx} \eta$).

\begin{theorem}\label{fouriercoeff}
Assume that $\beta=(\bx,\bx)$ is non-degenerate and in case A also assume that $\sigma_1(\beta)$ has signature $(r,s)$. Then for any closed differential form $ \eta\in \Omega^{d-d',d-d'}(\Gamma\backslash D)$ we have
\[\int_{\Gamma\backslash D}\eta\wedge\sum_{\by\in \Gamma \cdot \bx} \tilde{\varphi}_\infty(z,g',\by)=\kappa(g',\beta)\int_{C_\bx}\eta,\]
and
		\[a_{\cL,\beta,\tilde{\varphi}_\infty}(\eta)=\kappa(g',\beta)\int_{C_\beta}\eta, \]
where $\kappa$ is an analytic function in $g'$.
\end{theorem}

\begin{remark}
The proof of Theorem \ref{fouriercoeff} only depends on the following two properties of $\varphi_\infty$:
\begin{enumerate}
    \item $\varphi_\infty \in (\wedge^{(d',d')}\mathfrak{p}^*\otimes \iota(\mathcal{P}))^{K_\infty}$, where $\mathcal{P}$ is the polynomial Fock space (see Section \ref{dualpair}).
    \item $\varphi_\infty$ is closed.
\end{enumerate}
So the conclusion is true for any form $\psi$ that satisfies the above two conditions. 
\end{remark}

Let us also briefly recall Poincar\'e duality in terms of differential forms. For a closed submanifold $C$ inside a compact oriented manifold $M$, we say that $\tau$ is a Poincar\'e dual form of $C$ if it is a closed form such that 
\[\int_M \eta \wedge \tau=\int_C \eta\]
for any closed form $\eta$. Poincar\'e dual form is unique up to exact forms. 

The above definition of Poincar\'e dual form can be applied when $C$ is a subvariety of the projective variety $M$ if we interpret $\int_C \eta$ as the integration of $\eta$ over the nonsingular locus of $C$.

With the above theory of Poincar\'e duality in mind, Theorem \ref{fouriercoeff} is equivalent to the following theorem which is Theorem \ref{PDthm} in the introduction.
\begin{theorem}\label{generatingseries}
Assume that $\beta=(\bx,\bx)$ is non-degenerate and in case A also assume that $\sigma_1(\beta)$ has signature $(r,s)$. Then
\[\sum_{\by \in \Gamma \cdot \bx}[\varphi(z,g',\by)]=\kappa(g',\beta) \mathrm{PD}([C_{\bx}]),\]
where $[\varphi]$ is the cohomology class of $\varphi$ in $ H^*(\Gamma \backslash D)$
, $\mathrm{PD}([C_\bx])\in H^*(\Gamma\backslash D)$ is the Poincar\'e dual of $[C_\bx]$, and $\kappa$ is a function that is analytic in $G'$.
Moreover 
\[[\theta_{\cL,\beta,\varphi}(z,g')]=\kappa(g',\beta) \mathrm{PD}([C_\beta]).\]
\end{theorem}

When the function $\kappa(g',\beta)$ is nonzero,  $\frac{1}{\kappa(g',\beta)}[\theta_{\cL,\beta,\tilde{\varphi}_\infty}(z,g')]$ (see equation \eqref{fourierexpansionoftheta}) is the Poincar\'e dual of $C_\beta$. We will prove the following theorem  in section \ref{methodofLaplacesection}.
\begin{theorem}\label{genericnonzero}
There exists $m'\in \tilde{M}'_\infty$ (see \eqref{M'}) such that for sufficiently large $\lambda\in \R$,
\[\kappa(\lambda m',\beta)\neq 0.\]
In particular, $ \kappa(g',\beta)$ is nonzero for a generic $g'$ as it is analytic.
\end{theorem}

These theorems mean that $[\theta_{\cL,\tilde{\varphi}_\infty}]$ can be seen as a "generating" series of $\mathrm{PD}([C_\beta])$. Of course as for now we do not have an explanation for all the "Fourier coefficients". Only for those satisfying the condition of Theorem \ref{generatingseries} do we have an explanation.

\subsection{}The rest of the section will be devoted to proving Theorem \ref{fouriercoeff} under an assumption that we will verify later.

First we need the following "unfolding" lemma:
\begin{lemma}\label{unfoldinglemma}
\[\int_{\Gamma\backslash D}\eta\wedge \sum_{{\bf y} \in \Gamma \cdot \bx_j} \tilde{\varphi}_\infty(z,g', \by)= \int_{\Gamma_{\bx_j}\backslash D} \eta\wedge \tilde{\varphi}_\infty(z,g',{\bx_j}),\]
In particular
		\[\int_{\Gamma\backslash D}\eta\wedge \sum_{\bx\in \mathcal{L}^m,(\bx,\bx)=\beta} \tilde{\varphi}_\infty(z,g',\bx)=\sum_{j=1}^o \int_{\Gamma_{\bx_j}\backslash D} \eta\wedge \tilde{\varphi}_\infty(z,g',{\bx_j}).\]
\end{lemma}
\begin{proof}
The proof is the same as that of Lemma 4.1 of \cite{KM90}.
\end{proof}
		
For $\bx\in V^m$ assume $\beta=(x,x)$ satisfies the condition of Theorem \ref{fouriercoeff}. Let
\[U_\bx=\Span_{B_{v_1}} \{\bx\}, G_\bx=(G_{v_1})_{U_\bx}, \Gamma_\bx=\Gamma \cap G_\bx. \]
Choose $z'_0\in D(U_\bx)$.
The critical topological observation is that $\Gamma_\bx\backslash D_{\bx,z'_0}$ is a totally geodesic sub manifold of the space $E_\bx=\Gamma_\bx \backslash D$ and $E_\bx$ is in a natural way (topologically) a vector bundle over $\Gamma_\bx\backslash D_{\bx,z'_0}$. In fact the fibration $\pi:E_\bx\rightarrow \Gamma_\bx\backslash D_{\bx,z'_0}$ defined in Section \ref{fibrationpi} has fibers  $F_z D_{\bx,z'_0}$ diffeomorphic to the vector space $N_z D_{\bx,z'_0}$. 
The following theorem is a special case of Theorem 2.1 of \cite{KM88}.
\begin{theorem}\label{thomform}
Let $\Phi$ be a  degree $2d'$ differential form on $E_\bx$ satisfying 
\begin{enumerate}
		    \item $\Phi$ is closed.\
		    \item $||\Phi ||\leq \exp(-d \cdot \rho \cdot r) p(r) $ for some polynomial $p$, where $r$ is the geodesic distance to $\Gamma_\bx\backslash D_{\bx,z'_0}$ and $\rho$ is a positive number such that the sectional curvature of $D$ is bounded below by $-\rho^2$.
\end{enumerate}
If $\eta$ is a closed bounded $2(d-d')$-form on $E_\bx$ then we have 
\[\int_{E_\bx} \eta \wedge \Phi=\kappa\int_{\Gamma_\bx\backslash D_{\bx,z'_0}} \eta ,\]
where $\kappa$ is the constant
\[\int_{F_z D_{\bx,z'_0}} \Phi\]
which is independent of the choice of $z\in \Gamma_\bx\backslash D_{\bx,z'_0}$.
\end{theorem}
\begin{remark}
When $\kappa=1$, $\Phi$ is a Thom form of the fiber product $E_\bx\rightarrow \Gamma_\bx\backslash D_{\bx,z'_0}$.
\end{remark}

\noindent
{\bf Proof of Theorem \ref{fouriercoeff} assuming rapid decrease of $\tilde{\varphi}_\infty$ on $E_\bx$}:
By Lemma \ref{unfoldinglemma} it suffices to show the following. 
\begin{proposition}\label{Thomformprop}
Assume that $x\in V^m$ and $\beta=(x,x)$ satisfies the condition of Theorem \ref{fouriercoeff}. We have
\[\int_{\Gamma_{\bx}\backslash D} \eta\wedge \tilde{\varphi}_\infty(z,g',\bx) =\kappa(g',\beta)\int_{C_\bx}\eta,\]
where $\kappa(g',\beta)$ is the function in Theorem \ref{fouriercoeff}. In other words, the cohomology class $[\tilde{\varphi}_\infty(z,g',\bx)]$ in $H^*(E_\bx)$ is $\kappa(g',\beta)$ times the Thom class of the fiber bundle $\pi:E_\bx\rightarrow \Gamma_\bx\backslash D_{\bx,z'_0}$.
\end{proposition}
\begin{remark}
Some literature such as \cite{BottTu} requires a Thom form to be compactly supported on the fiber. However by the same method as in \cite{KM88}, $\tilde{\varphi}_\infty(z,g',\bx)$ can be shown to be cohomologous to a compactly supported form.
\end{remark}
\begin{proof}
We would like to apply Theorem \ref{thomform} to $\Phi=\tilde{\varphi}_\infty$.
By Theorem \ref{closedness} and its corollary, $\tilde{\varphi}_\infty$ satisfies condition (1) of Theorem \ref{thomform}.
We will verify that $\tilde{\varphi}_\infty$ satisfies condition (2) of Theorem \ref{thomform} in Theorem \ref{rapiddecreasingthm} and its corollary.
Define
\begin{equation}\label{kappaintegral}
    \kappa(g',\bx,z'_0):=\int_{F_{z_0} D_{\bx,z'_0}}\tilde{\varphi}_\infty(z,g',\bx)
\end{equation}
for any $z_0\in C_{\bx,z'_0}$. By Theorem \ref{thomform}, we know that
\[
    \int_{\Gamma_{\bx}\backslash D} \eta\wedge \tilde{\varphi}_\infty(z,g',\bx) =\kappa(g',\bx,z'_0)\int_{\Gamma_\bx\backslash D_{\bx,z'_0}}\eta.
\]
When the map 
$\rho_{U_\bx,z'_0}:\Gamma_\bx\backslash D_{U_\bx} \rightarrow \Gamma\backslash D$ (see Section \ref{generalizedspecial}) is an embedding, one can immediately conclude that 
\begin{equation}\label{thomequation1}
    \int_{\Gamma_{\bx}\backslash D} \eta\wedge \tilde{\varphi}_\infty(z,g',\bx) =\kappa(g',\bx,z'_0)\int_{C_{\bx,z'_0}}\eta.
\end{equation}
In general one can use Lemma \ref{resolutionofsingularity2} to see that the map
\[{\Gamma_\bx\backslash D_{\bx,z'_0}}\rightarrow C_{\bx,z'_0}\]
induced by $\rho_{U_\bx,z'_0}$ is birational, so 
\[\int_{C_{\bx,z'_0}}\eta=\int_{\Gamma_\bx\backslash D_{\bx,z'_0}}\eta\]
if we interpret $\int_{C_{\bx,z'_0}}\eta$ as integration over the nonsingular locus of $C_{\bx,z'_0}$. Hence equation (\ref{thomequation1}) holds again.

Thus in order to prove Theorem \ref{fouriercoeff}, it remains to show that 
\begin{enumerate}
    \item $\kappa(g',\bx,z'_0)$ only depends on $g'$ and $\beta$, so we can define $\kappa(g',\beta):=\kappa(g',\bx,z'_0) $.
    \item $\kappa(g',\beta)$ is an analytic function in $g'$.
\end{enumerate}

\begin{lemma}\label{lem:kappa independent of z'}
$\kappa(g',\bx,z'_0)$ is independent of the choice of $z'_0$. Moreover it only depends on $\beta=(\bx,\bx)$ when $\beta$ satisfies the condition of Theorem \ref{fouriercoeff}. 
\end{lemma}
\begin{proof}
Let $\eta$ be a $\Gamma$-invariant form on $D$ such that $\int_{C_{\bx,z'_0}}\eta\neq 0 $. For example, one can take $\eta=\omega_k^d$ where $\omega_k$ is the K\"ahler form of $M$ and $d$ is the complex dimension of $C_{\bx,z'_0}$.
For different choices of $z'_0$, $C_{\bx,z'_0}$ are homologous. Thus $\int_{C_{\bx,z'_0}}\eta$ is independent of $z'_0$. The left hand side of \eqref{thomequation1} is visibly independent of $z'_0$.  This shows that $\kappa(g',\bx,z'_0) $ is independent of the choice of $z'_0$ so we we can simply denote it by $\kappa(g',\bx)$.

Let $ {\bf \tilde{x}}\in V^m_\infty$ and $\tilde{z}$ be a point in $D(U_{\bf \tilde{x}})$. Assume that $\tilde{z}_0\in D_{{\bf \tilde{x}},\tilde{z}}$. Suppose $\beta=(\bx,\bx)=({\bf \tilde{x}},{\bf \tilde{x}})$.
Then by Witt's Theorem there is a $g\in G_\infty$ such that $g\bx={\bf \tilde{x}}$ and we have
\begin{align*}
		    \int_{F_{\tilde{z}_0} D_{{\bf \tilde{x}},\tilde{z}}}\tilde{\varphi}_\infty(z,g',{\bf \tilde{x}}) =& \int_{F_{\tilde{z}_0} D_{g\bx,\tilde{z}}}\tilde{\varphi}_\infty(z,g',g\bx)\\
		    =& \int_{F_{\tilde{z}_0} D_{g\bx,\tilde{z}}} L^*_{g^{-1}} (\tilde{\varphi}_\infty (z,g',\bx)) \\
		    =&\int_{L_{g^{-1}}(F_{\tilde{z}_0} D_{g\bx,\tilde{z}})} \tilde{\varphi}_\infty (z,g',\bx)\\
		    =& \int_{F_{g^{-1} \tilde{z}_0} D_{\bx,g^{-1}\tilde{z}}}\tilde{\varphi}_\infty (z,g',\bx).
\end{align*}
This proves that 
$\kappa(g',{\bf \tilde{x}})=\kappa(g',\bx)$. So the proof is finished.

\end{proof}
In particular, when $\beta$ satisfies the condition of Theorem \ref{fouriercoeff}, we can define
		\[\kappa(g',\beta)=\kappa(g',\bx)\]
for any $\bx\in V^m$ such that $(\bx,\bx)=\beta$.
\begin{lemma}\label{analyticity}
The function $\kappa(g',\beta)$ is analytic in $g'$. 
\end{lemma}
\begin{proof}
First we claim that $\tilde{\varphi}_{\infty}(z,g',\bx)$ is an analytic function on $\tilde{G'}_\infty$. First we know that $\tilde{G'}_\infty=P'_\infty \tilde{K}'_\infty$, where $P'$ is the Siegel parabolic as in Section \ref{dualpair} and $\tilde{K}'_\infty$ is the maximal compact subgroup of $\tilde{G'}_\infty$ fixing the Gaussian function. There are analytic functions 
\[s_1:\tilde{G'}_\infty\rightarrow P'_\infty ,\ s_2: \tilde{G'}_\infty\rightarrow \tilde{K}'_\infty\]
such that $g=s_1(g)s_2(g)$. Now any $\phi$ in the polynomial Fock space is $\tilde{K}'_\infty$-finite and $\tilde{K}'_\infty$ analytic. The action of $P'_\infty$ is given by \eqref{M'action} and \eqref{n'action}, hence analytic. This proves the claim.  

Now by Theorem \ref{rapiddecreasingthm}, $\tilde{\varphi}_{\infty}(z,g',\bx)$ is fast decreasing on the fiber $F_{z} D_{U,z'_0}$ and the decrease is locally uniform in $g'$. So $\kappa(g',\beta)$ as defined in \eqref{kappaintegral} is an analytic function in $g'$.
\end{proof}
This finishes the proof of Theorem \ref{fouriercoeff} and Proposition \ref{Thomformprop} under the assumption that $\tilde{\varphi}_\infty$ is rapid decreasing on $E_\bx$. 
\end{proof}

\section{Rapid decrease of Schwartz function valued forms on the normal bundle}\label{rapiddecreasechapter}
For a moment, we keep the notations of the previous section. For $ \psi\in (\bigwedge^\bullet\mathfrak{p}\otimes \mathcal{S}(V_{v_1}^m))^{K_{v_1}}$ define 
\begin{equation}\label{tildepsi}
    \tilde{\psi}(z,g',\bx)=L^*_{g_z^{-1}}(\omega(g_z,g')\psi(\bx)),
\end{equation}
where $g_z\in G_{v_1}$ such that $g_z z_0=z$.

Recall that the Riemannian distance $d(D', z)$ between a totally geodesic submanifold $D'$ and $z \in D$ is the length of the shortest geodesic joining $z$ to a point of $D'$. This geodesic is necessarily normal to $D'$. Choose a base point $z_0 \in D$. If $z = \exp_{z_0} {(tu)}$ for 
$u \in N_{z_0} D'$ and $\|u\|=1$, where $\exp_{z_0}$ denotes the exponential map of $D$ at the base point $z_0$, then 
\[d(D', z) = t.\]
The goal of the section is to prove the following theorem. 
\begin{theorem}\label{rapiddecreasingthm}
Fix a non-degenerate $\bx\in V^m$,  a point $z'_0\in D(U)$, where $U=\Span_{B_{v_1}}\{\bx\}$, an element $g'\in \tilde{G}'_{v_1}$ and  $ \psi\in (\bigwedge^\bullet\mathfrak{p}\otimes \mathcal{S}(V_{v_1}^m))^{K_{v_1}}$. Then for any real number $\rho>0$, there is a positive constant $C'_\rho$ depending on the above fixed data such that 
\[|| \tilde{\psi}(z,g',\bx)||\leq C'_{\rho} \exp({-\rho \cdot d(D_{U,z'_0}, z ) }),\]
where the norm is taken with respect to the Riemannian metric of the symmetric space $D$. The constant $C'_\rho$ depends continuously on $g'$.
\end{theorem}
\begin{corollary}
The form $\tilde{\varphi}_{\infty}(z,g',\bx)$ defined in \eqref{definitionofvarphiinfty} satisfies condition (2) of Theorem  \ref{thomform}. In particular it is integrable on $F_{z_0} D_{U,z'_0}$ for any $z_0\in D_{U,z'_0}$.
\end{corollary}
\begin{proof}
Recall that by our assumption, $G_\infty=G_{v_1}\times \prod_{v\neq v_1} G_v$ and $G_v$ is compact for $v\neq v_1$. Hence any $ \psi_\infty\in (\bigwedge^\bullet\mathfrak{p}\otimes \mathcal{S}(V_\infty^m))^{K_\infty}$ is constant in $G_v$ for $v\neq v_1$.  Theorem \ref{rapiddecreasingthm} implies immediately that $ \tilde{\varphi}_{\infty}(z,g',\bx)$ satisfies condition (2) of Theorem  \ref{thomform}. The integrablity statement follows from Theorem \ref{thomform}.
\end{proof}
In the rest of this section, we work over real vector spaces and real Lie groups. To simplify notations, we denote $G_{v_1}$ by $G$, $G'_{v_1}$ by $G'$ and $K_{v_1}$ by $K$. In other words $G=\rU(p,q)$ in case A, $\Sp(2n,\R)$ in case B and $\rO^*(2n)$ in case C. Throughout the section we assume that $V$ is a complex vector space with a non-degenerate skew Hermitian form $(,)$ such that $i(,)$ is of signature $(p,q)$. In case B and C, we assume that $p=q=n$ in addition.  We fix an orthogonal basis $\{v_1,\ldots,v_{p+q}\}$ such that 
\[(v_\alpha,v_\alpha)=-i \text{ for } 1\leq \alpha \leq p,\ (v_\mu,v_\mu)=i \text{ for } p+1\leq \mu \leq p+q.\]
In case A we denote by $D=\tilde{D}$ the symmetric space of $\rU(p,q)$. In case B (case C resp.) we denote by $D$ the symmetric space of $\Sp(2n,\R)$ ($\rO^*(2n)$ resp.) while we denote by $\tilde{D}$ the symmetric space of $\rU(n,n)$.

In case A, we let $m=r+s$, $\bx=(\vec{x}_1,\ldots,\vec{x}_{r+s})$ and
\[U:=\Span_\C\{\bx\}.\]
with the assumption that $U$ has signature $(r,s)$ with respect to $(,)$.

In case B, recall from  Section \ref{Spcocylesection} that $V_0$ is a $2n$ dimensional real vector space with a skew symmetric form $(,)_1$ and $V=V_0\otimes_\R \C$. 
Then $\Sp(2n,\R)=G(V_0,(,)_1)$ is a subgroup of $\rU(n,n)$ and its symmetric space $D$ embedds into $\tilde{D}$, the symmetric space of $\rU(n,n)$.
In this case let $m=2r$ and assume $\bx_0=(\vec{x}_1,\ldots,\vec{x}_{m})\in V_0^{m}$.
Define
\[U=\Span_\R\{\bx\}.\]
We assume that $U$ is non-degenerate with respect to $(,)_1$.

In case C, recall from Section \ref{seesawtwosection} and \ref{Ostarcocyclesection} that $V$ is an $n$-dimensional right $\H$-vector space with a skew Hermitian form $(,)_2$. Then $\rO^*(2n)=G(V,(,)_2)$ is a subgroup of $\rU(n,n)$ and its symmetric space $D$ embedds into $\tilde{D}$.
In this case let $m=r$ and assume $\bx=(\vec{x}_1,\ldots,\vec{x}_{r})\in V^r$ and
\[U=\Span_\H\{\bx\}.\]
Assume that $U$ is non-degenerate with respect to $(,)_2$.

In all three cases, a point $z\in D$ (in the latter two cases view $z$ as in $\tilde{D}$ instead of $D$) gives rise to a involution
\begin{equation}
    \theta_{z}:V\rightarrow V, \ \theta_{z}=(-Id_{z})\oplus Id_{z^\bot}.
\end{equation}
Define a positive definite Hermitian form $(,)_{z}$ on $V$ by
\begin{equation}\label{()_z}
    (x,y)_{z}=i (x,\theta_{z} y).
\end{equation}
We call this form the {\bf majorant} of $(,)$ with respect to $z$. We denote $(v,v)_{z}$ as $\|v\|_{z}^2$. $z$ also induces a Cartan decompostion 
\[\mathfrak{g}=\mathfrak{k}+\mathfrak{p}.\]
We identify $T_{z} D$ with $\mathfrak{p}$. Denote $(x,x)_z =$ by $||x||_z^2$.

In all the above cases, define $M:  D \times V^m\rightarrow \R$ to be the function
\begin{equation}\label{definitionMajorant}
    M(z,\bx)=\sum_{\ell=1}^m || \vec{x}_\ell||_{z}^2.
\end{equation}

\begin{lemma}\label{invarianceofmajorant}
For any $h\in G_{v_1}$ and $x,y \in V$, we have
\[(hx,hy)_{hz}=(x,y)_z.\]
In particular,
\[M(h z, h \bx)=M(z,\bx).\]
\end{lemma}
\begin{proof}
Choose a $g$ such that $g z_0=z$, then $h gh^{-1} hz_0=h z$. Hence
\begin{align*}
    (hx,hy)_{hz}=& (hx,\theta_{hz} hy)\\
    =&(hx, h \theta_z h^{-1} hy)\\
    =&(x,\theta_z y)\\
    = & (x,y)_z.
\end{align*}
The second statement of the lemma follows from the above and the definition of $M(z,x)$.
\end{proof}

The main technical ingredient for proving Theorem \ref{rapiddecreasingthm} is the following estimate of $M(z,x)$ for all three cases. 
\begin{theorem}\label{estimateofgaussian}
Let $\bx$, $U$ and $z'_0$ be as in Theorem \ref{rapiddecreasingthm} in each cases. 
There is positive constants $b$ and $c$ depending on $\bx$ and $z'_0$ such that
\begin{equation}
M(z,\bx) \geq c\exp(2b\cdot d(D_{U,z'_0},z))
\end{equation}
for all $z\in D$.
\end{theorem}
It is easy to see that 
\begin{equation}\label{invarianceofdist}
    d(D_{U,z'_0},z)=d(D_{gU,gz'_0},gz), \forall g\in G.
\end{equation}
Equation \eqref{invarianceofdist} and Lemma \ref{invarianceofmajorant} imply that in order to prove Theorem \ref{estimateofgaussian} it suffices to assume (by replacing $(\bx,z'_0)$ by $(g\bx,gz'_0)$ for some $g\in G$) that
\begin{enumerate}
    \item $\Span\{\bx\}=\Span_\C\{v_1,\ldots,v_{r},v_{p+1},\ldots,v_{p+s}\} $ and $z'_0=\Span_\C\{v_{p+1},\ldots,v_{p+s}\}$ in case A,\ 
    \item $\Span_\R\{\bx\}=\Span_\R\{E_1,\ldots,E_r,F_1,\ldots,F_r\} $ and $z'_0=\Span_\C\{E_1+iF_1,\ldots ,E_r+i F_r\}$ in case B,\ 
    \item In case C, $\Span_\H\{\bx\}=\Span_\H\{v_1,\ldots,v_{r}\} $ and $z'_0=\Span_\C\{v_{1} j,\ldots,v_{r} j\}$ in case C.
\end{enumerate}
Let $z_0=\pi(z)$ where $\pi$ is the fibration $D\rightarrow D_{U,z'_0}$ defined in \eqref{eq:D to D_U projection}.
Since $G_U$ fixes $\bx$, $U$ and $z'_0$, by Lemma \ref{invarianceofmajorant} and \eqref{eq:equivariant fibration}, we can further assume that
\begin{enumerate}
    \item $z_0=\Span_\C\{v_{p+1},\ldots,v_{p+q}\}$ in case A,\ 
    \item $z_0=\Span_\C\{E_1+iF_1,\ldots ,E_n+i F_n\}$ in case B,\ 
    \item $z_0=\Span_\C\{v_{1} j,\ldots,v_{n} j\}$ in case C,
\end{enumerate}
and $z\in F_{z_0} D_{U,z'_0}=\pi^{-1}(z'_0)$. In other words
\[z=\exp_{z_0}(X),\]
for some $X\in N_{z_0} D_{U,z'_0}$. It is well-known that (see Section 3 of Chapter IV of \cite{Helgason}) 
\[\exp_{z_0}(X)=\exp(X) z_0,\]
where we identify $N_{z_0} D_{U,z'_0}$ as a subspace of $\mathfrak{p}_0$ and $\exp$ is the exponential map of the group $G$. From now on we assume that  $z=\exp(X)z_0$ with $X\in N_{z_0} D_{U,z'_0}\subset \mathfrak{p}_0$. By Lemma 
\ref{invarianceofmajorant} we have
\begin{equation}\label{eq:Mz and Mz_0}
    M(z,\bx)=M(z_0,\exp(-X)\bx)=\sum_{j=1}^m ||\exp(-X)\vec{x}_j||_{z_0}^2.
\end{equation}

\subsection{Theorem \ref{estimateofgaussian} in case A}
We assume $n=p+q$ in this subsection.
The  theorem will be a consequence of Lemma \ref{closediag} through \ref{twobounds}.  Let $\Herm_n$ be the set of $n\times n$ Hermitian matrices with values in $\C$. For a matrix $A$, its norm is defined by 
\[\|A\|^2=\mathrm{tr}(A A^*).\]
By our choice of  $z_0$, $\{v_1,\ldots,v_{p+q}\}$ is an orthonormal basis for $(V,(,)_{z_0})$. Hence the norm $\|\cdot\|_{z_0}$ is the standard one with respect to the basis.

\begin{lemma} \label{closediag}
Let $A \in \Herm_n$ and $\epsilon >0$ be given.  Then there exists $\delta$ depending on $\epsilon$ and $A$ such that for any $B \in \Herm_n$ with $\|A -B\| < \delta$ there exist
$R,S \in \Uni(n)$ such that $R A R^{-1}$ and $SBS^{-1}$ are diagonal with $ \|R A R^{-1} - SBS^{-1}\| < \epsilon$ and $\|R - S\| < \epsilon$.  
\end{lemma}
\begin{proof}
For $U\in \rU(n)$, the statement in the lemma is true for $A$ if and only if it is true for $UAU^{-1}$ (with the same $\epsilon$ and $\delta$). Hence without lost of generality we can assume we can assume $A$ is diagonal of the form 
\begin{equation} \label{blockdiagonalmatrix}
A=
\begin{pmatrix} \lambda_1 I_{r_1} & 0 & 0  & \cdots  & 0 & 0 \\                     
0 & \lambda_2 I_{r_2} & 0 & \cdots & 0  &0 \\
\\
 0& 0 & 0 & \cdots & 0 & \lambda_k I_{r_k}
\end{pmatrix},
\end{equation}
 $\lambda_1,\lambda_2,\ldots \lambda_r$ are distinct eigenvalues of $A$ and $r_1+r_2+\cdots+r_k=n$.

First we assume that $\lambda_1,\lambda_2,\ldots \lambda_r$ are all nonzero.

Define a Lie subalgebra $\mathfrak{u}_A\subseteq \mathfrak{u}(n)$ to be the set of matrices of the form 
\begin{equation}
X=
\begin{pmatrix} 0_{r_1} & X_{1 2}   & \cdots   & X_{1 k} \\                     
-X^*_{1 2} & 0_{r_2}  & \cdots & X_{2 k} \\
\\
 -X^*_{1 k}& -X^*_{2 k} &  \cdots  & 0_{r_k}
\end{pmatrix}.
\end{equation}
We define a map
\[\phi: \mathfrak{u}_A \times \Herm_{r_1}\times \cdots \times\Herm_{r_k}\rightarrow \Herm_n\]
by the formula
\begin{equation} 
\phi(X,m_1, \cdots, m_k) =
\exp{(X)}A 
\begin{pmatrix} \exp( { m_1 }) & 0   & \cdots   & 0 \\                     
0 & \exp( { m_2})  & \cdots   &0 \\
\\
 0& 0  & \cdots & \exp ({m_k)}
\end{pmatrix} \exp{(-X)}.
\end{equation}
Then the differential of $\phi$ at $(0,0,\ldots,0)$
\[
d \phi_{0} (X,m_1, \cdots, m_k)= A\begin{pmatrix} m_1  & 0  & \cdots   & 0 \\                     
0 & m_2  & \cdots   &0 \\
\\
 0& 0  & \cdots  &  m_k
\end{pmatrix} + [X,A],
\]
which in turn is equal to 
\[
\begin{pmatrix} \lambda_1 m_1  & (\lambda_2-\lambda_1)X_{1 2}  & \cdots   & (\lambda_k-\lambda_1)X_{1 k} \\                     
(\lambda_2-\lambda_1)X^*_{1 2} & \lambda_2 m_2  & \cdots   & (\lambda_k-\lambda_2) X_{2 k} \\
\\
 (\lambda_k-\lambda_1)X^*_{1 k}& (\lambda_k-\lambda_2)X^*_{2 k}  & \cdots  &  \lambda_k m_k
\end{pmatrix}.
\]
Because we assume that $\lambda_1,\lambda_2,\ldots \lambda_k$ are distinct and nonzero, $d\phi_0$ is an isomorphism from $\mathfrak{u}_A \times \Herm_{r_1}\times \cdots \times\Herm_{r_k}$ to $ \Herm_n$. Hence by inverse function theorem,  $\phi$ is a diffeomorphism from the product of a ball $B(0,\eta)$ of radius $\eta$ around the origin in $\mathfrak{u}_A$ with a ball $B(0,\eta')$ of radius $\eta'$ around the origin of $ \Herm_{r_1}\times \cdots \times \Herm_{r_k}$ to a neighborhood $U(\eta,\eta')$ of $A$ in $\Herm_n$. 

For a given $\epsilon$, shrink the size of $\eta$ and  $\eta'$ if necessary such that  
\begin{equation} \label{choiceeta}
X\in \mathfrak{u}_A \text{ and } ||X||<\eta \Rightarrow ||I-\exp(X)||<\epsilon,
\end{equation} 
and 
\begin{equation} \label{choiceetaprime}
Y\in \Herm_{r_1}\times \cdots \times \Herm_{r_k} \text{ and }\|Y\| \leq \eta' \Rightarrow \|A - A\exp{(Y)}\| < \epsilon.
\end{equation}
Choose $\delta$ such that $B(A,\delta) \subset U(\eta, \eta')$. Suppose $B \in B(A,\delta)$.
Since $B \in U(\eta,\eta')$  we have a unique expression 
$$B = \exp{(X)} A \exp{(Y)} \exp{(-X)} \text{ with } X \in B(0,\eta) \text{ and } Y \in B(0,\eta').$$
Put $R_1 =I$ and $S_1 = \exp{(-X)}$ so
$R_1 A R_1^{-1}=A$ is diagonal and $S_1 B S_1^{-1} = A \exp{(Y)}$ is block diagonal of the form
\begin{equation} S_1 B S_1^{-1}=
\begin{pmatrix} B_{11} & 0   & \cdots   & 0 \\                     
0 & B_{22} & \cdots   &0 \\
\\
 0& 0  & \cdots  &  B_{kk}
\end{pmatrix},
\end{equation}
where $B_{ii}$ is of size $ r_i\times r_i$.

By the above  choices of $\eta$ (equation \eqref{choiceeta}), $\eta'$ (equation \eqref{choiceetaprime}) and $\delta$, it is clear that we have
\begin{equation}\label{firstdiagonalinequality}
   \|A- B\| < \delta \Rightarrow \|R_1 - S_1\|< \epsilon \text{ and } \|R_1 A R_1^{-1} - S_1 B S_1^{-1} \| < \epsilon.
\end{equation}

Now there is a block diagonal unitary matrix
\[R=
\begin{pmatrix} R_{11} & 0   & \cdots   & 0 \\                     
0 & R_{22} & \cdots   &0 \\
\\
 0& 0  & \cdots  &  R_{kk}
\end{pmatrix},
\]
where $R_{ii}\in \rU(r_i)$ such that $ R S_1 B S_1^{-1} R^{-1}$  is diagonal. Notice that $RA=AR$, hence $R A R^{-1}=A$ is also diagonal.

Let $S=RS_1$. Since $R$ is unitary, by \eqref{firstdiagonalinequality} we have 
\[||R-S||=||R(R_1-S_1)||< \epsilon,\]
and
\[||RAR^{-1}-S B S^{-1}||=||R(A-S_1 B S_1^{-1})R^{-1} ||< \epsilon.\]
The lemma is now proved for $A$ a block diagonal matrix of the form \eqref{blockdiagonalmatrix} and $\lambda_1,\ldots, \lambda_k$ are nonzero. 

In general we can choose a $\lambda$ such that $A'=A+\lambda I_n$ does not have zero eigenvalue. By the previous argument, there are $B'\in \Herm_n$ and $R,S \in \rU(n)$ such that $||R-S||< \epsilon $ and $||RA' R^{-1}-S B' S^{-1}||< \epsilon$. Now let $B=B'-\lambda I_n$. The lemma is now proved.
\end{proof}

Recall (see Section \ref{varphiUni}) that $\mathfrak{p}_0$ are Hermitian matrices of the form 
\[\mathfrak{p}_0=\left\{\left(\begin{array}{cc}
    0 & A \\
    A^* & 0
\end{array}\right)| A\in M_{p\times q}(\C)\right\},\]
and the tangent space $N_{z_0} D_{U,z'_0}$ of the fiber $F_{z_0} D_{U,z'_0}$ at $z_0$ can be identified with a subspace of $\mathfrak{p}_0$ given by equation \eqref{normalfibreUpq}.

Let $X \in N_{z_0} D_{U,z'_0}$ so in particular $X\in\mathfrak{p}_{z_0}$ and is Hermitian with respect to $(,)_0$. Let $\tilde{v}_1,\cdots \tilde{v}_n$ be an orthonormal basis for $(\ ,\ )_{z_0}$ of $V$ consisting of eigenvectors of $X$. Then
\[X (\tilde{v}_k) = \lambda_k \tilde{v}_k, 1 \leq k \leq n.\]

Suppose 
\[\vec{x}_j=\sum_{k=1}^n a_j^k \tilde{v}_k\]
for $1\leq k \leq n$. Then 
\[\|\vec{x}_j\|_{z_0}^2 = \sum_{\lambda_k} \|p_{\lambda_k}(\vec{x}_j)\|_{z_0}^2 = \sum_{k=1}^n |a_j^k|^2,
\]
where the summation is over all eigenvalues $\lambda_k$ of $X$ and $p_{\lambda_k}$ is the orthogonal projection with respect to the metric $(\ ,\ )_{z_0}$ onto the eigenspace of eigenvalue $\lambda_k$. 
\begin{remark}
When it is necessary to distinguish to which $X \in N_{z_0} D_{U,z'_0}$ the numbers $a_j^k$ and $\lambda_k$ belong, we will write 
$a_j^k(X)$ and $\lambda_k(X)$.
\end{remark}

\begin{lemma}\label{gxnorm}
We have
\[( \exp{(- tX)} \vec{x}_j, \exp{(-tX)} \vec{x}_j)_{z_0}  = 
\sum_{\lambda_k} \|p_{\lambda_k}(\vec{x}_j)\|_{z_0}^2  \exp{( - 2 \lambda_k t)},\]
where the sum is over all eigenvalues of $X$.
\end{lemma}

\begin{proof}
Since $\{\tilde{v}_k:1\leq k \leq n\}$ is an orthonormal basis for $V$ we have
\begin{align*}
( \exp{(-tX)} \vec{x}_j, \exp{(-tX)} \vec{x}_j)_{z_0}  = & \sum_{i=1}^n |(\tilde{v}_i,(\exp{-(tX)} \vec{x}_j)_{z_0}|^2 \\
=&  \sum_{i=1}^n |\big(\tilde{v}_i,\exp{(-tX)} (\sum_{k=1}^n a_j^k \tilde{v}_k) \big)_{z_0}|^2\\ 
= & \sum_{i=1}^n |(\tilde{v}_i, \sum_{k=1}^n a_j^k \exp{(-\lambda_k t)}\tilde{v}_k)_{z_0}|^2  \\
 =& \sum_{i=1}^n |a_j^i|^2 \exp{( - 2 \lambda_i t)}.
\end{align*}
\end{proof}
We now define 

\begin{equation}\label{definitionoff}
f(X)= -\sum_{j=1}^{r+s} \sum_{\substack{i=1 \\ \lambda_i (X) <0} }^n |a_j^i|^2(X) \lambda_i(X)=
      - \sum_{j=1}^{r+s} \sum_{\lambda (X) <0 }
    \|p_{ \lambda(X)}(\vec{x}_j)\|^2_{z_0} \lambda(X).
\end{equation}

Since all the terms in the sum defining $f(X)$ are nonnegative it follows that $f(X) =0$ if and only all the term in the sum are zero. 
By Lemma \ref{closediag}, we can prove the following.
 \begin{lemma} \label{continuous}
 $f(X)$ is continuous on $\mathfrak{p}_0\cong T_{z_0} D$.
 \end{lemma}
\begin{proof}
 Let $X_1 \in \mathfrak{p}_0$. Then for any $X \in \mathfrak{p}_0$ we have
  \begin{equation}
  |f(X) - f(X_1)| = |\sum_{j=1}^{r+s} \sum_{\substack{i=1 \\ \lambda_i(X) <0}}^n |a_j^i|^2(X) \lambda_i (X)- \sum_{j=1}^{r+s} \sum_{\substack{i=1 \\ \lambda_i(X_1) <0}}^n |a_j^i|^2(X_1) \lambda_i(X_1)|.
    \end{equation}
Let $\epsilon > 0$ be given. 
Apply Lemma \ref{closediag} with  $ A =X_1$ to find $\delta$ such that whenever $ \|X - X_1\| < \delta $, there exist unitary matrices $R,S$ such that $R X_1 R^{-1}$ and $S X S^{-1}$ are diagonal and
$$ \|R-S\| < \epsilon \text{ and } \| RAR^{-1} - SBS^{-1}\|< \epsilon . $$

But $\|R-S\| < \epsilon$ implies that suitably chosen  eigenvectors  of $A$ and $B$ are close. More precisely, if $\tilde{v}_i(X)$ (resp. $\tilde{v}_i(X_1)$), $1 \leq i \leq n$ is the  eigenvector of $X$ (resp. $X_1$), corresponding to the eigenvalue $\lambda_i(X)$ (resp. $\lambda_i(X_1) $) which is the $i$-th row of $R$ (resp. $i$-th row of $S$), we have
$$\epsilon^2 > \|R- S\|^2 = \sum_{i=1}^n \|\tilde{v}_i(X) - \tilde{v}_i(X_1)^2\|_{z_0}^2 \Rightarrow \|\tilde{v}_i(X)-\tilde{v}_i(X_1)\|_{z_0}^2 <\epsilon ^2, 1 \leq i \leq m. $$
Hence,  for all $i,j, 1 \leq i \leq n \text{ and } 1 \leq j \leq n$ we have 
\begin{equation*}
| a_j^i(X) - a_j^i(X_1)|^2 = | (\vec{x}_j, \tilde{v}_i(X) - v_i(X_1))_{z_0} |^2 \leq \|\vec{x}_j\|_{z_0}^2 \ \|\tilde{v}_i(X) - \tilde{v}_i(X_1)\|_{z_0} ^2 < \|\vec{x}_j\|_{z_0}^2 \ \epsilon^2.
\end{equation*}
Similarly $|a_j^i(X_1)|^2< \|\vec{x}_j\|_{z_0}^2 $.
Also 
$$\| RAR^{-1} - SBS^{-1}\|^2 = \sum_{i=1}^m(\lambda_i(X) - \lambda_i(X_1))^2,$$
consequently 
$$\| RAR^{-1} - SBS^{-1}\|^2< \epsilon^2 \Rightarrow  \sum_{i=1}^n(\lambda_i(X) - \lambda_i(X_1))^2 < \epsilon \Rightarrow  (\lambda_i(X) - \lambda_i(X_1))^2  < \epsilon^2, 1 \leq i \leq n.$$
Hence, for all $i, 1 \leq i \leq n$ we have 
$$|\lambda_i(X) - \lambda_i(X_1)|  < \epsilon.$$
Since $X$ is fixed, we assume that
\[\lambda_i(X)\leq M, \lambda_i(X_1) \leq M\]
for $1 \leq i \leq n$. 

  Now  using the identity $| ab - a'b' | \leq |b||a - a'| + |a'||b-b'|$ we obtain
  $$
  ||a_j^i|^2(X) \lambda_i (X) - |a_j^i|^2(X_1) \lambda_i(X_1)| 
  \leq |\lambda_i (X)||a_j^i|^2(X) - |a_j^i|^2(X_1)| +
  |a_j^i|^2(X_1)||\lambda_i (X) - \lambda_i(X_1)|.$$  
  Since $|\lambda_i (X)| \leq M \text{ and }  |a_j^i|^2(X_1) \leq \|\vec{x}_j\|_{z_0}^2$
  we have 
  \begin{align*}
  ||a_j^i|^2(X) \lambda_i (X) - |a_j^i|^2(X_1) \lambda_i(X_1)| \leq 
& M ||a_j^i|^2(X) - |a_j^i|^2(X_1)| + \|\vec{x}_j\|_{z_0} ^2|\lambda_i (X) - \lambda_i(X_1)|\\ \leq &M||\vec{x}_j||_{z_0}^2 \  \epsilon^2 + ||\vec{x}_j||_{z_0} ^2\epsilon=  ||\vec{x}_j||_{z_0} ^2 (M\epsilon^2 + \epsilon).
\end{align*}

Suppose that the strictly negative eigenvalues of $X$ are
$\lambda_1(X), \cdots, \lambda_k(X)$ and the strictly negative eigenvalues of $X_1$ are $\lambda_1(X_1), \cdots, \lambda_{\ell}(X_1)$. We assume $k > \ell$. The case $k = \ell$ is easier (in this case, we have only the first sum in Eqation \eqref{eigenvaluetail} below)  and the case $k<\ell$ can be treated in a manner symmetrical to that of   the case $k> \ell$. 

We have
\begin{equation}\label{eigenvaluetail}
  |f(X) - f(X_1)| = |\sum_{j=1}^{r+s} \sum_{i=1}^{\ell} \bigg( |a_j^i|^2(X) \lambda_i (X)-  |a_j^i|^2(X_1) \lambda_i(X_1)\bigg)
  -\sum_{j=1}^{r+s} \sum_{i=\ell+1}^{k} |a_j^i|^2(X) \lambda_i (X)|
 \end{equation}
The first sum is clearly majorized by $ \ell 
\sum_{j=1}^{r+s}  \|\vec{x}_j\|_{z_0} ^2 ( 2M \epsilon^2 + \epsilon )$ using the inequality immediately above. To majorize the second sum
we note that 
$$ \ell < i \leq k \Rightarrow \lambda_i(X) < 0 \text{ and } \lambda_i(X_1) \geq 0.$$
Hence  $|\lambda_i(X) - \lambda_i(X_1)| = -\lambda_i(X) +\lambda_i(X_1)$.  Note that each of the two terms is positive.
But $$|\lambda_i(X) - \lambda_i(X_1)| < \epsilon \Rightarrow -\lambda_i(X) +\lambda_1(X_1) < \epsilon \Rightarrow -\lambda_i(X) < \epsilon.$$
Hence the second summand is majorized by $(k - \ell) \sum_{j=1}^{r+s} \|\vec{x}_j\|_{z_0} ^2 \epsilon$.

Lemma \ref{continuous} follows. 

\end{proof}

Let $S(N_{z_0} D_{U,z'_0})$ be the unit sphere of $N_{z_0} D_{U,z'_0} $, then we have

\begin{lemma} \label{boundedbelow}
 $f(Y)$ does not take the value zero on $S(N_{z_0} D_{U,z'_0})$. As $S(N_{z_0} D_{U,z'_0})$ is compact and $f(Y)\geq 0$, there exists $C >0$ so that 
 $$f(Y) \geq C, \ \forall Y \in S(N_{z_0} D_{U,z'_0}).$$
 \end{lemma}
 \begin{proof}
Assume $Y\in S(N_{z_0} D_{U,z'_0})$ and $f(Y) =0$. Suppose $v$ is an eigenvector of $Y$ corresponding to a strictly negative eigenvalue so
 $$Y(v) = \lambda v, \lambda <0  .$$
 
 Then 
 \[f(Y) = 0 \Rightarrow \|\p_{\lambda}(\vec{x}_j)\|_{z_0}^2  = 0 \Rightarrow  (\vec{x}_j,v)_{z_0} = 0 , 1\leq j \leq r+s.\]

Let $U_+ = \Span \{ v_1,\ldots,v_r \}$,
 $U_- = \Span \{ v_{p+1},\ldots,v_{p+s} \}$,
 $U^\bot_+ = \Span \{ v_{r+1},\ldots,v_p \}$ and 
 $U^\bot_- = \Span\{ v_{p+s+1},\ldots,v_{p+q} \}$.
Then $v \perp U_+\oplus U_-$ as $\Span\{\vec{x}_1,\ldots,\vec{x}_{r+s}\}=U_+\oplus U_-$.
For any $u \in V$, we may write
$u= (v_1, w_1, v_2,w_2)$ with $v_1 \in U_+,w_1 \in U^\bot_+, v_2 \in U_-, w_2 \in U^\bot_-$.  Then in this representation we have 
$$v = ( 0, w_1, 0, w_2) \text{ with } w_1 \neq 0 \text{ or } w_2 \neq 0, $$
and 
$$Y(v) = \lambda v = ( 0,\lambda w_1, 0, \lambda w_2).$$
But since $Y \in N_{z_0} D_{\bx,z'_0}$ we  have (see equation (\ref{normalfibreUpq}))
$$ Y = \begin{pmatrix} 0 & 0 & a & b\\
 0 & 0 & c & 0\\
 a^* &  c^* & 0 & 0\\
  b^* &0 &0 &0
\end{pmatrix}.$$
Hence 
 $$Y(v) = (bw_2,0, c^* w_1,0) = \lambda v = (0,\lambda w_1,0, \lambda w_2).$$
Since $\lambda < 0$ the equation immediately above implies  $w_1 \text{ and } w_2 = 0 $, a contradiction.
\end{proof}

\begin{lemma} \label{twobounds}
Let $\bx=(\vec{x}_1,\ldots,\vec{x}_{m})$ and $U=\Span\{\bx\}$ and suppose $X\in S(N_{z_0}D_{U,z'_0})$. Then there exists  strictly positive numbers  $b, c$ depending only on $\bx$ and $z'_0$ and a negative eigenvalue $\lambda$ of $X$ such that for some $j, \text{ with } 1 \leq j \leq m$, such that
\[|\lambda| = - \lambda \geq b  \text{ and } \|p_{\lambda}(\vec{x}_j)\|_{z_0}^2 \geq c.\]
\end{lemma}
\begin{proof}
Since
\[f(X)= -\sum_{j=1}^{r+s} \sum_{\begin{subarray}{c}i=1 \\\lambda_i (X) <0 \end{subarray}}^n |a_j^i|^2(X) \lambda_i(X)\]
is   bounded  below by $C$, at least one of the terms in the sum  is  bounded below by  $c=\frac{C}{N}$, where $N$ is the number of terms in the sum ($N\leq n(r+s)$). Suppose this term is $-|a_j^i|^2 \lambda_i$.  Hence
$$|a_j^i|^2| \lambda_i| = -|a_j^i|^2 \lambda_i  \geq c \text{ for some } i,j.$$
But since  $\sum_{i=1}^n |a_j^i|^2 =\|\vec{x}_j\|_{z_0}^2$ it follows that 
$$|a_j^i|^2 \leq \|\vec{x}_j\|_{z_0}^2, 1 \leq i \leq n, 1 \leq j \leq r+s .$$
Hence 
$$|\lambda_i| \geq \frac{c}{\|\vec{x}_j\|_{z_0}^2}.$$
We put
\[b=\frac{c}{\|\vec{x}_j\|_{z_0}^2}.\]

Since $\|X\|=\sum_{i=1}^n \lambda_i^2=1$, it follows that
$$|\lambda_i| \leq 1,$$
hence
\[\|p_{\lambda_i}(\vec{x}_j)\|_{z_0}^2\geq|a_j^i|^2 \geq c.\]
\end{proof}

\noindent{\bf Proof of Theorem \ref{estimateofgaussian} in Case A:}
We assume $X\in S(N_{z_0} D_{\bx,z'_0})$, $z=\exp(Xt)z_0$ and $g=\exp(Xt)$. Then 
\[M(z,\bx)= \sum_{j=1}^{r+s} ( \exp{(-tX)} \vec{x}_j, \exp{(- tX)} \vec{x}_j)_{z_0}=\sum_{j=1}^{r+s} \sum_{i=1}^n |a^i_j|^2  \exp{( - 2 \lambda_i t)}.\]
But all the terms in the sum immediately above are nonnegative and we have  proved in Lemma \ref{twobounds} that one of them is minorized by
$c  \exp{( 2 b t)}$. Hence the entire sum is also minorized by $c  \exp{( 2 b t)}$ and we obtain 
$$M(z,\bx) \geq c \exp{(2  b t)}.$$

Since 
$$d( D_{U,z'_0}, z) =t,$$
 Theorem \ref{estimateofgaussian} is proved. 
\qedsymbol

\subsection{Proof of Theorem \ref{estimateofgaussian} in case B}\label{estimateofgaussianSp}
One way to proceed is to use the seesaw dual pair:
\begin{equation}
		    \xymatrixrowsep{0.3in}
		\xymatrixcolsep{0.3in}
		\xymatrix{ \Uni(n,n) &\Orth(2r,2r)\\
			\Sp(2n,\R)\ar[u] \ar[ur] & \Uni(r,r) \ar[u]\ar[ul]}
\end{equation}
together with a relation between $D_{U,z'_0}$ and $\tilde{D}_{U\otimes \C,z'_0}$
to reduce a substantial part of the problem to the unitary dual pair case. However we use a direct approach here instead. We proceed quickly by omitting the proofs that are similar to Case A. 

We know that $\Sp(2n,\R)\hookrightarrow \Uni(n,n)$ and the symmetric space $D$ embedds into $\tilde{D}$ (see Section \ref{Spcocylesection}). Recall that we have assumed that
\[\Span_\R\{\bx\}=\Span_\R\{E_1,\ldots,E_r,F_1,\ldots,F_r\} \text{ and }z_0=\Span_\C\{E_1+iF_1,\ldots ,E_n+i F_n\}.\]
The first condition is equivalent to (recall from Section \ref{seesawonesection} for our convention of the basis)
\begin{equation}\label{spanofy}
    \Span_\C\{\bx\}=\Span_\C\{v_1,\ldots,v_r,v_{n+1},\ldots,v_{n+r}\}.
\end{equation}
Also recall equation \eqref{eq:Mz and Mz_0}.
Since $T_{z_0} D \subset T_{z_0} \tilde{D}$, for $X \in T_{z_0} D$ we can define $f(X)$ as in equation \eqref{definitionoff} (see the paragraphs before equation \eqref{definitionoff} for the definition of $\lambda(X)$ and $p_{\lambda(X)}$):
\begin{equation}
    f(X)= - \sum_{j=1}^{2r} \sum_{\begin{subarray}{c}i=1 \\\lambda_i(X) <0 \end{subarray}}^n 
    \|p_{ \lambda_i(X)}(\vec{x}_j)\|_{z_0}^2 \lambda_i(X).
\end{equation}
By Lemma \ref{continuous}, we know $f$ is continuous on $T_{z_0} D$.

\begin{lemma}\label{boundedbelowSp}
$f(Y)$ does not take zero value on $S(N_{z_0}D_{{U},z'_0})$. As $S(N_{z_0}D_{U,z'_0})$ is compact there exists $C>0$ such that 
\[f(Y)\geq C, \forall Y\in S(N_{z_0}D_{{U},z'_0}).\]
\end{lemma}
\begin{proof}
Assume $Y\in S(N_{z_0}D_{{U},z'_0}) $ and $f(Y) =0$. 
Suppose $v\in V$ is an eigenvector of $Y$ corresponding to a strictly negative eigenvalue so
\[Y(v) = \lambda v, \lambda <0  .\]
Then 
\[f(Y) = 0 \Rightarrow \|\p_{\lambda}(\vec{x}_j)\|_{z_0}^2  = 0 \Rightarrow  (\vec{x}_j,v)_{z_0} = 0 \]
for $1\leq j \leq 2r$. Under the basis $\{v_1,\ldots,v_n,v_{n+1},\ldots,v_{2n}\}$ and the assumption \eqref{spanofy}, we have 
\[v=(0,w_1,0,w_2) \text{ with } w_1\neq 0 \text{ or } w_2\neq 0,\]
where $w_1\in \Span_\C\{v_{r+1},\ldots,v_n\}$ and  $w_2\in \Span_\C\{v_{n+r+1},\ldots,v_{2n}\}$.

Recall that for $Y \in N_{z_0} D_{{U},z'_0}$ we  have (see equation (\ref{normalspaceSp}) and equation (\ref{pofSp}))
\[ Y = \begin{pmatrix} 0 & 0 & a & b\\
 0 & 0 & b^t & 0\\
 a^* & \bar{b} & 0 & 0\\
 b^* &0 &0 &0
 \end{pmatrix},\]
where $a=a^t$. Hence 
 \[Y(v) = (bw_2,0, \bar{b} w_1,0) = \lambda v = (0,\lambda w_1,0, \lambda w_2).\]
Since $\lambda < 0$ the equation immediately above implies  $w_1=0 $ and $ w_2 = 0 $, a contradiction.
\end{proof}
 
With Lemma \ref{boundedbelowSp}, the conclusion of Lemma \ref{twobounds} holds for $Y\in N_{z_0} D_{U,z'_0}$ as well. Hence the rest of the proof of Theorem \ref{estimateofgaussian} in case B is the same as that of case A.
\qedsymbol

\subsection{Proof of Theorem \ref{estimateofgaussian} in case C}\label{estimateofgaussianOstar}
We know that $\rO^*(2n)\hookrightarrow \Uni(n,n)$ and the symmetric space $D$ embedds into $\tilde{D}$ (see section \ref{Ostarcocyclesection}). Recall that we have assumed that  
\begin{equation}\label{spanofxOstarcase}
    \Span_{\H}\{\bx \}=\Span_{\H}\{v_1,\ldots,v_r\} \text{ and }z_0=\Span_\C\{v_{n+1},\ldots ,v_{2n}\}.
\end{equation}
Also recall equation \eqref{eq:Mz and Mz_0}.
Since $T_{z_0} D \subset T_{z_0} \tilde{D}$, for $X \in T_{z_0} D$ we can define $f(X)$ as in equation \eqref{definitionoff} (see the paragraphs before equation \eqref{definitionoff} for the definition of $\lambda(X)$ and $p_{\lambda(X)}$):
\begin{equation}
    f(X)= - \sum_{j=1}^{r} \sum_{\begin{subarray}{c}i=1 \\\lambda_i(X) <0 \end{subarray}}^n 
    \|p_{ \lambda_i(X)}(\vec{x}_j)\|_{z_0}^2 \lambda_i(X).
\end{equation}
By Lemma \ref{continuous}, we know $f$ is continuous on $T_{z_0} D$.
\begin{lemma}\label{boundedbelowOstar}
$f(Y)$ does not take zero value on $S(N_{z_0}D_{{U},z'_0})$. As $S(N_{z_0}D_{{U},z'_0})$ is compact there exists $C>0$ such that 
\[f(Y)\geq C, \ \forall Y\in S(N_{z_0}D_{{U},z'_0}).\]
\end{lemma}
\begin{proof}
 Assume $Y \in N_{z_0} D_{{U},z'_0}$ identified with a subspace of $\mathfrak{p}_0$. Since $Y\in \mathfrak{o}^*(2n)$, it commutes with right multiplication by $j$ on V. In particular if $\lambda$ (recall that $\lambda$ must be real since $Y$ is Hermitian) is an eigenvalue of $Y$ and $V_\lambda$ is the $\lambda$-eigenspace of $Y$, $V_\lambda$ is preserved by right multiplication by $j$.
 
Suppose $V_\lambda$ is the $\lambda$-eigenspace of $Y$ corresponding to a strictly negative eigenvalue $\lambda$. Assume $f(Y)=0$, then
\[f(Y) = 0 \Rightarrow \|\p_{\lambda}(\vec{x}_\alpha)\|_{z_0}^2  = 0 \Rightarrow  \vec{x}_\alpha\bot V_\lambda ,\]
for $1\leq \alpha \leq r$. As $V_\lambda$ and $(,)_{z_0}$ are preserved by right multiplication by $j$, the above implies 
\[\vec{x}_\alpha j \bot \ V_\lambda\]
for $1\leq \alpha \leq r$. Under the basis $\{v_1,\ldots,v_n,v_1 j,\ldots, v_n j\}$ and the assumption \eqref{spanofxOstarcase}, we then have 
\[v=(0,w_1,0,w_2) \text{ with } w_1\neq 0 \text{ or } w_2\neq 0,\]
where $w_1\in \Span_\C\{v_{r+1},\ldots,v_n\}$ and  $w_2\in \Span_\C\{v_{r+1} j,\ldots,v_{n} j\}$.
  
Recall that for $Y \in N_{z_0} D_{{\bf U},z'_0}$ we  have (see equation (\ref{normalspaceOstar}) and equation (\ref{pofOstar}))
\[ Y = \begin{pmatrix} 0 & 0 & a & b\\
 0 & 0 & -b^t & 0\\
 a^* & -\bar{b} & 0 & 0\\
 b^* &0 &0 &0
 \end{pmatrix},\]
where $a=-a^t$. Hence 
 \[Y(v) = (bw_2,0, -\bar{b} w_1,0) = \lambda v = (0,\lambda w_1,0, \lambda w_2).\]
Since $\lambda < 0$ the equation immediately above implies  $w_1=0 $ and $ w_2 = 0 $, a contradiction.
\end{proof}

With Lemma \ref{boundedbelowSp}, the conclusion of Lemma \ref{twobounds} holds for $Y\in N_{z_0} D_{{U},z'_0}$ as well. Hence the rest of the proof of Theorem \ref{estimateofgaussian} in case C is the same as that of case A.
\qedsymbol

\subsection{Rapid decrease of the cocycles on the fiber $F_{z_0} D_{\bx,z'_0}$}\label{rapiddecreasing}
In this subsection we prove Theorem \ref{rapiddecreasingthm}.
\begin{proposition}
\label{decayofSchwartzfunction}
Let $\bx$, $U$ and $z'_0$ be as in Theorem \ref{rapiddecreasingthm} in each cases.  For any Schwartz function $ \psi \in \mathcal{S}(V^m)$ ($\mathcal{S}(V_0^m) $ in case B) and any constant $\rho>0$, there is a constant $C_\rho$ depending on $\bx$, $z'_0$ and $\psi$ such that 
\begin{equation}
    \psi (g^{-1}\bx)\leq C_\rho \exp(-\rho \, d(D_{U,z'_0}, gz_0 )), \ \forall g\in G.
\end{equation}
\end{proposition}
\begin{proof}
Since $\psi$ is a Schwartz function, for any positive integer $N$, there is a positive constant $C_N$ such that
\[\psi(\bx)\leq \frac{C_N}{(||\bx||_{z_0}^2)^N}.\]
By Theorem \ref{estimateofgaussian}, we know that
\[\|g^{-1}\bx\|_{z_0}^2=\sum_{j=1}^m \|g^{-1}x_j\|_{z_0}^2\geq c\cdot \exp(2b\cdot d(D_{U,z'_0}, gz_0 )).\]
Hence
\[\psi(g^{-1}\bx) \leq \frac{C_N}{c}\exp(-2Nb \cdot d(D_{U,z'_0}, gz_0 )).\]
We fix a $N>\frac{\rho}{2b}$ and let $C_\rho=\frac{C_N}{c}$, the theorem is proved.
\end{proof}
\begin{remark}
Suppose $\psi_{g'}=\omega(g')\psi_0$ for $g'\in \tilde{G}'$. Since $\tilde{G}'$ acts smoothly on $\mathcal{S}(V^m)$, the constant $C_N$ and $C_\rho$ in the above proof for $\psi_{g'}$ depends continuously on $g'$.
\end{remark}

\noindent {\bf Proof of Theorem \ref{rapiddecreasingthm}:}
Fix a $g\in G$ such that $ gz_0=z$. 
Assume
\[\psi= \sum_{I}^d \psi_{I}  \Omega_{I},\]
where $\psi_{I}\in\mathcal{S}(V^m)$ ($\mathcal{S}(V_0^m) $ in case B) are polynomial, $\Omega_{I} \in \bigwedge^\bullet \mathfrak{p}^*$ and are mutually perpendicular. Then
\[\tilde{\psi}(z,g',\bx)=\sum_{I}(\omega(g')\psi_I)(g^{-1}\bx)L^*_{g^{-1}}(\Omega_I).\]
Weil representation preserves the space of Schwartz functions, hence $\omega(g')\psi_I\in \mathcal{S}(V^m)$ ($\mathcal{S}(V_0^m) $ in case B).

By Proposition \ref{decayofSchwartzfunction}, we know that for any $I$ and $\rho>0$, there is a constant $C^I_\rho>0$ such that 
\[ (\omega(g')\psi_I )(g^{-1}\bx)\leq C^I_\rho \exp(-\rho \, d(D_{U,z'_0}, z )).\]
Since the left action of $G$ on $D$ is isometric, we know that
\[||L^*_{g^{-1}}(\Omega_I)||=||\Omega_I||.\]
Hence define 
\[C'_\rho=\sqrt{\sum_I (C^I_\rho)^2||\Omega_I||^2}.\]
We know that $|| \tilde{\psi}_{\infty}(z,g',\bx)||\leq C'_{\rho} \exp({-\rho d(D_{U,z'_0}, z ) })$.

By the remark after Theorem \ref{decayofSchwartzfunction} each constant $C_\rho^I$ depends continuously on $g'$, hence $C'_\rho$ also depends continuously on $g'$.
\qedsymbol

\section{Asymptotic evaluations of fiber integrals}\label{methodofLaplacesection}
We go back to the settings of Section \ref{thomformsection}. The goal of this section is to prove Theorem \ref{genericnonzero}.
We want to compute the fiber integral $\kappa(g',\beta)$ defined in \eqref{kappaintegral} by
\[\kappa(g',\beta)=\int_{F_{z_0} D_{U,z'_0}} \tilde{\varphi}_\infty(z,g',\bx),\]
where $(\bx,\bx)=\beta$, $U=\Span_{B_{v_1}}\{\bx\}$, $z'_0\in D(U)$ and $z_0\in D_{U,z'_0}$.  Our goal is to prove Theorem \ref{genericnonzero}.
Recall that $\varphi_\infty=\varphi\otimes\prod_{v\neq v_1} \varphi_v$, where $\varphi$ is the cocycle specified after equation \eqref{definitionofvarphiinfty}, $\varphi_v$ is the Gaussian function of $V^m_{v}$ and $v$ is an Archimedean place for the number field $k$. 
We know that $\varphi_v(g_v,g'_v)$ only depends on $g'_v$ and is nonzero for $v\neq v_1$. We have
\[\kappa(g',\beta)=\int_{F_{z_0} D_{U,z'_0}} \tilde{\varphi}(z,g'_{v_1},\bx)\cdot \prod_{v\in S_\infty, v\neq v_1} \varphi_{v}(g'_v).\]
So in order to prove the integral $\kappa(g',\beta)$ is nonzero, it suffices to compute the following rescaled integral
\begin{equation}
    \int_{F_{z_0} D_{U,z'_0}} \tilde{\varphi}(z,g',\bx)
\end{equation}
with $g'\in G'_{v_1}$, where $\tilde{\varphi}$ is defined in \eqref{tildepsi}. So from now on in this section, we change our notation and let $G=G_{v_1}$, $G'=G'_{v_1}$ and $\kappa(g',\beta)=\int_{F_{z_0} D_{U,z'_0}} \tilde{\varphi}(z,g',\bx)$.
Recall that $(G,G')$ can be the following three dual pairs
\begin{enumerate}
	        \item case A: $(\Uni(p,q), \Uni(r+s,r+s))$, \
	        \item case B: $(\Sp(2n,\R), \Orth(2r,2r))$, \
	        \item case C: $(\Orth^*(2n), \Sp(r,r))$. 
\end{enumerate} 
Let $B$ be $\C$ in case A, $\R$ in case B and $\H$ in case C. Let $m$ be $r+s$ in case A, $2r$ in case B and $r$ in case C.
Recall that the group $M'\subset G'$  is 
\[M'=\left\{m'(a)=\left(\begin{array}{cc}
		   a  & 0 \\
		    0 & (a^*)^{-1}
		\end{array}\right)\mid a\in \GL_m(B)\right\}.\]
An element $ (m'(a),\zeta)$ in its double cover acts by \eqref{M'action}. From this we know that 
\begin{equation}
    \kappa(g' (m'(a),\zeta),\beta)=\zeta |a|^{\frac{m}{2}} \kappa (g',a^* \beta a).
\end{equation}
Suppose $\beta$ satisfies the condition of Theorem \ref{fouriercoeff}, namely, $i \beta$ is non-degenerate and is of signature $(r,s)$ in case A. 
By the above formula and Gram-Schmidt process, we can choose $m=m'(a)$ such that $a^* \beta a$ is of the following form:
\begin{enumerate}
    \item the $r+s$ by $r+s$ diagonal matrix with diagonal entries $\{\underbrace{-i,\ldots,-i}_r,\underbrace{i,\ldots,i}_s\}$ in case A, \
    \item the $2r$ by $2r$ matrix $\left(\begin{array}{cc}
                0 & -I_r \\
                I_r & 0
            \end{array}\right)$ in case B,
            \
    \item the $r$ by $r$ $\H$-valued diagonal matrix with diagonal entries $\{\underbrace{-i,\ldots,-i}_r\}$ in case C. 
\end{enumerate}
So from now on we assume $\beta$ is of the above form. 
By translating by appropriate $g\in G$ together with the fact that $\kappa(g',\beta)$ does not depend on $z'_0$ (see Lemma \ref{lem:kappa independent of z'}), we can further assume 
\begin{align}\label{standardbx}
    &\bx=(v_1,\ldots,v_r,v_{p+1},\ldots,v_{p+s}), &z_0=\Span_\C\{v_{p+1},\ldots,v_{p+q}\}, \\ \nonumber
    &\bx=(E_1,\ldots,E_r,F_1,\ldots,F_r),  &z_0=\Span_\C\{v_{n+1},\ldots,v_{2n}\},\\  \nonumber
    &\bx=(v_1,\ldots,v_r), &z_0=\Span_\C\{v_{n+1},\ldots,v_{2n}\}
\end{align}
in the three cases respectively.
Let $a(t)\in \GL_m(B)$ be the scalar matrix $t \cdot Id$.
The exact value of $\kappa(g',\beta)$ is hard to compute in general, instead we approximate $\kappa(g',\beta)$ for $g'=(m'(a(t)),1)$ as $t\rightarrow \infty$. We need the following theorem.     
\begin{theorem}\label{methodoflaplace}
Let $f(x),h(x)$ be smooth functions on $\R^n$. And let $J(t)$ be the integral
        \[J(t)=\int_{\R^n}f(x)e^{-t  h(x)}dx.\]
And we assume that 
\begin{enumerate}
        \item The integral $J(t)$ converges absolutely for all $t >0$.
        \item For every $\epsilon >0$, $\rho (\epsilon)>0$, where
        \[\rho(\epsilon)=\inf\{h(x)-h(0):x\in\R^n,|x-0|\geq \epsilon\}.\]
        \item the Hessian matrix 
        \[A=(\frac{\partial^2 h}{\partial x_i \partial {x_j}})|_{x=0}\]
        is positive definite.
\end{enumerate}
Then we have 
\[J(t)\sim(\frac{2\pi}{t})^{\frac{n}{2}} f(0) \mathrm{det}(A)^{-\frac{1}{2}}\exp[-t h(0)]\]
as $t\rightarrow \infty$.
\end{theorem}
The above theorem is one special case of the so-called method of Laplace. The proof can be found in Section 5 of Chapter IX in \cite{Wong}. To apply Theorem \ref{methodoflaplace}, we choose a base point $z_0\in D_{U,z'_0}$.
Recall that we identify $T_{z_0}D\cong\mathfrak{g}_0 /\mathfrak{k}_0$ with  $\mathfrak{p}_0$. 

\noindent
\textbf{Proof of theorem  \ref{genericnonzero} under assumptions that will be checked later}:

\noindent
Define $\rho:N_{z_0} D_{U,z'_0}\rightarrow F_{z_0} D_{U,z'_0}$ by 
\[\rho: Y \mapsto \exp(-Y) z_0,\]
where we regard $N_{z_0} D_{U,z'_0} $ as a linear subspace of $\mathfrak{p}_0$.
Our strategy is to  apply Theorem \ref{methodoflaplace} to
\begin{equation}\label{J(t)}
    J(t)=\int_{F_{z_0} D_{U,z'_0}} \tilde{\varphi} (z,g'(t),\bx)=\int_{N_{z_0} D_{U,z'_0}}\rho^*(\tilde{\varphi}) (Y ,g'(t),\bx)
\end{equation}
for  $g'(t)=(m'(a(t)),1)$. Recall that 
\[\tilde{\varphi}_\infty(z,g')=(L_{g^{-1}})^*(\omega(g')\varphi)=(L_{g^{-1}})^*(\omega(g')\varphi),\]
where $g z_0=z$. By equation \eqref{varphiSchrodinger}, we have
\[\varphi=\varphi_0 \sum_{i,j=1}^d p_{ij}  \Omega_i \wedge  \bar{\Omega}_j,\]
where $p_{ij}$ are polynomial functions on $ V^m$, $\Omega_i \in \bigwedge^\bullet \mathfrak{p}_+^*$ and
\[\varphi_0(\bx)=\exp(-\pi \sum_{i=1}^m ||\vec{x}_m||_{z_0}^2)\]
with $\bx$ as in \eqref{standardbx}.
Define $h:N_{z_0} D_{U,z'_0}\rightarrow \R$ by 
\begin{equation}\label{h(Y)}
    h(Y)=\pi\cdot  M(\exp(Y)z_0,\bx)=\pi\cdot  M(z_0,\exp(-Y)\bx),
\end{equation}
where $M(z,\bx)$ is the function defined in \eqref{definitionMajorant}. Then we have
\[\varphi_0(\exp(-Y)\bx)=\exp(-h(Y)).\]
We also have (recall \eqref{M'action})
\begin{align}
    &(\rho^*)(\tilde{\varphi})(Y,g'(t),\bx)\\ \nonumber
    =&  t^{e}\exp(-h(Y)t^2)\sum_{i,j} p_{ij}(\exp(-Y)\bx \cdot t ) \rho^*(L^*_{\exp(-Y)}(\Omega_i)\wedge L^*_{\exp(-Y)}(\overline{\Omega}_j)),
\end{align}
where $e$ is $(p+q)(r+s)$ in case A and is $2nr$ in case B and C.
Since $\tilde{\varphi}$ is of Hodge degree $(N,N)$, where $N$ is the dimension of $N_{z_0} D_{U,z'_0}$, each term $\rho^*(L^*_{\exp(-Y)}(\Omega_i)\wedge L^*_{\exp(-Y)}(\overline{\Omega}_j)) $ in the above equation is some function times $\mathrm{dvol}$, the volume form of the Euclidean space $N_{z_0} D_{U,z'_0}$.
After combining terms according to $t$-degree, we have
\begin{equation}\label{finalsummation}
  (\rho^*)(\tilde{\varphi})(Y,g'(t),\bx)= t^{e}\sum_{i} t^i f_i(Y) \exp(-h(Y)t^2) \mathrm{dvol}.  
\end{equation}
Theorem \ref{methodoflaplace} can be applied to compute each 
\[\int_{N_{z_0} D_{U,z'_0}} f_i(Y) \exp(-h(Y)t^2) \mathrm{dvol}.\]
Condition (1) of Theorem \ref{methodoflaplace} is checked in the corollary of  Theorem \ref{rapiddecreasingthm}. Condition (2) and (3) will be checked later in this section.

As we are interested in asymptotic value when $t \rightarrow \infty$, only the highest degree term of $t$ in Equation \eqref{finalsummation} matters if it is nonzero. By Lemma \ref{highestermSchrodinger}, we know that the highest degree term of $ (\rho^*)(\tilde{\varphi})(Y,g'(t),\bx)$ evaluated at $\bx$ and $Y=0$ is nonzero. Hence the asymptotic value of $J(t)$ as $t\rightarrow \infty$ is nonzero and Theorem \ref{genericnonzero} is proved.
\qedsymbol

The rest of the section will be devoted to verifying condition (2) and (3) of Theorem \ref{methodoflaplace} (see Proposition \ref{secondcondition}, Proposition \ref{secondconditioncaseB} and Proposition \ref{secondconditioncaseC}) and computing $J(t)$ as in equation \eqref{J(t)} in each case. We will emphasize on  case A and proceed the other two cases quickly.

\subsection{Case A}
As before we assume \eqref{standardbx}.
Let 
\[U_j=\Span_\C \{v_j\},D_j=D_{U_j},1\leq j \leq p+q.\]
Recall that the definition of $D_j$ does not require a base point (see Remark \ref{definitionofKM}). 
\begin{lemma}\label{intersectionofDj}
\[D_{U,z'_0}=(\bigcap_{j=1}^r D_j)\bigcap (\bigcap_{j=p+1}^{p+s} D_j).\]
\end{lemma}
\begin{proof}
By Proposition \ref{DU=DXcapDY}, we know that
By definition for $1\leq j \leq p$, we have 
\[D_\alpha=\{z\in D\mid  z \subseteq U_\alpha^\bot\},\ D_\mu=\{z\in D\mid U_\mu \subseteq z  \}\]
for $1\leq \alpha \leq p$ and $p+1\leq \mu \leq p+q$.
On the other hand by \eqref{Dlambda0} we have
\[D_{U,z'_0}=\{z\in D\mid \oplus_{j=p+1}^{\ p+s} U_j \subset z \subset (\oplus_{j=1}^{r} U_j)^\bot\}.\]
The lemma follows.
\end{proof}

\begin{lemma}\label{sphericality}
For any element $z\in D$, we have
\[M(z,v_j)=cosh^2(t)+sinh^2(t),\]
where $t=d(z,D_j)$. In particular,
\[M(z,v_j)\geq M(z_0,v_j).\]
Equality holds if and only if $z\in D_{j}$.
\end{lemma}
\begin{proof}
Let us assume $1\leq j \leq p$. The case $p+1\leq j\leq p+q$ is similar. Without loss of generality we can assume $j=1$. 
It is easy to see that
\[\exp(-t E_{\alpha \mu})v_1=
\begin{cases}
cosh(t) \cdot v_1- i sinh(t) \cdot v_\mu & \text{if } \alpha=1,\\
v_1 & \text{otherwise},
\end{cases}\]
\[\exp(-t F_{\alpha \mu})v_1=
\begin{cases}
cosh(t) \cdot v_1-sinh(t) \cdot v_\mu & \text{if } \alpha=1,\\
v_1 & \text{otherwise},
\end{cases}\]
where $E_{\alpha \mu}$ and $F_{\alpha \mu}$ are defined in Section \ref{varphiUni}.
Recall that the group $G_{U_1}$ fixes $v_1$. Hence
\[M^1(z)=M^1(g z), \forall g\in G_{U_1}.\]
We have a $G_{U_1}$-equvariant fibration $\pi_1:D\rightarrow D_1$ (see Section \ref{fibrationpi}). By translating $z$ using an element in $G_{U_1}$, we can assume that $\pi_1(z)=z_0$ and is of the form $z=\exp(X)z_0$, where $X\in N_{z_0} D_1\subseteq \mathfrak{p}_0$. We have
\[N_{z_0} D_1=\Span_\R\{E_{1 \mu},F_{1 \mu}\mid p+1\leq \mu \leq p+q\}.\]
We assume that  
\[X=\sum_{\mu=p+1}^{p+q} x_{1 \mu} E_{1 \mu}- \sum_{\mu=p+1}^{p+q} y_{1 \mu} F_{1 \mu}=\sum_{\mu=p+1}^{p+q} v_1 \circ (x_{1 \mu }+i y_{1 \mu}) v_\mu.\]
We define 
\[t=\sqrt{\sum_{\mu=p+1}^{p+q}(|x_{1 \mu}|^2+|y_{1 \mu}|^2)}.\]
and
\[v=\frac{1}{t}\sum_{\mu=p+1}^{p+q} (x_{1 \mu }+i y_{1 \mu}) v_\mu.\]
Then $(v,v)=i$ and $(v_1,v)=0$. We have 
\[\exp(-X)(v_1)=cosh(t) \cdot v_1- i sinh(t) \cdot v.\]
So we have 
\[(z,v_1)=\|cosh(t) \cdot v_1- i sinh(t) \cdot v\|^2_{z_0}=cosh^2(t)+sinh^2(t).\]
Since $z=\exp(t v_1\circ v) z_0$ we know that $t=d(z,z_0)=d(z,D_1)$. The claim of the lemma is proved. 
\end{proof}

Define $M:D\rightarrow \R$ by (see \eqref{definitionMajorant})
\[M(z)=M(z,\bx).\]
Define  $h^j:N_{z_0} D_{U,z'_0}\rightarrow \R$ by
\begin{equation}
   h^j(X)=M(\exp (X) z_0,v_j).
\end{equation}
Then we have (see \eqref{h(Y)} for the definition of $h(Y)$)
\begin{equation}\label{hYcaseA}
    h(Y)=\pi(\sum_{j=1}^r h^j(Y)+\sum_{j=p+1}^{p+s} h^j(Y)).
\end{equation}

\begin{proposition}\label{secondcondition}
The function $h$ satisfies condition $(2)$ and $(3)$ of Theorem \ref{methodoflaplace}. In particular, $0$ is the unique minimal point of $h$.
\end{proposition}
\begin{proof}
By Lemma \ref{sphericality} and Lemma \ref{intersectionofDj}, $M(z)$ obtains its minimal value at $z$ if and only if $z\in D_{U,z'_0}$. In particular, $z_0$ is the unique minimal point of the function $M(z)$ on $F_{z_0} D_{U,z'_0}$. Hence $0$ is the unique minimal point of the function $h$ on $N_{z_0} D_{U,z'_0}$.
Suppose that $X\in S(N_{z_0} D_{U,z'_0})$. We define $h_{X}: \R \rightarrow \R $ by 
\[h_{X}(t)=h(\exp_{z_0}(X t)).\]
Then $t=0$ is a global minimum for $h_{X}(t)$, thus we have 
\begin{equation}\label{derivativeofh}
\frac{d}{dt}h_{X}(t)|_{t=0}=0.
\end{equation}
By Lemma \ref{gxnorm}, we have
\[( \exp{(- tX)} v_j, \exp{(-tX)} v_j)_{z_0}  = 
\sum_{k=1}^m |a_j^{k}|^2  \exp{( - 2 \lambda_k t)}.\]
From this we know
\begin{equation}\label{doubletderivativeofh}
    \frac{d^2}{dt^2}h_{X}(t)= \pi\cdot
\sum_{k=1}^n\sum_{j=1}^{m} 4 |a_j^{k}|^2 \lambda_k^2 \exp{( - 2 \lambda_k t)}.
\end{equation}
With Lemma \ref{twobounds} in mind, we have a uniform lower bound of $\frac{d^2}{dt^2}h_{X}(t)$ for all $X\in S(N_{z_0} D_{U,z'_0})$ and $t\geq 0$. A similar argument works for $t<0$. So we can assume that $\frac{d^2}{dt^2}h_X(t)\geq C$ for a positive constant $C$ and all $t\in \R, X\in S(N_{z_0} D_{U,z'_0})$.
It follows that 
\[h(X)\geq h(0)+\frac{1}{2} C ||X||^2.\]
Hence $h$ satisfies condition $(2)$ of Theorem \ref{methodoflaplace}. Condition (3) is also satisfied because we know from the above that the Hessian matrix of $h$ is positive definite with the smallest eigenvalue bigger or equal to $C$.
\end{proof}

Any $X\in N_{z_0} D_{U,z'_0}$ can be written as 
\[X=\sum_{(\alpha ,\mu )\in I} x_{\alpha \mu} E_{\alpha \mu}+\sum_{(\alpha,\mu )\in I} y_{\alpha \mu} F_{\alpha \mu},\]
where $I$ is the index set defined in equation \eqref{indexI}. 
We need
\begin{corollary}\label{hessian}
Suppose $(\alpha,\mu),(\beta,\nu)\in I$. For $1\leq j  \leq p$, we have
\[\frac{\partial^2 h^j}{\partial x_{\alpha \mu} \partial x_{\beta \nu}}|_{X=0}=\frac{\partial^2 h^j}{\partial y_{\alpha \mu} \partial y_{\beta \nu}}|_{X=0}=4 \delta_{\alpha \beta} \delta_{\alpha j} \delta_{\mu \nu},\text{ and }\frac{\partial^2 h^j}{\partial x_{\alpha \mu} \partial y_{\beta \nu}}|_{X=0}=0.\]
For $p+1\leq j \leq p+q$, we have
\[\frac{\partial^2 h^j}{\partial x_{\alpha \mu} \partial x_{\beta \nu}}|_{X=0}=\frac{\partial^2 h^j}{\partial y_{\alpha \mu} \partial y_{\beta \nu}}|_{X=0}=4 \delta_{\alpha \beta} \delta_{\mu j} \delta_{\mu \nu}, \text{ and } \frac{\partial^2 h^j}{\partial x_{\alpha \mu} \partial y_{\beta \nu}}|_{X=0}=0.\]
\end{corollary}
\begin{proof}
We need to compute the following
\[\frac{\partial^2}{\partial s \partial t}[(\exp{(sX+tY)}v,\exp{(sX+tY)}v)_{z_0}]|_{s=t=0}.\]
Using the second order approximation of the exponential map
\[\exp{(sX+tY)}=I+(sX+tY)+\frac{1}{2}(s^2 X^2+ t^2 Y^2 + s t X Y+s t Y X)+\text{higher order terms},\]
one can check that 
\begin{align}
    &\frac{\partial^2}{\partial s \partial t}[(\exp{(sX+tY)}v,\exp{(sX+tY)}v)_{z_0}]|_{s=t=0} \\ \nonumber
    =& \frac{1}{2}((XY+YX)v,v)_{z_0}+\frac{1}{2}(v,(XY+YX)v)_{z_0}+(Yv,Xv)_{z_0}+(Xv,Yv)_{z_0}\\ \nonumber
    =&2(Xv,Yv)_{z_0}+2(Yv,Xv)_{z_0}. \nonumber
\end{align}
The last equality follows from the fact that $(Xv,w)_{z_0}=(v,Xw)_{z_0}$ for any $X\in \mathfrak{p}_0$ and $v,w\in V$.
We also have
\[E_{\alpha \mu}(v_\alpha)=-i v_\mu,\ E_{\alpha \mu}(v_\mu)=i v_\alpha,\]
\[F_{\alpha \mu}(v_\alpha)=- v_\mu,\ F_{\alpha \mu}(v_\mu)=- v_\alpha\]
for $1\leq \alpha \leq p,p+1\leq \mu \leq p+q$. Also recall that $(,)_{z_0}$ is a positive definite Hermitian form with an orthonormal basis $\{v_1,\ldots,v_{p+q}\}$. With these preparations, the formulas in the lemma follow from straightforward calculations.
\end{proof}

The Hessian matrix $A$ of $h$ at $0$ is a $2(rq+ps-rs)$ by $2(rq+ps-rs)$ matrix.
\begin{corollary}\label{Hessian}
$A$ is diagonal with diagonal entries
\[\frac{\partial^2 h}{\partial^2 x_{\alpha \mu}}=\frac{\partial^2 h}{\partial^2 y_{\alpha \mu}}=8\pi\]
for $ 1\leq \alpha \leq r, p+1\leq \mu \leq p+s$. And
\[\frac{\partial^2 h}{\partial^2 x_{\alpha \mu}}=\frac{\partial^2 h}{\partial^2 y_{\alpha \mu}}=4\pi\]
for $ 1\leq \alpha \leq r, p+s+1\leq \mu \leq p+s$ or $ r+1\leq \alpha \leq p, p+1\leq \mu \leq p+s$. In particular it is positive definite with determinant
\[\mathrm{det}(A)=4^{2rq+2ps-rs}\cdot \pi^{2(ps+rq-rs)}.\]
\end{corollary}
\begin{proof}
The corollary follows from \ref{hessian} and \eqref{hYcaseA}.
\end{proof}

Recall that $J(t)$ is defined in \eqref{J(t)}.
\begin{proposition}\label{fiberintegralUpq}
\[J(t)  \sim  (-i)^{ps+rq-rs} 2^{ps+rq-5rs} \pi^{2ps+2rq-4rs} t^{(p+q)(r+s)-2rs} \exp(-(r+s)\pi t^2).\]
\end{proposition}
\begin{proof}
Recall the proof of Theorem \ref{genericnonzero} at the beginning of this section.
By Lemma \ref{highestermSchrodinger}, we know that the highest degree term of \eqref{finalsummation} evaluated at 
\[\bx=(v_1,\ldots,v_r,v_{p+1},\ldots, v_{p+s})\]
is
\[(2\sqrt{2}\pi t)^{2(rq+ps-rs)} i^*(\bigwedge_{(\alpha,\mu)\in I} \xi'_{\alpha \mu}\wedge  \xi''_{\alpha \mu})|_{z_0},\]
where $I$ is specified at \eqref{indexI}.
By Theorem \ref{methodoflaplace} we know
\begin{align*}
    &J(t)\\
    \sim& t^{(p+q)(r+s)}\cdot (\frac{2\pi}{t^2})^{ps+rq-rs} \cdot (-2\sqrt{2}\pi t)^{2r(q-s)+2s(p-r)} \cdot \mathrm{det}( A)^{-\frac{1}{2}}\\
    \cdot&\exp(-t^2\pi h(0))\cdot (-\frac{1}{2}i)^{ps+rq-rs}\\
    =& (-i)^{ps+rq-rs} 2^{ps+rq-5rs} \pi^{2ps+2rq-4rs} t^{(p+q)(r+s)-2rs} \exp(-(r+s)\pi t^2).
\end{align*}
The theorem is proved.
\end{proof}

\subsection{Method of Laplace for Case B}

We want to apply Theorem \ref{methodoflaplace} to compute $J(t)$
for the dual pair $(\Sp(2n,\R),\Orth(2r,2r))$. We need to check conditions (2) and (3) of Theorem \ref{methodoflaplace} again. We proceed quickly by omitting the proofs that are similar to those of case A. In this subsection we use the assumptions and notations of Section \ref{seesawonesection}.

Recall that we assume that 
\[z_0=\Span\{v_{n+1},\ldots, v_{2n}\}, \text{ and } \bx=(\vec{x}_1,\ldots,\vec{x}_{2r})=(E_1,\ldots, E_r, F_1,\ldots,F_r).\]
Recall from \eqref{h(Y)} that $h: N_{z_0} D_{U,z'_0}\rightarrow \R$ is defined by
\[h(Y)=\pi\cdot \sum_{j=1}^{2r} \|\exp(-Y) \vec{x}_j\|_{z_0}^2=\pi\cdot\sum_{j=1}^r (\|\exp(-Y) E_j\|_{z_0}^2+\|\exp(-Y) F_j\|_{z_0}^2).\]
Since $(,)_{z_0}$ is Hermitian, we have the following identity 
\begin{equation}
    \|x\|^2_{z_0}+\|y\|^2_{z_0}=\frac{1}{2}\|x-iy\|^2_{z_0}+\frac{1}{2}\|x+iy\|^2_{z_0}, \forall x,y \in V.
\end{equation}
In our coordinates, it is more natural to write 
\begin{align*}
    h(Y)=&\frac{1}{2}\pi\cdot\sum_{j=1}^r (\|\exp(-Y)( E_j-F_j i)\|_{z_0}^2+\|\exp(-Y) (E_j+F_j i)\|_{z_0}^2)\\
    =&\pi\cdot \sum_{j=1}^r (\|\exp(-Y)v_j\|_{z_0}^2+\|\exp(-Y) v_{n+j}\|_{z_0}^2).
\end{align*}
We then have the following proposition.
\begin{proposition}\label{secondconditioncaseB}
The function $h$ satisfies condition (2) and (3) of Theorem \ref{methodoflaplace}.
\end{proposition}
\begin{proof}
For $X\in S(N_{z_0} D_{U,z'_0})$, define $h_{X}: \R \rightarrow \R $ by 
\[h_{X}(t)=h(X t).\]
Then as in the proof of Proposition \ref{secondcondition}, \eqref{doubletderivativeofh} still holds and by Lemma \ref{twobounds} we know that $\frac{d^2}{dt^2}h_{X}(t)\geq C$  for all $t\in \R$ and $X\in S(N_{z_0} D_{U,z'_0})$. Using the formula
\[\frac{d}{dt}(\exp(Xt)v,\exp(Xt)v)_{z_0}|_{t=0}=(Xv,v)_{z_0}+(v,Xv)_{z_0},\]
one can also show that all first order derivatives of $h$ at $Y=0$ vanish.
This suggests that $Y=0$ is the unique minimal point of $h$ on $N_{z_0} D_{U,z'_0}$. These facts imply that 
\[h(X)\geq h(0)+\frac{1}{2} C ||X||^2.\]
Hence $h$ satisfies condition (2)
The Hessian matrix of $h$ is positive definite with the smallest eigenvalue no smaller than $C$. Hence condition (3) of Theorem \ref{methodoflaplace} is also satisfied.. 
\end{proof}
Recall that $J(t)$ is defined in \eqref{J(t)}. In this case we have
\begin{proposition}\label{fiberintegralSp}
\[J(t)\sim  (-i)^{\frac{1}{2}(2nr+r-r^2)} 2^{5nr+\frac{15}{4}r-\frac{19}{4}r^2}\pi^{2r(n-r+1)} t^{2nr+r-r^2} \exp(-r\pi t^2).\]
\end{proposition}
\begin{proof}
Apply Lemma \ref{highesttermofpolynomial}, Lemma \ref{onlyonetermSp} and Theorem \ref{methodoflaplace}. The details are similar to those of the proof of Proposition \ref{fiberintegralUpq}.
\end{proof}

\subsection{Method of Laplace for  Case C}
We want to apply theorem \ref{methodoflaplace} to compute $J(t)$ in case of the dual pair $(\Orth^*(2n),\Sp(r,r))$. We need to check conditions (2) and (3) of theorem \ref{methodoflaplace}. We use the assumptions and notations of Section \ref{seesawtwosection}.

Recall that we assume that 
\[z_0=\Span\{v_{n+1},\ldots, v_{2n}\}, \text{ and } \bx=(\vec{x}_1,\ldots,\vec{x}_{r})=(v_1,\ldots,v_r).\]

Recall from \eqref{h(Y)} that $h: N_{z_0} D_{U,z'_0}\rightarrow \R$ is defined by
\[h(Y)=\pi\cdot\sum_{s=1}^r \|\exp(-Y) \vec{x}_s\|_{z_0}^2=\pi\cdot\sum_{s=1}^r \|\exp(-Y) v_s\|_{z_0}^2.\]
We then have
\begin{proposition}\label{secondconditioncaseC}
The function $h $ satisfies condition (2) and (3) of Theorem \ref{methodoflaplace}.
\end{proposition}
\begin{proof}
Similar to that of Proposition \ref{secondconditioncaseB}.
\end{proof}
Recall that $J(t)$ is defined in \eqref{J(t)}. In this case we have
\begin{proposition}\label{fiberintegralOstar}
\[J(t)\sim  (-i)^{\frac{1}{2}(2nr-r^2-r)} 2^{4nr-\frac{13}{4}r-\frac{15}{4}r^2}\pi^{2r(n-r-1)} t^{2nr-r^2-r} \exp(-r\pi t^2).\]
\end{proposition}
\begin{proof}
Apply Lemma \ref{highesttermofpolynomial}, Lemma \ref{onlyonetermOstar} and Theorem \ref{methodoflaplace}. The details are similar to those of the proof of Proposition \ref{fiberintegralUpq}.
\end{proof}

\appendix
\section{The associated vector bundle}\label{appendix}
We go back to the settings of Section \ref{thomformsection}.
Recall that $\tilde{G}'_\infty$ is the metaplectic cover of $G'_\infty$, where 
${G}'_\infty= \prod_{v} G'_{v}$. Let $\tilde{K}'_\infty$ be the subgroup of  $\tilde{G}'_\infty$ which fixes the Gaussian function in equation \eqref{definitionofvarphiinfty}. Then $\tilde{K}'_\infty$ is a maximal compact subgroup of $\tilde{G}'_\infty$ which is the metaplectic cover of ${K}'_\infty= \prod_{v} K'_{v}$, a maximal compact subgroup of $G'_\infty$.

$\tilde{G}'$ acts on $\mathcal{S}(V_\infty^n)$ by the Weil representation $\omega$ and the action commutes with that of $G$.	 
In this appendix, we show that $\theta_{\cL,\tilde{\varphi}}$ is a matrix coefficient of an automorphic vector bundle 
\[\mathcal{E}\rightarrow \tilde{\Gamma}' \backslash \tilde{G}'_\infty / \tilde{K}'_\infty.\]
For each Archimedean place $v$ of the number field $k$ we have 
\begin{enumerate}
    \item $G'_v=\Uni(r+s,r+s)$, $K_v=\rU(r+s)\times \rU(r+s)$,
    \item  $G'_v=\Orth(2r,2r)$, $K_v=\rO(2r)\times \rO(2r)$, 
    \item  $G'_v=\Sp(r,r)$, $K_v=\Sp(r)\times \Sp(r)$,
\end{enumerate}
in case A, B and C respectively.

In order to determine the $\tilde{K}'_\infty$ action on
\[\varphi_\infty=\varphi \otimes \prod_{v\neq v_1}\varphi_v,\]
it suffices to compute the 
$\tilde{K}'_{v_1}$ action on $\varphi$ since  $\tilde{K}'_v$ acts on $\varphi_v$ trivially for $v\neq v_1$. It turns out that often times $\tilde{K}'_{v_1}$ is the trivial two-fold cover of $K'_{v_1}$ and the action descends to $K'_{v_1}$. The following general argument applies to both $\tilde{K}'_{v_1}$ and $K'_{v_1}$ representations so we just deal with the $\tilde{K}'_{v_1}$ case for brevity. 

We will see that $\varphi$ is a highest weight vector of an irreducible representation of $\tilde{K}'_{v_1}$. We denote this representation by $\sigma$.
To be more precise there is an irreducible representation $ (E_\sigma, \sigma)$  of $ \tilde{K}'_\infty$ inside, where $ E_\sigma \subset \mathcal{P}\subset\mathcal{S}(V_\infty^m)$ such that $\varphi_\infty \in E_\sigma$ and
\begin{equation}\label{K'actbysigma}
    \omega(g'k')\phi=\omega(g') (\sigma(k')\phi)
\end{equation}
for all $g'\in \tilde{G}'_\infty$, $k'\in \tilde{K}'_\infty $ and $\phi \in E_\sigma$.

Let $E_\sigma^*=\Hom_\C(E_\sigma,\C)$ be the dual representation of $E$. There is a canonical element $\Phi \in E_\sigma\otimes_\C E_\sigma^*$ which corresponds to the identity element in $E_\sigma \otimes_\C E^*_\sigma\cong \Hom_\C(E_\sigma,E_\sigma)$. Explicitly we choose a basis $\{\varphi^1,\ldots, \varphi^d\}$ of $E_\sigma$ and we assume $\varphi^1=\varphi_\infty$. Let $\{e_1,\ldots,e_d\}$ be the corresponding dual basis of $E_\sigma^*$. Then we have
\[\Phi=\sum_{i=1}^d \varphi^i \otimes e_i.\]
By definition the diagonal action of $\tilde{K}'_\infty$ on $\Phi$ leaves it invariant:
\begin{equation}
    (\sigma \otimes \sigma^*)(k')\Phi=\Phi, \, k'\in \tilde{K}'_\infty.
\end{equation}
Equivalently
\begin{equation}\label{K'equivariance}
    (\sigma \otimes Id)(k')\Phi=(Id \otimes \sigma^*)((k')^{-1})\Phi, \,\forall k'\in \tilde{K}'_\infty,
\end{equation}
where $Id$ stands for the trivial action. For each $\bx \in V_\infty^m$, 
\[\Phi(g',\bx)=((\omega(g')\otimes Id )\Phi)(\bx)\]
is a function on $\tilde{G}'_\infty$ with values in $E^*_\sigma$. Equation \eqref{K'actbysigma} and equation \eqref{K'equivariance} together imply that
\begin{equation}\label{associatedbundleequation}
    \Phi(g'k',\bx)=(Id\otimes \sigma^*)((k')^{-1}) \Phi(g',\bx).
\end{equation}
In other words, $\Phi(\cdot, \bx)$ is a section of the vector bundle $\mathcal{E}\rightarrow \tilde{G}'_\infty / \tilde{K}'_\infty$ associated to the representation $(E_\sigma^*,\sigma^*)$ of $\tilde{K}'_\infty$.

We now apply $\theta$-distribution to get
\[\theta_{\mathcal{L},\Phi}(g')=\sum_{\bx \in \mathcal{L}^m}\Phi(g',\bx). \]
Then $\theta_{\mathcal{L},\Phi}(g')$ is $\tilde{\Gamma}'$-invariant:
\[\theta_{\mathcal{L},\Phi}(\gamma'g')=\theta_{\mathcal{L},\Phi}(g'),\gamma'\in \tilde{\Gamma}'.\]
By equation \eqref{associatedbundleequation} and the $\tilde{\Gamma}'$-invariance of $\theta_{\mathcal{L},\Phi}$,
$\theta_{\mathcal{L},\Phi}$ is a section of the bundle $\tilde{\Gamma}' \backslash \mathcal{E}\rightarrow \tilde{\Gamma}' \backslash \tilde{G}'_\infty/ \tilde{K}'_\infty$. The cocycle $\theta_{\mathcal{L},\varphi_\infty}$ is then a matrix coefficient of $ \theta_{\mathcal{L},\Phi}$:
\[\theta_{\mathcal{L},\varphi_\infty}(g')=\langle \varphi_\infty \theta_{\mathcal{L},\Phi}(g')\rangle,\]
where $\langle,\rangle$ is an $\tilde{K}'_\infty$-invariant bilinear pairing between $E_\sigma$ and $E_\sigma^*$.

In case A, we will decide the $(\tilde{K}'_{v_1})^0$ (the identity component of $\tilde{K}'_{v_1}$) action by computing highest weights.
In case B and C, $\tilde{K}'_{v_1}$ is the trivial two-fold cover of $K'_{v_1}$ and the action descends to $K'_{v_1}$, so we will decide the $K'_{v_1}$ action. In case A and B the calculations are essentially done in \cite{KV}. In each case $\varphi=\iota (\phi^+\wedge \phi^-)$ (see \eqref{definitionofvarphirs}) and $(\tilde{K}'_{v_1})^0=K^{-}\times K^{+}$ such that $K^{-}$ acts trivially on $ \phi^-$ and $K^{+}$ acts trivially on $ \phi^-$ (see Remark \ref{K'actionremark}). So it is enough to determine the $K^{-}$ action on $\phi^+$ and $K^{+}$ action on $\phi^-$.

\subsection{Case A}
In this case ${G}'_{v_1}\cong \rU(m,m)$, where $m=r+s$. Recall that in Section \ref{FockmodelUpq} we choose a basis $\{w_1,\ldots, w_m,w_{m+1},\ldots,w_{2m}\}$ of $W_{v_1}$ with the Hermitian form $\langle,\rangle$ such that 
\begin{enumerate}
		    \item $\langle w_a,w_a\rangle=1 $, \
		    \item $\langle w_k,w_k\rangle=-1 $
\end{enumerate}
for $1\leq a \leq m, m+1\leq k \leq 2m$ and $\langle w_j,w_k\rangle=0$ if $j \neq k$.
Define 
\[W_+=\Span_\C\{w_1,\ldots,w_m\},W_-=\Span_\C\{w_{m+1},\ldots,w_{2m}\}.\]
Denote the group $\rU(W_+,\langle,\rangle)$ ($\rU(W_-,\langle,\rangle)$ resp.) by $\rU(m,0)$ ($ \rU(0,m)$ resp.).
Then $K'_{v_1}=\rU(0,m)\times \rU(m,0)$ is a maximal compact subgroup of $G'_{v_1}$.
Let $\mathfrak{k}'_0$ be its Lie algebra and $\mathfrak{k}'$ be its complexification.

Following a calculation like that of Section 7 of \cite{KM90}, we can show that $\mathfrak{k}'$ acts on the Fock model by 
\[\omega(e_{ab})=-\sum_{\alpha=1}^p u_{\alpha,b}^-\frac{\partial}{\partial u_{\alpha,a}^-}+\sum_{\mu=p+1}^{p+q} u_{\mu,a}^-\frac{\partial}{\partial u_{\mu,b}^-}-\frac{1}{2}\delta_{ab}(p-q)\]
for $1\leq a,b \leq m$ and
\begin{equation}\label{ursaction}
  \omega(e_{k \ell})=\sum_{\alpha=1}^p u_{\alpha,k}^-\frac{\partial}{\partial u_{\alpha,\ell}^-}-\sum_{\mu=p+1}^{p+q} u_{\mu,\ell}^-\frac{\partial}{\partial u_{\mu,k}^-}+\frac{1}{2}\delta_{k\ell}(p-q)  
\end{equation}
for $m+1\leq k,\ell \leq 2m$.

Recall that in Section \ref{Andersoncocyle} we define the element $\phi^+_{r,s}$ using the special harmonic $(f_1^-)^{q-s} (f_2^-)^{p-r}$.
Here $f_1^-$ and $f_2^-$ are as in Definition \ref{specialharmonics} except that we shift $b$ to $b+m$, where $b$ is the second index of the variable $u_{ab}^-$. This is because in Definition \ref{specialharmonics} our assumption is that $W$ is negative definite.

Let $\mathfrak{t}$ be the diagonal torus of $\mathfrak{u}(0,m)$, $\mathfrak{n}$ be the strictly upper triangular Lie algebra of $\mathfrak{u}(0,m)$:
\[\mathfrak{n}=\Span_\C\{e_{k \ell}\mid m+1\leq k <\ell \leq 2m\}.\]
Using equation (\ref{ursaction}), it is easy to see that $(f_1^-)^{q-s} (f_2^-)^{p-r}$ has weight 
\[(\underbrace{-s+\frac{1}{2}(p+q),\ldots,-s+\frac{1}{2}(p+q)}_r,\underbrace{r-\frac{1}{2}(p+q),\ldots,r-\frac{1}{2}(p+q)}_s)\]
under $\mathfrak{t}$. Moreover we can show that $(f_1^-)^{q-s} (f_2^-)^{p-r}$ is annihilated by $\mathfrak{n}$. We have three cases
\begin{enumerate}
    \item $m+1\leq k <\ell \leq m+r$,\
    \item $m+1\leq k \leq m+r$,$m+r+1\leq \ell \leq m+r+s$,\
    \item $m+r+1\leq k <\ell \leq m+r+s$.
\end{enumerate}
In case (1), $e_{k \ell}$ ($k<\ell$) replaces a column of $f_1^-$ by an existing column and acts trivially on all the variables in $f_2^-$. In case (2), $e_{k \ell}$
acts trivially on all the variables in $f_1^-$ and $f_2^-$. In case (3), $e_{k \ell}$ acts trivially in all the variables in $f_1^-$ and replaces a column of $f_2^-$ by an existing column. In any case $e_{k \ell}$ annihilates $(f_1^-)^{q-s} (f_2^-)^{p-r}$.

The conclusion is that $\phi^+_{r,s}$ is a highest weight vector of $\mathfrak{u}(0,m)$. Similarly $\phi^-_{r,s}$ is a lowest weight vector of 
$\mathfrak{u}(m,0)$ with weight
\[(\underbrace{s-\frac{1}{2}(p+q),\ldots,s-\frac{1}{2}(p+q)}_r,\underbrace{-r+\frac{1}{2}(p+q),\ldots,-r+\frac{1}{2}(p+q)}_s).\]
It generates an irreducible representation of highest weight
\[(\underbrace{-r+\frac{1}{2}(p+q),\ldots,-r+\frac{1}{2}(p+q)}_s,\underbrace{s-\frac{1}{2}(p+q),\ldots,s-\frac{1}{2}(p+q)}_r).\]

\subsection{Case B}
In this case $G'_{v_1}=\rO(2r,2r)$. 

Let us use the notation of subsection \ref{seesawonesection}. Recall that $W$ is a complex vector space with a Hermitian form $\langle,\rangle$ of signature $(r,r)$. We denote by $W_\R$ the underlying real vector space of $W$ and let $\langle,\rangle_\R=\mathrm{Re}\langle,\rangle$.
 $G'_{v_1}$ is the linear isometry group of $(W_\R,\langle,\rangle_\R)$. We can choose an orthonormal basis $\{w_a,w_k\mid 1\leq a \leq r, r+1\leq k \leq 2r\}$ of $W$ such that 
 \[\langle w_a,w_a\rangle=1,\langle w_k,w_k\rangle=-1\]
for $1\leq a \leq r, r+1\leq k \leq 2r$. Define 
\[W_+=\Span_\C\{w_1,\ldots,w_r\},W_-=\Span_\C\{w_{r+1},\ldots,w_{2r}\}.\]
Denote the group $\rO(W_+,\langle,\rangle_\R)$ ($\rO(W_-,\langle,\rangle_\R)$ resp.) by $\rO(2r,0)$ ($ \rO(0,2r)$ resp.).
Then $K'_{v_1}=\rO(0,2r)\times \rO(2r,0)$ is a maximal compact subgroup of $G'_{v_1}$.
Let $\mathfrak{k}'_0$ be its Lie algebra and $\mathfrak{k}'$ be its complexification. Then (c.f. Section 7 of \cite{KM90})
\[\mathfrak{k}'_0\cong \wedge^2_\R (W_+)\oplus \wedge^2_\R (W_-)=\mathfrak{o}(2r,0)\oplus \mathfrak{o}(0,2r),\  \mathfrak{k}'\cong \wedge^2_\C (W_+\otimes_\R \C)\oplus \wedge^2_\C (W_-\otimes_\R \C).\]

First we focus on $\rO(0,2r)$. Denote the complex structure $1\otimes i$ on $W_-\otimes_\R \C$ by right multiplication by $i$. Define 
		\[w'_k=w_k+i w_k i, \ w''_k=w_k-i w_k i,\]
where $r+1 \leq k \leq 2r$. We take $\{w''_{r+1},\ldots,w''_{2r},w'_{r+1},\ldots,w'_{2r}\}$ to be the basis of $W_-\otimes_\R \C$.
Notice that if we extend the form $\langle,>_\R$ complex linearly to a symmetric form on $W_\R \otimes_\R \C$. Then
		\[\langle w'_k,w''_\ell>_\R=- 2\delta_{k \ell},\ \langle w'_k,w'_\ell>_\R=\langle w''_k,w''_\ell>_\R=0.\]
Then we have a split torus $\mathfrak{t}$ of $\mathfrak{o}(0,2r)\otimes_\R \C$ spanned by
\[\{w'_k\wedge w''_k\mid r+1 \leq k \leq 2r\}.\]
We also have a nilpotent algebra $\mathfrak{n}$ of $\mathfrak{o}(0,2r)\otimes_\R \C$:
\[\mathfrak{n}=\Span\{w''_k\wedge w''_\ell \mid r+1 \leq k,\ell \leq 2r\}\oplus \Span\{w'_k\wedge w''_\ell \mid r+1 \leq k<\ell \leq 2r\}.\]
The Lie algebra $\mathfrak{o}(0,2r)\otimes_\R \C$ acts on the Fock model in the following way:
\[\omega(w'_k\wedge w''_\ell)=
\omega(2e_{r+k,r+\ell}-2e_{\ell,k})=2\sum_{\alpha=1}^n(u_{\alpha+n,k}^-\frac{\partial}{\partial u_{\alpha+n,\ell}^-}-u_{\alpha,\ell}^- \frac{\partial}{\partial u_{\alpha,k}^-}),\]
\[\omega(w'_k\wedge w'_\ell)=\omega(2e_{r+k,\ell}-2e_{r+\ell,k})=2\sum_{\alpha=1}^n(u_{\alpha+n,k}^-\frac{\partial}{\partial u_{\alpha,\ell}^-}-u_{\alpha+n,\ell}^-\frac{\partial}{\partial u_{\alpha,k}^-}),\]
\[\omega(w''_k\wedge w''_\ell)=\omega(2e_{k,r+\ell}-2e_{\ell,r+k})=2\sum_{\alpha=1}^n(u_{\alpha,k}^-\frac{\partial}{\partial u_{\alpha+n,\ell}^-}-u_{\alpha,\ell}^-\frac{\partial}{\partial u_{\alpha+n,k}^-}).\]

Recall that in Section \ref{Andersoncocyle} we define the element $\phi^+_r$ using the special harmonic $(f^-)^{n-r+1}$.
Here $f^-$ is as in Definition \ref{specialharmonics} except that we shift $b$ to $b+r$, where $b$ is the second index of the variable $u_{ab}^-$. This is because in Definition \ref{specialharmonics} our assumption is that $W$ is negative definite.

Then it is easy to see that $(f^-)^{n-r+1}$ is annihlated by $\mathfrak{n}$ and has weight 
\[(\underbrace{(n-r+1),\ldots,(n-r+1)}_r)\]
under $\mathfrak{t}$. So $\phi^+_r$ is a highest weight vector of $\mathfrak{so}(0,2r)$. 

Under the group $\rO(0,2r)$, $\phi^+_r$ generates an irreducible representation that splits into two irreducible representations of $\mathfrak{so}(0,2r)$ with highest weights 
\[(\underbrace{(n-r+1),\ldots,(n-r+1)}_{r-1},\pm(n-r+1)).\]

Similarly $\phi^-_r$ is a lowest weight vector of $\mathfrak{so}(2r,0)$ of weight 
\[(\underbrace{-(n-r+1),\ldots,-(n-r+1)}_r).\]
It generates an irreducible representation of $\rO(2r,0)$ with highest weights
\[(\underbrace{(n-r+1),\ldots,(n-r+1)}_{r-1},\pm(n-r+1)).\]

\subsection{Case C} In this case $G'_{v_1}=\Sp(r,r)$. 

Let us use the notation of subsection \ref{seesawonesection}. Recall that $W_\H$ is a $\H$-vector space with a Hermitian form $\langle,_\H$ of signature $(r,r)$.
 $G'_{v_1}\cong $ is the linear isometry group of $(W_\H,\langle,\rangle_\H)$.
We can choose an orthonormal basis $\{w_a,w_k\mid 1\leq a \leq r, r+1\leq k \leq 2r\}$ of $W_\H$ such that 
 \[\langle w_a,w_a\rangle_\H=1,\langle w_k,w_k\rangle_\H=-1\]
for $1\leq a \leq r, r+1\leq k \leq 2r$.  Define 
\[W_+=\Span_\H\{w_1,\ldots,w_r\},W_-=\Span_\H\{w_{r+1},\ldots,w_{2r}\}.\]
Denote the group $\Sp(W_+,\langle,\rangle_\H)$ ($\Sp(W_-,\langle,\rangle_\H)$ resp.) by $\Sp(r,0)$ ($ \Sp(0,r)$ resp.).
Then $K'_{v_1}=\Sp(0,r)\times \Sp(r,0)$ is a maximal compact subgroup of $G'_{v_1}$.
Let $\mathfrak{k}'_0$ be its Lie algebra and $\mathfrak{k}'$ be its complexification. Then 
\[\mathfrak{k}'\cong \Sym^2_\C (W_+)\oplus \Sym^2_\C (W_-)\cong \mathfrak{sp}(r,0)\otimes_\R \C\oplus \mathfrak{sp}(0,r)\otimes_\R \C\cong\mathfrak{sp}(2r,\C)\oplus\mathfrak{sp}(2r,\C).\]

First we focus on $\Sp(0,r)$. We take the complex basis $\{w_{r+1},\ldots,w_{2r},j w_{r+1},\ldots,j w_{2r}\}$ of $W_-$.
We have a split torus $\mathfrak{t}$ of $\mathfrak{sp}(2r,\C)$:
\[\mathfrak{t}=\Span_\C\{e_{a,a}-e_{a+r,a+r}\mid 1 \leq a \leq r\}.\]
We also have a nilpotent algebra $\mathfrak{n}$ of $\mathfrak{sp}(2r,\C)$:
\[\mathfrak{n}=\Span\{e_{r+\ell,r+k}-e_{k,\ell} \mid r+1 \leq k<\ell \leq 2r\}\oplus \Span\{ e_{k,s+\ell}+e_{\ell,s+k}\mid r+1 \leq k,\ell \leq 2r\}.\]
The Lie algebra $\mathfrak{sp}(0,r)$ acts by the Weil representation in the following way:
\[\omega(e_{k,\ell}-e_{r+\ell,r+k})
		    =\sum_{\alpha=1}^n( u_{\alpha,k}^-\frac{\partial}{\partial u_{\alpha,\ell}^-}-u_{\alpha+n,\ell}^-\frac{\partial}{\partial u_{\alpha+n,k}^-}),\]
\[
		    \omega(e_{k,r+\ell}+e_{\ell,r+k})
		    =\sum_{\alpha=1}^n(u_{\alpha,k}^-\frac{\partial}{\partial u_{\alpha+n,\ell}^-}+u_{\alpha,\ell}^-\frac{\partial}{\partial u_{\alpha+n,k}^-}),\]
\[
		    \omega(e_{r+k,\ell}+e_{r+\ell,k})
		    =\sum_{\alpha=1}^n(u_{\alpha+n,\ell}^-\frac{\partial}{\partial u_{\alpha,k}^-}+u_{\alpha+n,k}^-\frac{\partial}{\partial u_{\alpha,\ell}^-}).
\]

Recall that in Section \ref{Andersoncocyle} we define the element $\phi^+_r$ using the special harmonic $(f^-)^{n-r-1}$.
Here $(f^-)$ is as in Definition \ref{specialharmonics} except that we shift $b$ to $b+r$, where $b$ is the second index of the variable $u_{ab}^-$. This is because in Definition \ref{specialharmonics} our assumption is that $W$ is negative definite. 

Moreover it is easy to see that $(f^-)^{n-r-1}$ is annihlated by $\mathfrak{n}$ and has weight 
\[(\underbrace{(n-r-1),\ldots,(n-r-1)}_r)\]
under $\mathfrak{t}$. So $\phi^+_r$ is a highest weight vector of $\Sp(0,r)$.

Similarly $\phi^-_r$ is a lowest weight vector of $\Sp(r,0)$ of weight 
\[(\underbrace{-(n-r-1),\ldots,-(n-r-1)}_r).\]
It generates an irreducible representation of $\Sp(r,0)$ with highest weight
\[(\underbrace{(n-r-1),\ldots,(n-r-1)}_r).\]

\noindent
{\bf Conflict of Interest Statement:}
On behalf of all authors, the corresponding author states that there is no conflict of interest.

\bibliographystyle{alpha}
\bibliography{reference}

\newcommand{\etalchar}[1]{$^{#1}$}
\begin{thebibliography}{BHK{\etalchar{+}}20}

\bibitem[Ada94]{Adams}
J~Adams.
\newblock The theta correspondence over {R}, lecture notes.
\newblock In {\em Workshop on the Theta correspondence and automorphic forms,
  University of Maryland}, 1994.

\bibitem[And83]{Anderson}
Greg~W Anderson.
\newblock Theta functions and holomorphic differential forms on compact
  quotients of bounded symmetric domains.
\newblock {\em Duke Mathematical Journal}, 50(4):1137--1170, 1983.

\bibitem[BB66]{BailyBorel}
Walter~L Baily and Armand Borel.
\newblock Compactification of arithmetic quotients of bounded symmetric
  domains.
\newblock {\em Annals of mathematics}, pages 442--528, 1966.

\bibitem[BHK{\etalchar{+}}20]{bruinier2017modularity}
Jan Bruinier, Benjamin Howard, Stephen~S Kudla, Michael Rapoport, and Tonghai
  Yang.
\newblock Modularity of generating series of divisors on unitary shimura
  varieties.
\newblock {\em Ast\'erisque}, 421(2):7--125, 2020.

\bibitem[BMM16]{BMM2}
Nicolas Bergeron, John Millson, and Colette Moeglin.
\newblock The {Hodge} conjecture and arithmetic quotients of complex balls.
\newblock {\em Acta Mathematica}, 216(1):1--125, 2016.

\bibitem[BMM17]{BMM1}
Nicolas Bergeron, John Millson, and Colette Moeglin.
\newblock Hodge type theorems for arithmetic manifolds associated to orthogonal
  groups.
\newblock {\em International Mathematics Research Notices},
  2017(15):4495--4624, 2017.

\bibitem[Bor69]{Borel}
Armand Borel.
\newblock {\em Introduction aux groupes arithm{\'e}tiques}.
\newblock Number 1341. Hermann, 1969.

\bibitem[Bor99]{borcherds1999}
Richard~E. Borcherds.
\newblock The {Gross-Kohnen-Zagier} theorem in higher dimensions.
\newblock {\em Duke Math. J.}, 97(2):219--233, 04 1999.

\bibitem[BT13]{BottTu}
Raoul Bott and Loring~W Tu.
\newblock {\em Differential forms in algebraic topology}, volume~82.
\newblock Springer Science \& Business Media, 2013.

\bibitem[BW13]{BW}
Armand Borel and Nolan~R Wallach.
\newblock {\em Continuous cohomology, discrete subgroups, and representations
  of reductive groups}, volume~67.
\newblock American Mathematical Soc., 2013.

\bibitem[BWR15]{bruinier2015kudla}
Jan~Hendrik Bruinier and Martin Westerholt-Raum.
\newblock Kudla’s modularity conjecture and formal {Fourier--Jacobi} series.
\newblock In {\em Forum of Mathematics, Pi}, volume~3. Cambridge University
  Press, 2015.

\bibitem[Cho49]{Chow}
Wei-Liang Chow.
\newblock On compact complex analytic varieties.
\newblock {\em American Journal of Mathematics}, 71(4):893--914, 1949.

\bibitem[FF97]{fulton1997young}
Mr~William Fulton and William Fulton.
\newblock {\em Young tableaux: with applications to representation theory and
  geometry}, volume~35.
\newblock Cambridge University Press, 1997.

\bibitem[FM02]{FM1}
Jens Funke and John Millson.
\newblock Cycles in hyperbolic manifolds of non-compact type and {Fourier}
  coefficients of siegel modular forms.
\newblock {\em manuscripta mathematica}, 107(4):409--449, 2002.

\bibitem[FM06]{FM2}
Jens Funke and John Millson.
\newblock Cycles with local coefficients for orthogonal groups and
  vector-valued siegel modular forms.
\newblock {\em American Journal of Mathematics}, 128(4):899--948, 2006.

\bibitem[FM13]{FM3}
Jens Funke and John Millson.
\newblock Boundary behaviour of special cohomology classes arising from the
  {Weil} representation.
\newblock {\em Journal of the Institute of Mathematics of Jussieu.},
  12(3):571--634, 2013.

\bibitem[FM{\etalchar{+}}14]{FM4}
Jens Funke, John Millson, et~al.
\newblock The geometric theta correspondence for {Hilbert} modular surfaces.
\newblock {\em Duke Mathematical Journal}, 163(1):65--116, 2014.

\bibitem[Gar18]{garcia}
Luis~E Garcia.
\newblock Superconnections, theta series, and period domains.
\newblock {\em Advances in Mathematics}, 329:555--589, 2018.

\bibitem[Hel79]{Helgason}
Sigurdur Helgason.
\newblock {\em Differential geometry, Lie groups, and symmetric spaces}.
\newblock Academic press, 1979.

\bibitem[How79]{Howe1}
R~Howe.
\newblock $\theta$-series and invariant theory.
\newblock {\em Automorphic Forms, Representations and $ L $-Functions:
  Automorphic Forms, Representations and L-functions}, 1(Part 1):275--285,
  1979.

\bibitem[How89]{Howe2}
Roger Howe.
\newblock Remarks on classical invariant theory.
\newblock {\em Transactions of the American Mathematical Society},
  313(2):539--570, 1989.

\bibitem[HP20]{howard2017arithmetic}
Benjamin Howard and Keerthi~Madapusi Pera.
\newblock Arithmetic of {Borcherds} products.
\newblock {\em Ast\'erisque}, 421(4):187--297, 2020.

\bibitem[HZ76]{HZ}
Friedrich Hirzebruch and Don Zagier.
\newblock Intersection numbers of curves on {Hilbert} modular surfaces and
  modular forms of nebentypus.
\newblock {\em Inventiones mathematicae}, 36(1):57--113, 1976.

\bibitem[KM82]{KMCompo}
Stephen~S Kudla and John~J Millson.
\newblock Geodesic cycles and the {Weil} representation {I}; quotients of
  hyperbolic space and siegel modular forms.
\newblock {\em Compositio Mathematica}, 45(2):207--271, 1982.

\bibitem[KM86]{KMI}
Stephen~S Kudla and John~J Millson.
\newblock The theta correspondence and harmonic forms. {I}.
\newblock {\em Mathematische Annalen}, 274(3):353--378, 1986.

\bibitem[KM87]{KMII}
Stephen~S Kudla and John~J Millson.
\newblock The theta correspondence and harmonic forms. {II}.
\newblock {\em Mathematische Annalen}, 277(2):267--314, 1987.

\bibitem[KM88]{KM88}
Stephen~S Kudla and John~J Millson.
\newblock Tubes, cohomology with growth conditions and an application to the
  theta correspondence.
\newblock {\em Canadian Journal of Mathematics}, 40(1):1--37, 1988.

\bibitem[KM90]{KM90}
Stephen~S Kudla and John~J Millson.
\newblock Intersection numbers of cycles on locally symmetric spaces and
  {Fourier} coefficients of holomorphic modular forms in several complex
  variables.
\newblock {\em Publications Math{\'e}matiques de l'IH{\'E}S}, 71:121--172,
  1990.

\bibitem[Kud96]{kudlanotes}
Stephen Kudla.
\newblock Notes on the local theta correspondence.
\newblock {\em unpublished notes, available online}, 1996.

\bibitem[KV78]{KV}
Masaki Kashiwara and Michele Vergne.
\newblock On the {Segal-Shale-Weil} representations and harmonic polynomials.
\newblock {\em Inventiones mathematicae}, 44(1):1--47, 1978.

\bibitem[LM{\etalchar{+}}93]{LiMillson}
Jian-Shu Li, John~J Millson, et~al.
\newblock On the first betti number of a hyperbolic manifold with an arithmetic
  fundamental group.
\newblock {\em Duke mathematical journal}, 71(2):365--401, 1993.

\bibitem[Mok15]{mok2015endoscopic}
Chung~Pang Mok.
\newblock {\em Endoscopic classification of representations of quasi-split
  unitary groups}, volume 235.
\newblock American Mathematical Society, 2015.

\bibitem[MS]{MS}
John~J Millson and Yousheng Shi.
\newblock Theta series and generalized special cycles.
\newblock {\em in preparation}.

\bibitem[Pau98]{Paul}
Annegret Paul.
\newblock Howe correspondence for real unitary groups.
\newblock {\em Journal of functional analysis}, 159(2):384--431, 1998.

\bibitem[Ser56]{Serre}
Jean-Pierre Serre.
\newblock G{\'e}om{\'e}trie alg{\'e}brique et g{\'e}om{\'e}trie analytique.
\newblock In {\em Annales de l'institut Fourier}, volume~6, pages 1--42, 1956.

\bibitem[VZ84]{VoganZuckerman}
David~A Vogan and Gregg~J Zuckerman.
\newblock Unitary representations with non-zero cohomology.
\newblock {\em Compositio Mathematica}, 53(1):51--90, 1984.

\bibitem[Wei64]{Weil}
Andr{\'e} Weil.
\newblock Sur certains groupes d'op{\'e}rateurs unitaires.
\newblock {\em Acta mathematica}, 111(1):143--211, 1964.

\bibitem[Won01]{Wong}
Roderick Wong.
\newblock {\em Asymptotic approximations of integrals}.
\newblock SIAM, 2001.

\bibitem[Zha09]{zhang2009modularity}
Wei Zhang.
\newblock {\em Modularity of generating functions of special cycles on Shimura
  varieties}.
\newblock PhD thesis, Columbia University, 2009.

\end{thebibliography}

\end{document}